\documentclass [11pt,reqno]{amsart}
\usepackage {amsmath,amssymb,verbatim,geometry}
\usepackage[all]{xy}
\newif\ifpdf
\ifpdf
\usepackage[pdftex]{hyperref}
\else
\fi

\numberwithin{equation}{section}       

\newtheorem{prop} {Proposition} [section]
\newtheorem{thm}[prop] {Theorem} 
\newtheorem{defi}[prop] {Definition}
\newtheorem{lem}[prop] {Lemma}
\newtheorem{cor}[prop]{Corollary}
\newtheorem{prop-def}[prop]{Proposition-Definition}

\newtheorem*{thmA}{Theorem A} 
 
\newtheorem*{thmB}{Theorem B} 
\newtheorem*{thmC}{Theorem C}

\theoremstyle{remark}
\newtheorem{ex}[prop]{Example}
\newtheorem{rem}[prop]{Remark}
\newtheorem*{ackn}{Acknowledgements}

\newcommand{\C}{{\mathbb{C}}}
\newcommand{\N}{{\mathbb{N}}}
\newcommand{\PP}{{\mathbb{P}}}
\newcommand{\Q}{{\mathbb{Q}}}
\newcommand{\R}{{\mathbb{R}}}

\newcommand{\fa}{{\mathfrak{a}}}

\newcommand{\cE}{{\mathcal{E}}}

\newcommand{\cH}{{\mathcal{H}}}

\newcommand{\cL}{{\mathcal{L}}}
\newcommand{\cM}{{\mathcal{M}}}

\newcommand{\cO}{{\mathcal{O}}}

\newcommand{\cT}{{\mathcal{T}}}

\newcommand{\ie}{i.e.~}

\newcommand{\fX}{{\mathfrak{X}}}

\newcommand{\cK}{{\mathcal{K}}}

\newcommand{\fm}{{\mathfrak{m}}}

\newcommand{\tf}{\widetilde{\varphi}}
\newcommand{\tmu}{\widetilde{\mu}}
\newcommand{\tom}{\widetilde{\omega}}

\newcommand{\tX}{\widetilde{X}}

\newcommand{\tU}{\widetilde{U}}

\renewcommand{\a}{\alpha}
\renewcommand{\b}{\beta}
\newcommand{\de}{\delta}
\newcommand{\e}{\varepsilon}
\newcommand{\om}{\omega}
\newcommand{\f}{\varphi}
\newcommand{\p}{\psi}
\newcommand{\la}{\lambda}

\newcommand{\D}{\Delta}

\newcommand{\MA}{\mathrm{MA}}
\newcommand{\Amp}{\mathrm{Amp}\,}

\newcommand{\Ric}{\mathrm{Ric}}

\newcommand{\loc}{\mathrm{loc}}

\newcommand{\reg}{\mathrm{reg}}

\newcommand{\pddt}{\frac{\partial}{\partial t}}

\newcommand{\mes}{\mathrm{m}}

\newcommand{\psh}{{\mathrm{PSH}}}
\newcommand{\pshn}{{\mathrm{PSH}_\mathrm{norm}}}
\newcommand{\pshf}{{\mathrm{PSH}_\mathrm{full}}}
\newcommand{\enorm}{{\cE^1_\mathrm{norm}}}
\newcommand{\Tf}{{\cT_\mathrm{full}}}

\newcommand{\din}{\operatorname{Ding}}
\newcommand{\mab}{\operatorname{Mab}}

\newcommand{\dbar}{\overline{\partial}}

\renewcommand{\Re}{\operatorname{Re}}
\renewcommand{\Im}{\operatorname{Im}}

\newcommand{\Ker}{\operatorname{Ker}}

\newcommand{\Aut}{\operatorname{Aut}}

\newcommand{\ord}{\operatorname{ord}}

\newcommand{\tr}{\operatorname{tr}}

\newcommand{\supp}{\operatorname{supp}}

\title[K\"ahler-Einstein metrics and the K\"ahler-Ricci flow]{K\"ahler-Einstein metrics and the K\"ahler-Ricci flow on log Fano varieties}

\date{\today}

\author{R.J.Berman
    \and
    S.Boucksom
    \and
    P.Eyssidieux
    \and
    V.Guedj
    \and
    A. Zeriahi*}

\address{Chalmers Techniska H{\"o}gskola \\
 G{\"o}teborg\\
 Sweden}
\email{robertb@math.chalmers.se}

\address{CNRS-\'Ecole Polytechnique\\
 C.M.L.S.\\
 91128 Palaiseau cedex\\
 France}
\email{sebastien.boucksom@polytechnique.edu}

\address{Institut Universitaire de France \& Institut Fourier, Saint-Martin d'H{\`e}res\\
France}
\email{Philippe.Eyssidieux@ujf-grenoble.fr}

 \address{Institut Universitaire de France \& Institut Math\'ematiques de Toulouse, Toulouse,
France}
\email{vincent.guedj@math.univ-toulouse.fr}

\address{Institut Math\'ematiques de Toulouse, Universit{\'e} Paul Sabatier, Toulouse
France}
\email{ahmed.zeriahi@math.univ-toulouse.fr}

\subjclass[2010]{Primary: 32W20, seconday: 32Q20 \\
*The authors are partially supported by the French A.N.R. project MACK}

\begin{document}

\begin{abstract} We prove the existence and uniqueness of K\"ahler-Einstein metrics on $\Q$-Fano varieties with log terminal singularities (and more generally on log Fano pairs) whose Mabuchi functional is proper. We study analogues of the works of Perelman on the convergence of the normalized K\"ahler-Ricci flow, and of Keller, Rubinstein on its discrete version, Ricci iteration.  In the special case of (non-singular) Fano manifolds, our results on Ricci iteration yield smooth convergence without any additional condition, improving on previous results. Our result for the K\"ahler-Ricci flow provides weak convergence independently of Perelman's celebrated estimates. 
\end{abstract} 

\maketitle

\setcounter{tocdepth}{1}

\tableofcontents

\newpage

\section*{Introduction}

Complex Monge-Amp\`ere equations have been one of the most powerful tools in K\"ahler geometry since T.~Aubin and S.T.~Yau's classical works  \cite{Aub78,Yau}, culminating in Yau's solution to the Calabi conjecture. A notable application is the construction of K\"ahler-Einstein metrics on compact K\"ahler manifolds. Whereas their existence on manifolds with ample and trivial canonical class was settled in \cite{Aub78} and \cite{Yau} respectively, determining necessary and sufficient conditions for a Fano manifold to carry a K\"ahler-Einstein metric is still an open problem that attracts a lot of attention 
(see \cite{PS10})\footnote{This problem has been solved recently by Chen, Donaldson and Sun  \cite{CDS14,CDS3}, see also [Tia15].}. 

In recent years, following the pioneering work of H.~Tsuji~\cite{Tsu}, degenerate complex Monge-Amp\`ere equations have been intensively studied by many authors. In relation to the Minimal Model Program, they led to the construction of singular K\"ahler-Einstein metrics with zero or negative Ricci curvature \cite{EGZ1} or, more generally, of canonical volume forms on compact K\"ahler manifolds with nonnegative Kodaira dimension  \cite{ST08,ST}. 

Making sense of and constructing K\"ahler-Einstein metrics on (possibly singular) Fano varieties turns out to require more advanced tools in the study of degenerate complex Monge-Amp\`ere equations. 
The purpose of this article is to develop these tools, following the first step taken in \cite{BBGZ}, so as to investigate K\"ahler-Einstein metrics on singular Fano varieties, and more generally on log Fano pairs. 

A main motivation to study them comes from the fact that singular K\"ahler-Einstein Fano varieties arise naturally as Gromov-Hausdorff limits of K\"ahler-Einstein Fano manifolds. This had been strongly suggested by \cite{CCT,Tia10,LX}, among other works, and was finally established very recently by S.~Donaldson and S.~Sun in \cite{DS}. 

\smallskip

Let $X$ be a $\Q$-Fano variety, \ie a normal projective complex variety such that $-K_X$ is $\Q$-Cartier and ample, without any further \emph{a priori} restriction on its singularities. Any reasonable notion of a K\"ahler-Einstein metric on $X$ should at least restrict to a K\"ahler metric $\om$ on the regular locus $X_\reg$ with $\Ric(\om)=\om$. 

We show (see Proposition \ref{prop:KEklt}) that the existence of such a metric $\om$ on $X_\reg$ forces $X$ to have \emph{log terminal singularities}, a class of singularities which comprises quotient singularities and is characterized in analytic terms by a finite volume condition. 

We prove that the volume $\int_{X_\reg}\om^n$ is automatically finite, bounded above by $c_1(X)^n$ with $n:=\dim_{\C} X$. Relying on finite energy techniques and regularity of solutions to Monge-Amp\`ere equations, we further show that $\int_{X_\reg}\om^n=c_1(X)^n$ if and only if $\om$ extends to a closed positive $(1,1)$-current $\om$ on $X$ lying in $c_1(X)$ and having continuous local potentials, or equivalently the curvature form of a continuous psh metric on the $\Q$-line bundle $-K_X$. We then say that $\om$ is a K\"ahler-Einstein metric on $X$. 

In \cite[Theorem 1.2]{DS} it is proved that the Gromov-Hausdorff limit of any sequence $(X_j,\om_j)$ of K\"ahler-Einstein Fano manifolds with fixed volume $c_1(X_j)^n=V$ is  a $\Q$-Fano 
variety $X$ with log terminal singularities, equipped with  a K\"ahler-Einstein metric $\om$ in the above sense, with $c_1(X)^n=V$. Combined with \cite{LX}, this strongly suggests that K\"ahler-Einstein $\Q$-Fano varieties can be used to compactify the moduli space of K\"ahler-Einstein Fano manifolds. We thank O.Debarre and B.Totaro for emphasizing this point.

\smallskip

 We now give precise formulations of our main results. To simplify the exposition, we only consider the  easier context of Fano varieties and refer the reader to the sequel for the corresponding statements for log Fano pairs.

\subsection*{Existence and uniqueness of K\"ahler-Einstein metrics}
 
We define a Mabuchi functional $\mab$ (extending the classical Mabuchi $K$-energy) and a $J$-functional $J$ for a given $\Q$-Fano variety $X$ with log terminal singularities, and we say as usual that the Mabuchi functional is \emph{proper} if $\mab\to+\infty$ as $J\to+\infty$. 

Our first main result is as follows.

\begin{thmA} Let $X$ be a $\Q$-Fano variety with log terminal singularities. 
\begin{itemize}
\item[(i)] The identity component $\Aut^0(X)$ of the automorphism group of $X$ acts transitively on the set of K\"ahler-Einstein metrics on $X$,
\item[(ii)] If the Mabuchi functional of $X$ is proper, then $\Aut^0(X)=\{1\}$ and $X$ admits a unique K\"ahler-Einstein metric, which is the unique minimizer of the Mabuchi functional. 
\end{itemize}
\end{thmA}

When $X$ is non-singular, the first point is a classical result of S.~Bando and T.~Mabuchi \cite{BM87}. Our proof in the present context builds on the recent work of B.~Berndtsson \cite{Bern11}. The second point generalizes a result of W.Y.~Ding and G.~Tian (see \cite{Tian}), and relies as in \cite{Ber10} on the variational approach developed in our previous work \cite{BBGZ}. It should be recalled that, when $X$ is non-singular and $\Aut^0(X)=\{1\}$, a deep result of G.~Tian \cite{Tia97}, strengthened in \cite{PSSW}, conversely shows that the existence of a K\"ahler-Einstein metric implies the properness of the Mabuchi functional. 
It would of course be very interesting to establish a similar result  for singular varieties\footnote{This should be a consequence of the present article and the recent work of Darvas-Rubinstein \cite{DR15}.} .

\subsection*{Ricci iteration}
In their independent works \cite{Kel} and \cite{Rub}, J.~Keller and Y.~Rubinstein investigated the dynamical system known as \emph{Ricci iteration}, defined by iterating the inverse Ricci operator. Our second main result deals with the existence and convergence of Ricci iteration in the more general context of $\Q$-Fano varieties.

\begin{thmB} Let $X$ be a $\Q$-Fano variety with log terminal singularities. 
\begin{itemize}
\item[(i)] Given a smooth form $\om_0\in c_1(X)$, there exists a unique sequence of closed positive currents $\om_j\in c_1(X)$ with continuous potentials on $X$, smooth on $X_\reg$, and such that 
$$
\Ric(\om_{j+1})=\om_j
$$
on $X_\reg$ for all $j\in\N$. 

\item[(ii)] If we further assume that the Mabuchi functional of $X$ is proper and let $\om_{\mathrm{KE}}$ be the unique K\"ahler-Einstein metric provided by Theorem A, then $\lim_{j\to+\infty}\om_j=\om_{\mathrm{KE}}$, the convergence being in $C^\infty$-topology on $X_\reg$, and uniform on $X$ at the level of potentials. 
\end{itemize}
\end{thmB}

When $X$ is non-singular, this result settles \cite[Conjecture 3.2]{Rub}, which was obtained in \cite[Theorem 3.3]{Rub} under the more restrictive assumption that Tian's $\alpha$-invariant satisfies $\a(X)>1$ (an assumption that implies the properness of the Mabuchi functional).   Building on a preliminary version of the present paper, a more precise version of Theorem B was obtained in \cite[Theorem 2.5]{JMR} for K\"ahler-Einstein metrics with cone singularities along a smooth hypersurface of a non-singular variety.

\subsection*{Convergence of the K\"ahler-Ricci flow}

When $X$ is a $\Q$-Fano variety with log terminal singularities, the work of J.~Song and G.~Tian \cite{ST} shows that given an initial closed positive current $\om_0\in c_1(X)$ with continuous potentials, there exists a unique solution $(\om_t)_{t>0}$ to the normalized K\"ahler-Ricci flow, in the following sense: 
\begin{itemize} 
\item[(i)] For each $t>0$, $\om_t$ is a closed positive current in $c_1(X)$ with continuous potentials; 
\item[(ii)] On $X_\reg\times]0,+\infty[$, $\om_t$ is smooth and satisfies $\dot\om_t=-\Ric(\om_t)+\om_t$; 
\item[(iii)] $\lim_{t\to 0_+}\om_t=\om_0$, in the sense that their local potentials converge in 
${\mathcal C}^0(X_\reg)$.  
\end{itemize}
Our third main result studies the long time behavior of this normalized K\"ahler-Ricci flow, and provides a weak analogue for singular Fano varieties of G.~Perelman's\footnote{Perelman explained his celebrated estimates during a seminar talk at MIT in 2003 (see \cite{SeT08}). These have been used since then in studying the ${\mathcal C}^{\infty}$-convergence of the normalized K\"ahler-Ricci flow under various assumptions (see notably \cite{TZ,PSS,PSSW08b}).} result on the convergence of the K\"ahler-Ricci flow on K\"ahler-Einstein Fano manifolds:

\begin{thmC} Assume that the Mabuchi functional of $X$ is proper, and denote by $\om_{\mathrm{KE}}$ its unique K\"ahler-Einstein metric. Then $\lim_{t\to+\infty}\om_t=\om_{\mathrm{KE}}$ and $\lim_{t\to+\infty}\om_t^n=\om_\mathrm{KE}^n$, both in the weak topology. 
\end{thmC}

When $X$ is non-singular, the above result is certainly weaker than Perelman's theorem, which yields convergence in $C^\infty$-topology. On the other hand, our approach, which relies on a variational argument using results of \cite{BBGZ}, is completely independent of Perelman's deep estimates - which are at any rate out of reach for the moment on singular varieties.

\subsection*{The strong topology of currents with finite energy}

The classical differential-geometric approach to the above convergence results requires delicate \emph{a priori} estimates to guarantee compactness of a given family of metrics in $C^{\infty}$-topology. Our approach consists in working with the set $\cT^1$ of closed positive $(1,1)$-currents $\om$ in the fixed cohomology class $c_1(X)$ and having \emph{finite energy} in the sense that $J(\om)<+\infty$. As a consequence of \cite{BBGZ}, the map $\om\mapsto V^{-1}\om^n$ (which corresponds to the complex Monge-Amp\`ere operator at the level of potentials) sets up a bijection between $\cT^1$ and the set $\cM^1$ of probability measures with finite energy. 

With respect to the weak topology of currents, compactness in $\cT^1$ is easily obtained: any set of currents with uniformly bounded energy is weakly compact. But the drawback of this weak topology is that the Monge-Amp\`ere operator is {\it not} weakly continuous as soon as $n\ge 2$. 

An important novelty of this article is to define, study and systematically use a strong topology on $\cT^1$ and $\cM^1$, which turns them into complete metric spaces and with respect to which the bijection $\cT^1\simeq\cM^1$ described above becomes a homeomorphism. At the level of potentials, the strong topology appears as a higher dimensional and non-linear version of the Sobolev $W^{1,2}$-norm\footnote{This topology has been recently studied further in \cite{Dar14}.}. 

Relying on a property of Lelong numbers of psh functions proved in Appendix A, we prove that $\om$-psh functions with finite energy have identically zero Lelong numbers on any resolution of singularities of $X$, and hence satisfy an exponential integrability condition. This is used to prove a compactness result in the strong topology of $\cM^1$ for probability measures with uniformly bounded entropy, which is a key point to our approach. Under the assumption that the Mabuchi functional is proper, the entropy bound is easily obtained along the normalized K\"ahler-Ricci flow and Ricci iteration, thanks to the monotonicity property of the Mabuchi functional.

\subsection*{Structure of the article}
The article is organized as follows:
\begin{itemize}
\item Section \ref{sec:fec} is a recap on finite energy currents. It provides in particular a crucial integrability property of their potentials (finite energy functions).
\item Section \ref{sec:strong}  introduces the strong topology and establishes strong compactness of measures with uniformly bounded entropy.  
\item Section \ref{sec:lfp}  studies the first basic properties of K\"ahler-Einstein metrics on log Fano pairs, showing in particular that the singularities are at most log terminal;
\item Section \ref{sec:variational} gives a variational characterization of K\"ahler-Einstein metrics, extending some results from \cite{BBGZ};
\item Section \ref{sec:unique} provides an extension of the uniqueness results of Bando-Mabuchi and Berndtsson to the context of log Fano pairs, finishing the proof of Theorem A;
\item Section \ref{sec:iteration} studies 'Ricci iteration' in the context of log Fano pairs, 
and contains the proof of Theorem B. 
\item Section \ref{sec:flow} recalls Song and Tian's construction of the normalized K\"ahler-Ricci flow on a $\Q$-Fano variety, and proves Theorem C;
\item Section \ref{sec:example}  adapts a construction of \cite{AGP06} to get examples of log Fano pairs and log terminal Fano varieties with K\"ahler-Einstein metrics that are not of orbifold type;
\item The article ends with three appendices. The first one proves an Izumi-type estimate that plays a crucial role for the integrability properties of quasi-psh functions with finite energy. The second one provides an explicit version of Paun's Laplacian estimate \cite{Pau}, on which relies the $C^\infty$-convergence in Theorem B.
The third one gives a detailed proof of a version of Berndtsson's subharmonicity theorem that plays a crucial role in the proof of the uniqueness of
K\"ahler-Einstein metrics.

\end{itemize}

\text{ }

\noindent {\it Nota Bene.} 
The current version of this article differs substantially from the first version of the arXiv preprint. The latter dealt more generally with Mean Field Equations with reference measures having H\"older continuous potentials, a point of view that made it less readily accessible to differential geometers. Meanwhile, the preprint \cite{DS} appeared, giving further motivation for the general context of our previous work. Another new feature of the present version is the Izumi-type result proved in Appendix A. 

Various important works have appeared since the the first version of our work was circulating. We have only mentioned in footnotes those that are 
immediately connected to the contents of the present article.

\begin{ackn} 
The authors would like to express their gratitude to Bo Berndtsson for uncountably many interesting discussions related to this work, and in particular for his help regarding the uniqueness theorem.  We are also grateful to Tomoyuki Hisamoto and Dror Varolin for helpful discussions, and we thank the referee for useful suggestions.
\end{ackn}

\section{Finite energy currents} \label{sec:fec}

The goal of this section is to establish a number of preliminary facts about functions and measures with finite energy on a normal compact K\"ahler space, which rely on a combination of the main results from \cite{EGZ1,BEGZ,BBGZ,EGZ2}. 

\subsection{Plurisubharmonic functions and positive $(1,1)$-currents}\label{sec:psh}
Let $X$ be a normal, connected, compact complex space, and denote by $n$ its (complex) dimension. 

By definition, a \emph{K\"ahler form} $\om_0$ on $X$ is locally the restriction to $X$ of an ambient K\"ahler form in a polydisc where $X$ is locally realized as a closed analytic subset. In particular, it admits local potentials, which are smooth strictly psh functions. 

A function $\f:X\to[-\infty,+\infty[$ is \emph{$\om_0$-plurisubharmonic} ($\om_0$-psh for short) if  $\f+u$ is psh for each local potential $u$ of $\om_0$, which means that $\f+u$ is the restriction to $X$ of an ambient psh function in a polydisc as above (see \cite{FN,DemSMF} for further information on psh functions in this context).

We denote by $\psh(X,\om_0)$ the set of $\om_0$-psh functions on $X$, endowed with its natural weak topology. By Hartogs' lemma, $\f\mapsto\sup_X\f$ is continuous on $\psh(X,\om_0)$ (with respect to the weak topology), and we say that $\f\in\psh(X,\om_0)$ is \emph{normalized} if $\sup_X\f=0$. We denote by
$$
\pshn(X,\om_0)\subset\psh(X,\om_0)
$$
the set of normalized $\om_0$-psh functions, which is compact for the weak topology. We also denote by $\cT(X,\om_0)$ the set of all closed positive $(1,1)$-currents $\om$ $dd^c$-cohomologous to $\om_0$, endowed with the weak topology. The map 
$$
\f\mapsto\om_\f:=\om_0+dd^c\f
$$ 
defines a homeomorphism
$$
\pshn(X,\om_0)\simeq\cT(X,\om_0).
$$
with respect to the weak topologies. We denote its inverse by $\om\mapsto\f_\om$, so that $\f_\om$ is the unique function in $\pshn(X,\om_0)$ such that 
$$
\om=\om_0+dd^c\f_\om.
$$
When $X$ is non-singular, Demailly's regularization theorem \cite{Dem92} (see also \cite{BK07}) shows that any $\f\in\psh(X,\om_0)$ is the decreasing limit of a sequence of smooth $\om_0$-psh functions. The analogous statement is not known in the general singular case (except when $\omega_0 \in c_1(L)$ represents the first Chern class of an ample line bundle $L$ \cite{CGZ}). However, it follows from \cite{EGZ2} that every $\f\in\psh(X,\om_0)$ is the decreasing limit of a sequence of \emph{continuous} $\om_0$-psh functions. 

If we let $\pi:\tX\to X$ be a resolution of singularities, then $\tom_0:=\pi^*\om_0$ is a semipositive $(1,1)$-form which is big in the sense that $\int_{\tX}\tom_0^n>0$. Since $X$ is normal, $\pi$ has connected fibers, hence every $\tom_0$-psh function on $\tX$ is of the form $\f\circ\pi$ for a unique $\om_0$-psh function $\f$ on $X$. We thus have a homeomorphism
$$
\psh(X,\om_0)\simeq\psh(\tX,\tom_0). 
$$
This reduces the study of $\om_0$-psh functions on $X$ to $\tom_0$-psh functions on $\tX$, where $\tX$ is thus a compact K\"ahler manifold and $\tom_0$ is a semipositive and big $(1,1)$-form.

\subsection{Functions with full Monge-Amp\`ere mass}
Let $X$ be a normal compact complex space endowed with a fixed K\"ahler form $\om_0$. For each $\f\in\psh(X,\om_0)$, the functions $\f_j:=\max\{\f,-j\}$ are $\om_0$-psh and bounded for all $j\in\N$. The Monge-Amp\`ere measures 
$(\om_0+dd^c\f_j)^n$ are therefore well-defined in the sense of Bedford-Taylor, with 
$$
\int_X(\om_0+dd^c\f_j)^n=\int_X\om_0^n=:V.
$$ 
By \cite{BT87}, the positive measures $\mu_j:={\bf 1}_{\{\f>-j\}}(\om_0+dd^c\f_j)^n$ satisfy
$$
{\bf 1}_{\{\f>-j\}}\mu_{j+1}=\mu_j,
$$
and in particular $\mu_j\le\mu_{j+1}$. As in \cite{BEGZ}, we say that $\f$ has \emph{full Monge-Amp\`ere mass} (this is called finite energy in \cite{GZ}) if $\lim_{j\to\infty}\mu_j(X)=1$, \ie
$$
\lim_{j\to\infty}\int_{\left\{\f\le -j\right\}}\left(\om_0+dd^c\max\{\f,-j\}\right)^n=0.
$$
In that case we set $(\om_0+dd^c\f)^n:=\lim_{j\to+\infty}\mu_j$, which is thus a positive measure on $X$ with mass $V$. More generally, for any $\f_1,...,\f_n\in\pshf(X,\om_0)$ the positive measure 
$$
(\om_0+dd^c\f_1)\wedge...\wedge(\om_0+dd^c\f_n)
$$ 
is also well-defined, and depends continuously on the $\f_j$'s when the latter converge monotonically. 

We denote by 
$$
\pshf(X,\om_0)\subset\psh(X,\om_0)
$$ 
the set of $\om_0$-functions with full Monge-Amp\`ere mass. We say that a current $\om\in\cT(X,\om_0)$ has full Monge-Amp\`ere mass if so is $\f_\om$, and we write
$$
\Tf(X,\om_0)\subset\cT(X,\om_0)
$$
for the set of currents with full Monge-Amp\`ere mass. For each $\f\in\pshf(X,\om_0)$ we define a probability measure  
$$
\MA(\f):=V^{-1}\om_\f^n. 
$$ 

We denote by $\cM(X)$ the set of probability measures on $X$, endowed with the weak topology, and we call 
$$
\MA:\pshf(X,\om_0)\to\cM(X)
$$ 
the \emph{Monge-Amp\`ere operator}. We emphasize that for $n\ge 2$ this operator is \emph{not} continuous in the weak topology of $\pshf(X,\om_0)$. 

For each $\f\in\pshf(X,\om_0)$, the measure $\MA(\f)$ is \emph{non-pluripolar}, \ie it puts no mass on pluripolar sets. Conversely, applying \cite[Corollary 4.9]{BBGZ} to a resolution of singularities shows that any non-pluripolar probability measure $\mu$ is of the form $\MA(\f)$ for some $\f\in\pshf(X,\om_0)$. In other words, the map $\Tf(X,\om)\to\cM(X)$ defined by $\om\mapsto V^{-1}\om^n$ is injective, and its image is exactly the set of non-pluripolar measures in $\cM(X)$.  

\smallskip

The following crucial integrability property of functions with full Monge-Amp\`ere mass relies on a non-trivial property of Lelong numbers of psh functions proved in Appendix A. 

\begin{thm}\label{thm:lelres} 
Let $\f\in\pshf(X,\om_0)$ and let $\pi:\tX\to X$ be any resolution of singularities of $X$. Then $\tf:=\f\circ\pi$ has zero Lelong numbers. Equivalently, $e^{-\tf}\in L^p(\tX)$ for all $p<+\infty$. 
\end{thm}

\begin{proof} 
By Corollary \ref{cor:izumi} from Appendix A, we are to show that the slope $s(\f,x)$ of $\f$ at any point $x\in X$ is zero. Let $(f_i)$ be local generators of the maximal ideal of $\cO_{X,x}$ and set $\p:=\log\sum_i|f_i|$, which is a psh function defined on a neighborhood $U$ of $x$, with an isolated logarithmic singularity at $x$. Let $0\le \theta\le 1$ be a smooth function with compact support in $U$ such that $\theta\equiv 1$ on a neighborhood of $x$. A standard computation (see e.g. \cite[Lemma 3.5]{Dem92}) shows that 
$$
\rho:=\log\left(\theta e^\p+(1-\theta)\right)
$$
satisfies $dd^c\rho\ge-C\om_0$ for some $C>0$. It follows that $\e\rho$ is $\om_0$-psh for all $\e>0$ small enough. If $s(\f,x)>0$ then $\f\le\e\rho+O(1)$ for any $0<\e<s(\f,x)$. Since $\f$ has full Monge-Amp\`ere mass, it follows from \cite[Proposition 2.14]{BEGZ} that $\e\rho$ also has full Monge-Amp\`ere mass, which is impossible since $(\om_0+dd^c\e\rho)^n\left(\{x\}\right)=\e^n m(X,x)>0$ by \cite{DemSMF}. 

Finally, the last equivalence is a classical result of Skoda \cite{Sko}. 
\end{proof}

\subsection{$\a$-invariants and tame measures}\label{sec:tame}

The following  uniform integrability exponent generalizes the classical one of \cite{Tia87,TY87}:

\begin{defi}\label{defi:alpha} 
The \emph{$\a$-invariant} of a measure $\mu\in\cM(X)$ (with respect to $\om_0$) is defined as
$$
\a_{\om_0}(\mu):=\sup\left\{\a\ge 0\mid\sup_{\f\in\pshn(X,\om_0)}\int_X e^{-\a\f}d\mu<+\infty\right\}. 
$$
\end{defi}
Note that $\a_{\om_0}(\mu)>0$ implies that $\mu$ is non-pluripolar. We also introduce the following \emph{ad hoc} terminology. 

\begin{defi}\label{defi:tame} We say that a positive measure $\mu$ on $X$ is \emph{tame} if $\mu$ puts no mass on closed analytic sets and if there exists a resolution of singularities $\pi:\tX\to X$ such that the lift $\tmu$ of $\mu$ to $\tX$ has $L^p$-density (with respect to Lebesgue measure) for some $p>1$. 
\end{defi}
By the \emph{lift} of $\mu$, we mean the push-forward by $\pi^{-1}$ of its restriction to the Zariski open set over which $\pi$ is an isomorphism.  Note that this is well-defined precisely because $\mu$ puts no mass on closed analytic sets. As we shall see, this \emph{a priori} artificial looking notion of a tame measure appears naturally in the context of log terminal singularities.  

\begin{prop}\label{prop:hold} 
Let $\mu$ be a tame measure on $X$. Then $\a_{\om_0}(\mu)>0$, and for each $\f\in\pshf(X,\om_0)$ we have $e^{-\f}\in L^p(\mu)$ for all finite $p$. In particular, $e^{-\f}\mu$ is also tame. Furthermore given $p<+\infty$ and a weakly compact subset $\cK\subset\pshf(X,\om_0)$, both the identity map and the map $\f\mapsto e^{-\f}$ define continuous maps $\cK\to L^p(\mu)$. 
\end{prop}

\begin{proof}
Let $\pi:\tX\to X$ be a resolution of singularities such that the lift $\tmu$ of $\mu$ to $\tX$ has $L^p$-density for some $p>1$. By \cite{Sko} and H\"older's inequality, it follows that there exists $\e>0$ such that $\int_{\tX} e^{-\e\tf} d\tmu$ is finite and uniformly bounded for all normalized $\tom_0$-psh functions $\tf$ on $\tX$. 
Since $\mu$ puts no mass on closed analytic subsets, we infer
$
\int_X e^{-\e \f} d \mu=\int_{\tX} e^{-\e\tf} d\tmu<+\infty,
$
hence $\a_{\om_0}(\mu)>0$. 

If we take $\f\in\pshf(X,\om_0)$ and set $\tf:=\f\circ\pi$, Theorem \ref{thm:lelres} shows that $e^{-\tf}$ belongs to $L^q$ for all $q<+\infty$. By \cite{Zer01}, we even have a uniform $L^q$-bound as long as $\tf$ stays in a weakly compact set of $\tom_0$-psh functions. Using the elementary inequality
$$
\left| e^a-e^b \right| \leq |a-b| e^{a+b}, \; \;
a,b >0,
$$
the continuity of $\f \mapsto e^{-\f}$ now follows from H\"older's inequality.
\end{proof}

\begin{lem}\label{lem:tame} Let $\mu\in\cM(X)$ be a tame probability measure. Then there exists a unique $\om\in\Tf(X,\om_0)$ such that $V^{-1}\om^n=\mu$, and $\om$ has continuous potentials. 
\end{lem}
\begin{proof} We have already observed that $\mu$ is non-pluripolar, which implies that $\mu=V^{-1}\om^n$ for  unique $\om\in\Tf(X,\om_0)$. To see that $\om$ has continuous potentials, let $\pi:\tX\to X$ be a resolution of singularities such that $\tmu$ has $L^p$-density for some $p>1$. By \cite{EGZ1,EGZ2}, there exists a continuous $\tom_0$-psh function $\tf$ on $\tX$ such that $V^{-1}(\tom_0+dd^c\tf)^n=\tmu$. If we let $\f$ be the corresponding $\om_0$-psh function on $X$, then $\f$ is continuous on $X$ by properness of $\pi$, and the result follows since $\om:=\om_0+dd^c\f$ by uniqueness.\end{proof}

\subsection{Functions and currents of finite energy} \label{sec:e1}
We introduce
$$
\cE^1(X,\om_0):=\left\{\f\in\pshf(X,\om_0),\,\f\in L^1\left(\MA(\f)\right)\right\},
$$
and say that functions $\f\in\cE^1(X,\om_0)$ have \emph{finite energy}. We denote by 
$$
\cT^1(X,\om_0)\subset\Tf(X,\om_0)
$$ 
the corresponding set of \emph{currents with finite energy}, \ie currents of the form $\om_\f$ with $\f\in\cE^1(X,\om_0)$. It is important to note that $\cT^1(X,\om_0)$ is \emph{not} a closed subset of $\cT(X,\om_0)$. 

\smallskip

The functional $E:\cE^1(X,\om_0)\to\R$ defined by 
\begin{equation}\label{equ:energy}
E(\f)=\frac{1}{n+1}\sum_{j=0}^nV^{-1}\int_X\f\,\om_\f^j\wedge\om_0^{n-j}
\end{equation}
is a primitive of the Monge-Amp\`ere operator, in the sense that
$$
\frac{d}{dt}_{t=0}E\left(t\f+(1-t)\p\right)=\int(\f-\p)\MA(\p)
$$
for any two $\f,\p\in\cE^1(X,\om_0)$. This implies that
\begin{equation}\label{equ:energybis}
E(\f)-E(\p)=\frac{1}{n+1}\sum_{j=0}^nV^{-1}\int_X(\f-\p)\,(\om_0+dd^c\f)^j\wedge(\om_0+dd^c\p)^{n-j}
\end{equation}
for all $\f,\p\in\cE^1(X,\om_0)$. 
The energy functional $E$ satisfies $E(\f+c)=E(\f)+c$ for $c\in\R$, and it is concave and non-decreasing on $\cE^1(X,\om_0)$ (we refer the reader to \cite{BEGZ} for the proofs of these results).

If we extend it to $\psh(X,\om_0)$ by setting $E(\f)=-\infty$ for $\f\in\psh(X,\om_0)\setminus\cE^1(X,\om_0)$ then $E:\psh(X,\om_0)\to[-\infty,+\infty[$ so defined is upper semi-continuous. As a consequence, the convex set
\begin{equation}\label{equ:ec}
\cE^1_C(X,\om_0):=\left\{\f\in\cE^1(X,\om_0),\,\sup_X\f\le C\text{ and }E(\f)\ge -C\right\}
\end{equation}
is compact (for the $L^1$-topology) for each $C>0$. 

We also recall the definition of two related functionals that were originally introduced by Aubin. The \emph{$J$-functional} (based at a given $\psi\in\cE^1(X,\om_0)$) is the functional on $\cE^1(X,\om_0)$ defined by setting 
$$
J_{\p}(\f):=E(\p)-E(\f)+\int_X(\f-\p)\MA(\p). 
$$
Note that $J_{\psi}(\f)=E(\p)-E(\f)+E'(\p)\cdot(\f-\p)\ge 0$ by concavity of $E$. 
For $\p=0$ we simply write 
$$
J(\f):=J_{0}(\f)=V^{-1}\int_X\f\,\om_0^n-E(\f).
$$
Finally the \emph{$I$-functional} is the symmetric functional defined by 
\begin{equation}\label{equ:I}
I(\f,\p)=\int_X(\f-\p)\left(\MA(\p)-\MA(\f)\right)=\sum_{j=0}^{n-1}V^{-1}\int_X d(\f-\p)\wedge d^c(\f-\p)\wedge\om_\f^j\wedge\om_\p^{n-1-j}
\end{equation}
which is also non-negative by concavity of $E$. When $\p=0$ we simply write 
$$
I(\f):=I(\f,0)=V^{-1}\int_X\f\,\om_0^n-\int_X\f\MA(\f).
$$ 
These functionals compare as follows (see for instance \cite[Lemma 2.2]{BBGZ}):
\begin{equation}\label{equ:IJ}
n^{-1} J_{\f}(\p)\le J_{\p}(\f)\le I(\f,\p)\le (n+1)J_{\p}(\f).
\end{equation}

We will also use:
\begin{lem}\label{lem:encompare}
For each $\psi\in\cE^1(X,\om_0)$, there exists $A,B>0$ such that
$$
A^{-1}J(\f)-B\le J_\psi(\f)\le A J(\f)+B
$$
for all $\f\in\cE^1(X,\om_0)$.
\end{lem}
\begin{proof} By translation invariance, we may restrict to $\f\in\enorm(X,\om_0)$. We compute
$$
J_\psi(\f)-J(\f)=E(\psi)+\int_X(\f-\p)\MA(\p)-\int_X\f\MA(0).
$$
Since $\f$ is normalized, $\int_X\f\MA(0)$ is uniformly bounded by \cite[Lemma 3.2]{BBGZ}. On the other hand, by \cite[Proposition 3.4]{BBGZ}, there exists $C>0$ such that
$$
\left|\int_X\f\MA(\psi)\right|\le C J(\f)^{1/2}
$$
for all $\f\in\enorm(X,\om_0)$. We thus get $J_\psi=J+O(J^{1/2})$ on $\enorm(X,\om_0)$, and the result follows.
\end{proof}

As opposed to $E$, these functionals are translation invariant, hence descend to $\cT^1(X,\om_0)$. For $\om,\om'\in\cT^1(X,\om)$ we set $J_\om(\om')=J_{\f_\om}(\f_{\om'})$ and $I(\om,\om')=I(\f_{\om},\f_{\om'})$. For each $C>0$ we set
$$
\cT^1_C(X,\om_0):=\left\{\om\in\cT^1(X,\om_0)\mid J(\om)\le C\right\}.
$$

\subsection{Measures of finite energy}

As in \cite{BBGZ}, we define the \emph{energy} of a probability measure $\mu$ on $X$ (with respect to $\om_0$) as\begin{equation}\label{equ:estar}
E^*(\mu):=\sup_{\f\in\cE^1(X,\om_0)}\left(E(\f)-\int\f\,\mu\right)\in[0,+\infty]. 
\end{equation}
This defines a convex lsc function $E^*:\cM(X)\to[0,+\infty]$, and a probability measure $\mu$ is said to have \emph{finite energy} if $E^*(\mu)<+\infty$. We denote the set of probability measures with finite energy by
$$
\cM^1(X,\om_0):=\left\{\mu\in\cM(X)\mid E^*(\mu)<+\infty\right\},
$$ 
and we set for each $C>0$
\begin{equation}\label{equ:level}
\cM^1_C(X,\om_0):=\left\{\mu\in\cM(X)\mid E^*(\mu)\le C\right\}.
\end{equation}
By \cite[Theorem 4.7]{BBGZ}, the map $\om\mapsto V^{-1}\om^n$ is a bijection between $\cT^1(X,\om_0)$ and $\cM^1(X,\om_0)$ (but here again it is not continuous with respect to weak convergence). 

The concavity of $E$ shows that
\begin{equation}\label{equ:eij}
E^*\left(\MA(\f)\right)=E(\f)-\int_X\f\,\MA(\f)=J_{\f}(0)=(I-J)(\f), 
\end{equation}
and consequently the following Legendre duality relation holds:
\begin{equation}\label{equ:dual}
E(\f)=\inf_{\mu\in\cM(X)}\left(E^*(\mu)+\int_X\f\,d\mu\right).
\end{equation}
We also note that 
\begin{equation}\label{equ:compen}
n^{-1} E^*(\MA(\f))\le J(\f)\le n E^*(\MA(\f)), 
\end{equation}
by (\ref{equ:IJ}), hence $\om\mapsto V^{-1}\om^n$ maps $\cT^1_C(X,\om_0)$ into $\cM^1_{Cn}(X,\om_0)$, and similarly for its inverse. 

\smallskip

We end this section with the following continuity properties:

\begin{prop}\label{prop:conte1} 
Let $\mu\in\cM^1(X,\om_0)$ be a measure with finite energy. Then $\mu$ acts continuously on $\cE^1_C(X,\om_0)$ for each $C>0$. 
Dually, every $\f\in\cE^1(X,\om_0)$ acts continuously on $\cM^1_C(X,\om_0)$. 
\end{prop}
\begin{proof} The first assertion is \cite[Theorem 3.11]{BBGZ}. Let us prove the dual assertion. By \cite{EGZ2}, we may choose a decreasing sequence of continuous $\om_0$-psh functions $\f_k$ converging pointwise to $\f$ on $X$. This monotone convergence guarantees that $I(\f_k,\f)\to 0$. By \cite[Lemma 5.8]{BBGZ}, it follows that the map $\mu\mapsto\int_X\f\mu$ is the uniform limit on $\cM^1_C(X,\om_0)$ of the maps $\mu\mapsto\int_X\f_k\mu$, each of which is continuous by continuity of $\f_k$. 
\end{proof}

\subsection{A quasi-triangle inequality for $I$}
When $X$ has dimension $1$, 
$$
I(\f,\p)=\int_X d(\f-\p)\wedge d^c(\f-\p)
$$ 
coincides with the squared $L^2$-norm of $d(\f-\p)$ for all $\f,\p\in\cE^1(X,\om_0)$. By Cauchy-Schwarz, $I^{1/2}$ satisfies the triangle inequality, and the convexity inequality $(x+y)^2\le 2(x^2+y^2)$ for $x,y\in\R$ yields 
$$
\tfrac 1 2 I(\f_1,\f_2)\le I(\f_1,\f_3)+I(\f_3,\f_2)
$$ 
for any $\f_1,\f_2,\f_3\in\cE^1(X,\om_0)$. 

Our goal here is to establish the following higher dimensional version of this inequality. 
 \begin{thm} \label{thm:qt}
There exists a constant $c_n>0$, only depending on the dimension $n$, such that 
$$
c_n  I(\f_1,\f_2)  \leq I(\f_1,\f_3)  +I(\f_3,\f_1). 
$$
for all $\f_1, \f_2, \f_3 \in \mathcal  E^1 (X,\om_0)$.
\end{thm}

For any $\f_1,\f_2,\p \in \mathcal E^1 (X,\om)$ we set
\begin{equation}\label{equ:l2norm}
 \Vert d (\f_1 - \f_2)\Vert_{\p} :=  \left(\int_X d (\f_1 - \f_2) \wedge d^c (\f_1 -\f_2) \wedge (\om + dd^c\p)^{n-1}\right)^{1/2}, 
\end{equation}
which is the $L^2$-norm of $d(\f_1-\f_2)$ with respect to $\om_\p$ when the latter is a K\"ahler form. 
Using (\ref{equ:I}) it is easy to see that
\begin{equation}\label{equ:compareI}
    \Vert d (\f_1 - \f_2)\Vert_{\frac{\f_1+\f_2}{2}}^2 \leq I(\f_1,\f_2) \leq 2^{n -1}  \Vert d (\f_1 - \f_2)\Vert_{\frac{\f_1+\f_2}{2}}^2.
\end{equation}

\begin{lem} \label{lem:BBGZ}
There exists $c_n > 0$ only depending on $n$ such that for all $\f_1, \f_2, \p \in \cE(X,\om_0)$, 
$$
c_n\Vert d (\f_1 - \f_2)\Vert_{\p}^2
   \leq I(\f_1,\f_2)^{1/2^{n-1}} \left(I (\f_1,\p)^{1-1/2^{n-1}} + I (\f_2,\p)^{1-1/2^{n-1}}\right).
$$
\end{lem}

\begin{proof}   
The proof is a refinement of \cite[Lemma 3.12]{BBGZ} and \cite[Lemma 3.1]{GZ11}. Set $u:=\f_1-\f_2$, $v:=(\f_1+\f_2)/2$ and for each $p=0,...,n-1$,
$$
b_p:=\int_X du\wedge d^c u\wedge\om_{\psi}^{p}\wedge\om_v^{n-p-1}. 
$$
By (\ref{equ:compareI}) we have $b_0\le I(\f_1,\f_2)$, while 
$$
b_{n-1}= \Vert d (\f_1 - \f_2)\Vert_{\p}^2
$$ 
is the quantity we are trying to bound. Let us first check that
\begin{equation}\label{equ:bp}
b_{p+1} \le b_p+ 4\sqrt{b_p I(\p,v)}
\end{equation}
for $p=0,...,n-2$. Using integration by parts and the identity $dd^c u = \om_{\f_1}- \om_{\f_2}$, we compute
 \begin{eqnarray*}
b_{p+1}-b_p & = & \int_X du\wedge d^c u\wedge dd^c(\psi-v)\wedge\om_\p^p\wedge\om_v^{n-p-2} \\
& = & -\int_X du\wedge d^c(\p-v)\wedge dd^c u\wedge\om_\p^p\wedge\om_{v}^{n-p-2} \\
& = & -\int_X du\wedge d^c(\p-v)\wedge\om_{\f_1}\wedge\om_\p^p\wedge\om_{v}^{n-p-2}+ \int_X du\wedge d^c(\psi-v)\wedge\om_{\f_2}\wedge\om_\p^p\wedge\om_{v}^{n-p-2}.
\end{eqnarray*}

For $i=1,2$ the Cauchy-Schwarz inequality yields
$$
\left|\int_X du\wedge d^c(\psi-v)\wedge\om_{\f_i}\wedge\om_\p^p
\wedge\om_{v}^{n-p-2}\right|\le\left(\int_X du\wedge d^c u\wedge\wedge\om_{\f_i}\wedge\om_\p^p
\wedge\om_{v}^{n-p-2}\right)^{1/2}\times
$$
$$
\times\left(\int_X d(\p-v)\wedge d^c(\psi-v)\wedge\om_{\f_i}\wedge\om_\p^p
\wedge\om_{v}^{n-p-2}\right)^{1/2}\le 2 b_p^{1/2} I(\psi,v)^{1/2},
$$
noting that $\om_{\f_i}\le 2\om_v$ and using (\ref{equ:I}). This proves (\ref{equ:bp}). 

Next we use the convexity of $\f \mapsto J_\p(\f)$ and (\ref{equ:IJ}) to get
$$
2I(v,\p) \leq (n +1)\left(I (\f_1,\p)  +  I (\f_2,\p)\right).
$$ 
Setting
\begin{equation}\label{equ:H_n}
H_n:=8(n+1)\left(I (\f_1,\p)  +  I (\f_2,\p)\right)
\end{equation}
and 
\begin{equation}\label{equ:h}
h(t):=t+\sqrt{H_n t}
\end{equation}
we thus have $b_{p+1}\le h(b_p)$ for $p=0,...,n-1$ by (\ref{equ:bp}). Since $b_0 \leq I(\f_1,\f_2)$ and $h$ is non-decreasing, it follows that
\begin{equation}\label{equ:boundh}
\Vert d (\f_1 - \f_2)\Vert_{\p}^2\leq h^{n-1}\left(I(\f_1,\f_2)\right),
\end{equation}
where $h^{n-1}:=h \circ \cdots \circ h$ denotes the $(n-1)$-st iterate of $h:\R_+\to\R_+$.

It is easy to check by induction on $p\in\N$  that 
\begin{equation}\label{equ:hp}
h^{p} (t) \leq 4 H_n^{1-1/2^{p}}t^{1/2^{p}}\text{  for  }0\le t \leq 2^{-2^{p+1}} H_n
\end{equation}
There are two cases:
\begin{itemize}
\item  If $I(\f_1,\f_2) \leq 2^{-2^n} H_n$, we can apply (\ref{equ:hp}) with $p=n-1$, which combines with (\ref{equ:boundh}) to yield
$$
\Vert d (\f_1 - \f_2)\Vert_{\p}^2\le 4 H_n^{1-1/2^{n-1}} I (\f_1,\f_2)^{1/2^{n-1}}. 
$$
We conclude by definition of $H_n$ and the subadditivity of $t\mapsto t^{1/2^{n-1}}$.  

\item Assume now that $I(\f_1,\f_2) \ge 2^{-2^n} H_n$. Using (\ref{equ:I}) we have
\begin{eqnarray*}
 \Vert d (\f_1-\f_2)\Vert_\p & \leq & \Vert d (\f_1-\p)\Vert_\p + \Vert d (\p-\f_2)\Vert_\p,\\
 & \leq & I(\f_1,\p)^{1/2} + I(\f_2,\p)^{1/2},
 \end{eqnarray*}
hence
$$
 \Vert d (\f_1-\f_2)\Vert_\p^2\le 2(I(\f_1,\p)+I(\f_2,\p))=\frac{1}{4(n+1)}H_n. 
 $$
$$
=\frac{1}{4(n+1)}H_n^{1-1/2^{n-1}}H_n^{1/2^{n-1}}\le\frac{1}{n+1}H_n^{1-1/2^{n-1}}I(\f_1,\f_2)^{1/2^{n-1}},
$$
and we conclude as before. 
\end{itemize}

\end{proof}

\begin{proof}[Proof of Theorem \ref{thm:qt}]
Set $\p := \frac{\f_1+\f_2}{2}$. The triangle inequality 
$$
\Vert d (\f_1-\f_2)\Vert_\p\le \Vert d (\f_1-\f_3)\Vert_\p + \Vert d (\f_2-\f_3)\Vert_\p
$$
and (\ref{equ:compareI}) yield
$$
I(\f_1,\f_2)\le 2^n\left(\Vert d(\f_1-\f_3)\Vert_{\p}^2+ \Vert d(\f_2-\f_3)\Vert_{\p}^2\right)
$$
Applying Lemma \ref{lem:BBGZ} we obtain
\begin{eqnarray*}
c_n I(\f_1,\f_2) & \leq & I(\f_1,\f_3)^{1/2^{n-1}} \left(I (\f_1,\p)^{1-1/2^{n-1}} + I (\f_3,\p)^{1-1/2^{n-1}}\right)  \\
   &+&  I(\f_2,\f_3)^{1/2^{n-1}} \left(I (\f_2,\p)^{1-1/2^{n-1}} + I (\f_3,\p)^{1-1/2^{n-1}}\right).
 \end{eqnarray*}
As before the convexity of $J_\p$ and (\ref{equ:IJ}) imply
$$
I(\f_1,\p)\le\frac{n+1}{2}I(\f_1,\f_2),\,\,I(\f_2,\p)\le\frac{n+1}{2}I(\f_1,\f_2)
$$
and
$$
I(\f_3,\p)\le\frac{n+1}{2}\left(I(\f_1,\f_3)+I(\f_2,\f_3)\right).
$$
Plugging this in the above inequality, we obtain after changing $c_n$, 
$$
c_n I(\f_1,\f_2)\le\left(I(\f_1,\f_3)^{1/2^{n-1}}+I(\f_2,\f_3)^{1/2^{n-1}}\right)\times
$$
$$
\times\left(I(\f_1,\f_2)^{1-1/2^{n-1}}+I (\f_1,\f_3)^{1-1/2^{n-1}} + I (\f_2,\f_3)^{1-1/2^{n-1}}\right).
$$
Note that we can assume $I(\f_1,\f_2) \geq \max \{ I (\f_1,\f_3), I(\f_2,\f_3)\}$, 
otherwise the usual triangle inequality holds and we are done.
It follows therefore from our last inequality that
 $$
c_n I(\f_1,\f_2)\le 3\left(I(\f_1,\f_3)^{1/2^{n-1}}+I(\f_2,\f_3)^{1/2^{n-1}}\right)I(\f_1,\f_2)^{1-1/2^{n-1}}, 
$$
and the result follows thanks to the convexity inequality $\left(x+y\right)^{2^{n-1}}\le 2^{2^{n-1}-1}\left(x^{2^{n-1}}+y^{2^{n-1}}\right)$ for $x,y\in\R$. 
 \end{proof}

\section{The strong topology} \label{sec:strong}

\subsection{The strong topology for functions}

In order to circumvent the discontinuity of the Monge-Amp\`ere operator with respect to the weak topology, we introduce a strong topology that makes it continuous.  

\begin{defi} The \emph{strong topology} on $\cE^1(X,\om_0)$ is defined as the coarsest refinement of the weak topology such that $E$ become continuous.
\end{defi} 

\begin{lem}\label{lem:boundeden}
Every strongly convergent sequence in $\cE^1(X,\om_0)$ is contained in $\cE^1_C(X,\om_0)$ for some $C>0$. 
\end{lem}

\begin{proof} If $\f_j\to\f$ is a strongly convergent sequence in $\cE^1(X,\om_0)$ then both $\sup_X\f_j$ and $E(\f_j)$ are convergent sequences, hence are bounded.
\end{proof}

As the next result shows, on $\enorm(X,\omega_0)$,  the strong topology corresponds to the notion of \emph{convergence in energy} from \cite[\S 5.3]{BBGZ}.

\begin{prop}\label{prop:conven} 
If $\f_j,\f\in\enorm(X,\om_0)$ are \emph{normalized} $\om_0$-psh functions, then $\f_j\to\f$ in the strong topology iff $I(\f_j,\f)\to 0$.
\end{prop}

\begin{proof}
By (\ref{equ:IJ}) we have
\begin{equation}\label{equ:IE}
(n+1)^{-1}I(\f_j,\f)\le E(\f)-E(\f_j)+\int_X(\f_j-\f)\MA(\f)=J_{\f}(\f_j)\le I(\f_j,\f). 
\end{equation}
If $\f_j\to\f$ strongly then all $\f_j$ belong to $\cE^1_C(X,\om_0)$ for some $C>0$ by Lemma \ref{lem:boundeden}. By Proposition \ref{prop:conte1} it follows that $\int_X(\f_j-\f)\MA(\f)\to 0$, hence $I(\f_j,\f)\to 0$ by (\ref{equ:IE}). 

Assume conversely that $I(\f_j,\f)\to 0$. By \cite[Proposition 5.6]{BBGZ} it follows that $\f_j\to\f$ weakly. It remains to show that $E(\f_j)\to E(\f)$.  Using (\ref{equ:IJ}) we see that $\f_j\in\cE^1_C(X,\om_0)$ for some fixed $C>0$, hence $\int_X(\f_j-\f)\MA(\f)\to 0$ by Proposition \ref{prop:conte1}, and we get $E(\f_j)\to E(\f)$ using again (\ref{equ:IE}). 
\end{proof}

We are now going to show that $I$ defines a \emph{complete metrizable uniform structure} on $\enorm(X,\om)$, whose underlying topology is the strong topology. These results will not be used in the rest of the article. 

When $n=\dim_{\C} X=1$, (\ref{equ:I}) shows that $I(u,v)^{1/2}$ coincides with the $L^2$-norm of the gradient of $u-v$. As a consequence, $I^{1/2}$ defines a complete metric space structure on $\enorm(X,\om)$. Since the unit ball of the Sobolev space $W^{1,2}$ is not compact, it is easy to see that the sets $\cE^1_C(X,\om_0)$, even though they are weakly compact, are \emph{not} compact in the strong topology already for $n=1$ (compare also Lemma \ref{lem:compact}). 

In higher dimension, the quasi-triangle inequality of Theorem \ref{thm:qt} implies that $I$ defines a uniform structure, which is furthermore metrizable for general reasons \cite{Bourbaki}. Let us now show that it is complete.

\begin{prop}\label{prop:complete} Every sequence $(\f_j)$ in $\enorm(X,\om_0)$ such that $\lim_{j,k\to+\infty}I(\f_j,\f_k)=0$ converges in the strong topology of $\enorm(X,\om_0)$.
\end{prop}
\begin{proof} Pick $j_0$ such that $I(\f_j,\f_k)\le 1$ for all $j,k\ge j_0$. By (\ref{equ:IJ}) it follows that $J_{\f_{j_0}}(\f_j)$ is bounded, hence also $J(\f_j)$, by Lemma \ref{lem:encompare}. We may thus find $C>0$ such that $\f_j\in\cE^1_C(X,\om_0)$ for all $j$. Using the Cauchy-like assumption, it is as usual enough to show that some subsequence of $(\f_j)$ is strongly convergent. By weak compactness of $\cE^1_C(X,\om_0)$, we may thus assume after perhaps to a subsequence that $\f_j$ converges weakly to some $\f\in\cE^1_C(X,\om_0)$. Let us show that $\f_j\to\f$ strongly as well. Let $\e>0$, and pick $N\in\N$ such that $I(\f_j,\f_k)\le \e$ for all $j,k\ge N$. By (\ref{equ:IJ}) we get
$$
E(\f_j)-E(\f_k)+\int_X(\f_k-\f_j)\MA(\f_j)=J_{\f_j}(\f_k)\le I(\f_j,\f_k)\le \e.
$$
Since $E$ is usc in the weak topology and $\MA(\f_j)$ is weakly continuous on $\cE^1_C(X,\om_0)$ by Proposition \ref{prop:conte1}, it follows that by letting $k\to\infty$ with $j$ fixed that
$$
J_{\f_j}(\f)=E(\f_j)-E(\f)+\int_X(\f-\f_j)\MA(\f_j)\le\e
$$
for all $j\ge N$. Using again (\ref{equ:IJ}) we thus see as desired that $I(\f_j,\f)\to 0$. 
\end{proof}

\subsection{The strong topology for currents and measures}
 
We introduce the following dual strong topologies:

\begin{defi} The \emph{strong topology} on $\cT^1(X,\om_0)$ and  $\cM^1(X,\om_0)$ are respectively defined as the coarsest refinement of the weak topology such that $J$ and $E^*$ become continuous.
\end{defi}

\begin{prop}\label{prop:conven2}
The maps $\f\mapsto\om_\f$ and $\om\mapsto V^{-1}\om^n$ define homeomorphisms
$$
\enorm(X,\om_0)\simeq\cT^1(X,\om_0)\simeq\cM^1(X,\om_0)
$$
for the strong topologies. 
\end{prop}

\begin{proof}
The strong bicontituity of $\f\mapsto\om_\f$ is easy to see, since we have by definition $J(\f)=V^{-1}\int_X\f\,\om_0^n-E(\f)$ where $\f\mapsto\int_X\f\,\om_0^n$ is weakly continuous. It is thus enough to establish the continuity in the strong topology of the map $\enorm(X,\om_0)\to\cM^1(X)$ $\f\mapsto\MA(\f)$ and of its inverse. 

We thus consider $\f_j,\f\in\enorm(X,\om_0)$ and set $\mu_j:=\MA(\f_j)$, $\mu:=\MA(\f)$. Suppose first that $\f_j\to\f$ strongly. By \cite[Proposition 5.6]{BBGZ} we have $\int_X\p\mu_j\to\int_X\p\mu$ uniformly with respect to $\p\in\cE^1_C(X,\om_0)$ for each $C>0$. This shows on the one hand that $\mu_j\to\mu$ weakly, and on the other hand that $\int_X\f_j\mu_j\to\int_X\f\mu$. It follows that $E^*(\mu_j)=E(\f_j)-\int_X\f_j\mu_j$ converges to $E^*(\mu)=E(\f)-\int_X\f\mu$, so that $\mu_j\to\mu$ strongly. 

Conversely, assume that $\mu_j\to\mu$ strongly. Then $E^*(\mu_j)=E(\f_j)-\int_X\f_j\MA(\f_j)$ converges to $E^*(\mu)=E(\f)-\int_X\f\MA(\f)$, and also $\int_X\f\MA(\f_j)\to\int_X\f\MA(\f)$ by Proposition \ref{prop:conte1}, since $\MA(\f_j)=\mu_j\to\MA(\f)=\mu$ weakly. Adding up we get 
$$
E(\f_j)-E(\f)+\int_X(\f-\f_j)\MA(\f_j)\to 0,
$$ 
hence $I(\f_j,\f)\to 0$ by (\ref{equ:IE}), which shows as desired that $\f_j\to\f$ strongly. 
\end{proof}

The following dual equicontinuity properties hold:

\begin{prop}\label{prop:equicont} 
For each $C>0$ we have: 
\begin{itemize} 
\item[(i)] the set of probability measures $\cM^1_C(X,\om_0)$ acts equicontinuously on $\cE^1(X,\om_0)$ with its strong topology;
\item[(ii)] the set of $\om_0$-psh functions $\cE^1_C(X,\om_0)$ acts equicontinuously on $\cM^1(X,\om_0)$ with its strong topology. 
\end{itemize} 
\end{prop}

\begin{proof} In view of  Propositions \ref{prop:conven} and \ref{prop:conven2}, (i) follows from \cite[Lemma 5.8]{BBGZ}, while (ii) follows from  \cite[Lemma 3.12]{BBGZ}. 
\end{proof}

As a consequence, we get the following characterization of strongly compact subsets of $\cM^1(X,\om_0)$:

\begin{lem}\label{lem:compact} Let $\cK$ be a weakly compact subset of $\cM^1(X,\om_0)$. The following properties are equivalent:
\begin{itemize}
\item[(i)] $\cK$ is strongly compact;
\item[(ii)] $\cK\subset\cM^1_A(X,\om_0)$ for some $A>0$, and for each $C>0$ $\cK$ acts equicontinuously on 
$\cE^1_C(X,\om_0)$ equipped with its weak topology. 
\end{itemize}
\end{lem}

\begin{proof}  
Let us prove (i)$\Rightarrow$(ii). Since $E^*$ is by definition continuous with respect to this topology, it is bounded on $\cK$, which shows that $\cK$ is a subset of $\cM^1_A(X,\om_0)$ for some $A>0$, and is necessarily weakly closed. Assume by contradiction that $\cK$ fails to act equicontinuously on $\cE^1_C(X,\om_0)$ for some $C>0$. Then there exists sequence $\mu_j\in\cK$ and $\f_j\in\cE^1_C(X,\om_0)$ such that $\f_j\to\f$ weakly but $\int_X(\f_j-\f)\mu_j$ stays away from $0$. Using the compactness assumption we assume after perhaps passing to a subsequence that $\mu_j$ converges in energy to some $\mu\in\cK$. By Proposition \ref{prop:equicont} we then have $\int_X\f_j\mu_j\to\int_X\f\mu$, and we also have $\int_X(\f_j-\f)\mu\to 0$ by Proposition \ref{prop:conte1}. It follows that $\int_X(\f_j-\f)\mu_j\to 0$, a contradiction. 

Conversely, assume that (ii) holds and let $\mu_j$ be a sequence in $\cK$. Set $\f_j:=\f_{\mu_j}$, so that $\f_j\in\cE^1_C(X,\om_0)$ for a uniform $C>0$ by (\ref{equ:compen}). By weak compactness of $\cE^1_C(X,\om_0)$, we may assume after perhaps passing to a subsequence that $\f_j$ converges weakly to some $\f\in\cE^1_C(X,\om_0)$.  The equicontinuity assumption therefore implies that $\int_X(\f_j-\f)\MA(\f_j)\to 0$. We also have $\int_X(\f_j-\f)\MA(\f)\to 0$ by Proposition \ref{prop:conte1}, hence 
$$
I(\f_j,\f)=\int_X(\f_j-\f)\left(\MA(\f)-\MA(\f_j)\right)\to 0.
$$
By Proposition \ref{prop:conven} it follows that $\f_j\to\f$ in energy, hence $\mu_j\to\mu:=\MA(\f)$, and we have proved as desired that $\mu_j$ admits a limit point in $\cK$. 
\end{proof}

\subsection{Entropy and the H\"older-Young inequality}\label{subsec:entropy}

We first recall a general definition: 

\begin{defi} 
Let $\mu,\nu$ be probability measures on $X$. The \emph{relative entropy}  $H_\mu(\nu)\in[0,+\infty]$ of $\nu$ with respect to $\mu$ is defined as follows. If $\nu$ is absolutely continuous with respect to $\mu$ and $f:=\frac{d\nu}{d\mu}$ satisfies $f\log f\in L^1(\mu)$ then 
$$
H_\mu(\nu):=\int_X f\log f\,d\mu=\int_X\log\left(\frac{d\nu}{d\mu}\right)d\nu. 
$$
Otherwise one sets $H_\mu(\nu)=+\infty$.
\end{defi}

We write
$$
\cH(X,\mu):=\left\{\mu\in\cM(X)\mid H_\mu(\nu)<+\infty\right\}
$$
and for each $C>0$ 
$$
\cH_C(X,\mu):=\left\{\mu\in\cM(X)\mid H_\mu(\nu)\le C\right\}, 
$$
a compact subset of $\cM(X)$. We will use the following basic properties of the relative entropy. 

\begin{prop}\label{prop:ent} 
Let $\mu,\nu$ be probability measures on $X$.
\begin{itemize}
\item[(i)] We have
$$
H_\mu(\nu)=\sup_{g\in C^0(X)}\left(\int g\,d\nu-\log\int e^g d\mu\right).
$$
\item[(ii)] $H_\mu(\nu)\ge 2 \|\mu-\nu\|^2$. In particular $H_\mu(\nu)=0$ iff $\nu=\mu$.
\end{itemize}
\end{prop}

Part (i) says that $H_\mu$ is the \emph{Legendre transform} of the convex functional $g\mapsto\log\int e^g d\mu$. In particular it is convex and lower semi-continuous on $\cM(X)$. We refer to \cite[Lemma 6.2.13]{DZ} for a proof. 

The norm in (ii) denotes the total variation of $\mu-\nu$, \ie its operator norm as an element of $C^0(X)^*$. The inequality in (ii) is known as \emph{Pinsker's inequality}, see~\cite[Exercise 6.2.17]{DZ} for a proof. For later use we note:

\begin{lem}\label{lem:legsci} 
For each lower semi-continuous function $g$ on $X$ we have
$$
\sup_{\nu\in\cM(X)}\left(\int g\,d\nu-H_\mu(\nu)\right)=\log\int e^gd\mu.
$$
\end{lem}

\begin{proof} 
When $g$ is continuous this follows from Legendre duality (\ie the Hahn-Banach theorem). Assume now that $g$ is an arbitrary lsc function. The inequality 
$$
\int g\,d\nu\le\log\int e^g d\mu+H_\mu(\nu)
$$
is a direct consequence of Jensen's inequality. Conversely since $g$ is lsc there exists an increasing sequence of continuous functions $g_j\le g$ increasing pointwise to $g$. By the continuous case we get for each $j$ 
$$
\log\int e^{g_j}d\mu=\sup_{\nu\in\cM(X)}\left(\int g_j\,d\nu-H_\mu(\nu)\right)\le\sup_{\nu\in\cM(X)}\left(\int g\,d\nu-H_\mu(\nu)\right) 
$$
and the result follows by monotone convergence.
\end{proof}

We now briefly recall some facts on Orlicz spaces.
\begin{defi} A \emph{weight} is a convex non-decreasing lower semicontinuous function $\chi:[0,+\infty]\to[0,+\infty]$ such that $\chi^{-1}\{0\}=\{0\}$ and $\chi(+\infty)=+\infty$. Its \emph{conjugate weight} $\chi^*:[0,+\infty]\to[0,+\infty]$ is the Legendre transform of $\chi\left(|\cdot|\right)$, \ie 
$$
\chi^*(t):=\sup_{s\ge 0}\left(st-\chi(s)\right).
$$
\end{defi}
By Legendre duality we have $\chi^{**}=\chi$. Apart from the well-known case of the conjugate weights $s^p/p$ and $t^q/q$ with $\frac{1}{p}+\frac{1}{q}=1$, the main example for us will be:

\begin{ex}\label{ex:conj} 
The congugate weight of $\chi(s):=(s+1)\log(s+1)-s$ is 
$$
\chi^*(t)=e^t-t-1.
$$ 
\end{ex}

\begin{defi} 
Let $\mu$ be a positive measure on $X$ and let $\chi$ be a weight. The \emph{Orlicz space} $L^\chi(\mu)$ is defined as the set of all measurable functions $f$ on $X$ 
such that $\int\chi\left(\e |f|\right)d\mu<+\infty$ for some $\e>0$. 
\end{defi}

Observe that  $f\in L^\chi(\mu)$ iff $\e f$ belongs to the convex symmetric set
$$
B:=\left\{g\in L^\chi(\mu),\,\int\chi\left(|g|\right)d\mu\le 1\right\}
$$
for $0<\e\ll 1$. The \emph{Luxembourg norm} on $L^\chi(\mu)$ is then defined as the gauge of $B$, \ie one sets for $f\in L^\chi(\mu)$ 
$$
\|f\|_{L^\chi(\mu)}:=\inf\left\{\la>0,\,\int\chi\left(\la^{-1}|f|\right)d\mu\le 1\right\}.
$$
It turns $L^\chi(\mu)$ into a Banach space. 
\begin{prop}\label{prop:hy}[H\"older-Young inequality] For any two measurable functions $f\in L^\chi(\mu)$ and $g\in L^{\chi^*}(\mu)$ we have 
$$
\int\left|f g\right|d\mu\le 2\|f\|_{L^\chi(\mu)}\|g\|_{L^{\chi^*}(\mu)}.
$$
\end{prop}
We recall the straightforward proof for the convenience of the reader.
\begin{proof} We may assume that the right-hand side is non-zero. By homogeneity we may assume that $\|f\|_{L^\chi(\mu)}=\|g\|_{L^{\chi^*}(\mu)}=1$, hence $\int\chi\left(|f|\right)d\mu\le 1$ and $\int\chi^*\left(|g|\right)d\mu\le 1$. We have 
$|f g|\le\chi(|f|)+\chi^*(|g|)$ pointwise on $X$ by definition of $\chi^*$, hence $\int|fg|\,d\mu\le 2$ after integrating, and the result follows. 
\end{proof}

\begin{cor}\label{cor:hybis} Let $\nu=f\,\mu$ be a positive measure that is absolutely continuous with respect to $\mu$, and let $\chi$ be a weight function such that $\int\chi(f)\,d\mu\le A$ for some $1\le A<+\infty$. Then we have 
$$
\|g\|_{L^1(\nu)}\le 2A\|g\|_{L^{\chi^*}(\mu)}
$$
for every measurable function $g$. 
\end{cor}
\begin{proof} The assumption amounts to $\|f\|_{L^{A^{-1}\chi}(\mu)}\le 1$, and the weight $\tau_A(t):=A^{-1}\chi^*(At)$ is conjugate to $A^{-1}\chi$. On the other hand it  follows from the definition that 
$$
\|g\|_{L^{\tau_A}(\mu)}=A\|g\|_{L^{A^{-1}\chi^*}(\mu)}\le A\|g\|_{L^{\chi^*}(\mu)}
$$
since $A\ge 1$, and the result follows from the H\"older-Young inequality.
\end{proof}

\subsection{Strong compactness of measures with bounded entropy}

The goal of this section is to prove the following result, which is a main ingredient in the convergence 
of Ricci iteration and of the K\"ahler-Ricci flow. 

\begin{thm}\label{thm:compactness} 
Let $\mu_0$ be a tame probability measure on $X$. Then $\cH(X,\mu_0)\subset\cM^1(X,\om_0)$, and for each $A>0$ the set $\cH_A(X,\mu_0)$ is strongly compact.
\end{thm}

We first show that  measures of finite entropy have finite energy:
 
\begin{lem}\label{lem:entproper} 
Assume that $\mu_0$ is a probability measure on $X$ with $\a_{\om_0}(\mu_0)>0$. 
\begin{itemize}
\item[(i)] For each $0<\a<\a_{\om_0}(\mu_0)$ there exists $C>0$ such that 
$$
H_{\mu_0}(\mu)\ge\a\,E^*(\mu)-C.
$$ 
for all $\mu\in\cM(X)$. In particular we have
$$
\cH(X,\mu_0)\subset\cM^1(X,\om_0),
$$ 
and for each $A>0$ there exists $B>0$ such that $\cH_A(X,\mu_0)\subset\cM^1_B(X,\om_0)$. 

\item[(ii)] If $\a_{\om_0}(\mu_0)>\frac{n}{n+1}$ then there exists $\e,C>0$ such that
$$
H_{\mu_0}(\mu)\ge(1+\e)E^*(\mu)-C
$$
for all $\mu$. 
\end{itemize}
\end{lem}

\begin{proof} 
By definition of $\a_{\om_0}(\mu_0)$, given $\a<\a_{\om_0}(\mu_0)$ there exists $C>0$ such that
$$
\log\int e^{-\a\f}\mu_0\le-\a\sup_X\f+C
$$
for all $\om_0$-psh functions $\f$, hence
$$
-\log\int e^{-\a\f}\mu_0\ge\a V^{-1}\int_X\f\,\om_0^n-C\ge\a E(\f)-C.
$$ 
By Lemma \ref{lem:legsci} this implies
\begin{equation}\label{equ:hmin}
H_{\mu_0}(\mu)\ge\a\sup_\f\left(V^{-1}\int\f\,\om_0^n-\int\f\,\mu\right)-C\ge\a E^*(\mu)-C,
\end{equation}
which already proves (i). In order to prove (ii) we may assume that $H_{\mu_0}(\mu)$ is finite. By (i) it follows that $E^*(\mu)$ is finite as well, hence $\mu=\MA(\f)$ for some $\f\in\cE^1(X,\om_0)$. By the first inequality in (\ref{equ:hmin}) we then obtain $H_{\mu_0}(\mu)\ge\a\,I(\f)-C$, hence
$$
H_{\mu_0}(\mu)-E^*(\mu)\ge(\a-1)I(\f)+J(\f)-C
$$
$$
\ge\left(\a-1+(n+1)^{-1}\right)I(\f)-C=\left(\a-n(n+1)^{-1}\right)I(\f)-C
$$
by (\ref{equ:IJ}). (ii) follows since $I(\f)\ge I(\f)-J(\f)=E^*(\mu)$. 
\end{proof}

\begin{proof}[Proof of Theorem \ref{thm:compactness}]
By Lemma \ref{lem:compact}, it is enough to show that $\cH_A(X,\mu_0)$ acts equicontinuously on $\cE^1_C(X,\om_0)$ for each $C>0$. Let thus $\f_j\to\f$ be a weakly convergent sequence in $\cE^1_C(X,\om_0)$, and let $\mu=f\mu_0$ be a measure in $\cH_A(X,\mu_0)$. Introduce as in Example \ref{ex:conj} the weight $\chi(s):=(s+1)\log(s+1)-s$, whose conjugate function is $\chi^*(t)=e^ t-t-1$. We have $\chi(s)\le s\log s+O(1)$ on $[0,+\infty[$, hence $\int\chi(f)d\mu\le A_1$ for some $A_1\ge 1$ only depending on $A$. By Corollary \ref{cor:hybis} it follows that 
$$
\|\f_j-\f\|_{L^1(\mu)}\le 2A_1\|\f_j-\f\|_{L^{\chi^*}(\mu_0)}
$$
We are thus reduced to showing that $\|\f_j-\f\|_{L^{\chi^*}(\mu_0)}\to 0$. Using the inequality $\chi^*(t)\le t e^t$ and the definition of the norm $\|\cdot\|_{L^{\chi^*}(\mu_0)}$, we see that it is enough to show that
\begin{equation}\label{equ:limsup}
\lim_{j\to+\infty}\int_X |\f_j-\f|\exp\left(\la|\f_j-\f|\right)\mu_0=0
\end{equation}
for every given $\la>0$. But by Proposition \ref{prop:hold} there exists $B>0$ only depending on $C$ and $\la$ such that $\int e^{-2\la\f}\mu_0$ and $\int e^{-2\la\f_j}\mu_0$ are both bounded by $B$. 
Since $\sup_X\f_j\le C$ and $\sup_X\f\le C$, it follows that 
$$
\int_X\exp\left(2\la|\f_j-\f|\right)\mu_0\le B_1
$$
for some other constant $B_1>0$ independent of $j$. By H\"older's inequality we infer
$$
\int_X|\f_j-\f|\exp\left(\la|\f_j-\f|\right)d\mu\le B_1^{1/2}\|\f_j-\f\|_{L^2}(\mu_0),
$$
and (\ref{equ:limsup}) now follows since $\f_j\to\f$ in $L^2(\mu_0)$ by Proposition \ref{prop:hold}.
\end{proof}

As a consequence we get the following stability result for Monge-Amp\`ere equations: 
\begin{cor}\label{cor:stability} Let $\mu_0$ be a tame probability measure on $X$ and let $A>0$. For each $\mu\in\cM^1(X,\om_0)$ let $\f_\mu\in\cE^1(X,\om_0)$ be the unique normalized solution of $\MA(\f_\mu)=\mu$. Then $\mu\mapsto\f_\mu$ defines a continuous map from $\cH_A(X,\mu_0)$ with its weak topology to $\cE^1(X,\om_0)$ with its strong topology. 
\end{cor}

\begin{proof} By Theorem \ref{thm:compactness} and Lemma \ref{lem:compact}, the weak and strong topologies coincide on $\cH_A(X,\mu)$. We conclude using Proposition \ref{prop:conven}.
\end{proof}

This result should be compared with \cite[Theorem A]{EGZ1}, which (combined with \cite{EGZ2} to get the "continuous approximation property") implies that $\mu\mapsto\f_\mu$ defines a continuous map $L^p(\mu_0)\to C^0(X)$ for $p$ large enough.

\section{K\"ahler-Einstein metrics on log Fano pairs}\label{sec:lfp}

\subsection{Log terminal singularities}\label{sec:klt}
A \emph{pair} $(X,D)$ is the data of a connected normal compact complex variety $X$ and an effective $\Q$-divisor $D$ such that $K_X+D$ is $\Q$-Cartier. We may then consider the $dd^c$-cohomology class of $-(K_X+D)$, that we denote by $c_1(X,D)$. We write 
$$
X_0:=X_{\reg}\setminus\supp D.
$$ 
Given a log resolution $\pi:\tX\to X$ of $(X,D)$ (which may and will always be chosen to be an isomorphism over $X_0$), there exists a unique $\Q$-divisor $\sum_i a_i E_i$ whose push-forward to $X$ is $-D$ and such that 
$$
K_{\tX}=\pi^*(K_X+D)+\sum_i a_i E_i. 
$$
The coefficient $a_i\in\Q$ is known as the \emph{discrepancy} of $(X,D)$ along $E_j$, and the pair $(X,D)$ is \emph{klt} (a short-hand for Kawamata log terminal) if $a_j>-1$ for all $j$. It is a basic fact about singularities of pairs that the same condition will then hold for all log resolutions of $X$. When $D=0$, one simply says that $X$ is \emph{log terminal} when the pair $(X,0)$ is klt (so that $K_X$ is in particular $\Q$-Cartier, \ie $X$ is $\Q$-Gorenstein). 

The discrepancies $a_i$ admit the following analytic interpretation. Let $r$ be a positive integer such that $r(K_X+D)$ is Cartier. If $\sigma$ is a nowhere vanishing section of the corresponding line bundle over a small open set $U$ of $X$ then 
\begin{equation}\label{equ:adaptedloc}
\left(i^{r n^2}\sigma\wedge\bar\sigma\right)^{1/r}
\end{equation}
defines a smooth, positive volume form on $U_0:=U\cap X_0$. If $f_j$ is a local equation of $E_j$ around a point of $\pi^{-1}(U)$, it is easily seen that we have
\begin{equation}\label{equ:pullmes}
\pi^*\left(i^{r n^2}\sigma\wedge\bar\sigma\right)^{1/r}=\prod_i|f_i|^{2a_i}dV
\end{equation}
locally on $\pi^{-1}(U)$ for some local volume form $dV$. Since $\sum_i E_i$ has normal crossings, this shows that $(X,D)$ is klt iff each volume form of the form (\ref{equ:adaptedloc}) has locally finite mass near singular points of $X$.

The previous construction globalizes as follows: 

\begin{defi}\label{defi:adapted} Let $(X,D)$ be a pair and let $\phi$ be a smooth Hermitian metric on the $\Q$-line bundle $-(K_X+D)$. The corresponding \emph{adapted measure} $\mes_\phi$ on $X_\reg$ is locally defined by choosing 
a nowhere zero section $\sigma$ of $r(K_X+D)$ over a small open set $U$ and setting
\begin{equation}\label{equ:adapted}
\mes_\phi:=\left(i^{r n^2}\sigma\wedge\overline{\sigma}\right)^{1/r}/|\sigma|_{r\phi}^{2/r}.
\end{equation}
\end{defi}
The point of the definition is that the measure $\mes_\phi$ does not depend on the choice of $\sigma$, hence is globally defined. The above discussion shows that $(X,D)$ is klt iff $\mes_\phi$ has finite total mass on $X$, in which case we view it as a Radon measure on the whole of $X$. 

\begin{lem}\label{lem:adapted} 
Let $(X,D)$ be a klt pair and let $\mes_\phi$ be an adapted measure as above.  
\begin{itemize}
\item[(i)] If $\pi:\tX\to X$ is a log resolution of $(X,D)$ then the lift $\widetilde{\mes}_\phi$ of $\mu_0$ to $\tX$ writes $\widetilde{\mes}_\phi=e^{\p^+-\p^-}dV$, where $\p^\pm$ are quasi-psh functions with analytic singularities, smooth over $\pi^{-1}(X_0)$, and $e^{-\p^-}\in L^p$ for some $p>1$. In particular, $\widetilde{\mu}_0$ is tame (Definition \ref{defi:tame}). 
\item[(ii)] On $X_0$ the Ricci curvature of the volume form $\mes_\phi$ coincides with the curvature of $\phi$. 
\item[(iii)] The Bergman space $\cO(X_0)\cap L^2(\mes_\phi)$ contains only constant functions. 
\end{itemize}
\end{lem}
Property (ii) should be compared with \cite[p.319, Remark]{DT}. 
\begin{proof} Write as above $K_{\tX}=\pi^*(K_X+D)+\sum_j a_j E_j$. Let $\phi_j$ be a smooth Hermitian metric on the line bundle $\cO_{\tX}(E_j)$ and let $s_j\in H^0(\tX,E_j)$ be a section with $E_j$ as its zero divisor. If we set 
$$
\psi^+:=\sum_{a_j>0}2a_j\log|s_j|_{\phi_j}\text{ and }\psi^-:=\sum_{a_j<0}2(-a_j)\log|s_j|_{\phi_j}
$$
then (\ref{equ:pullmes}) immediately shows that
$$
\widetilde{\mes}_\phi=e^{\psi^+-\psi^-}dV
$$
where $dV$ is a (smooth positive) volume form on $\tX$. Since $a_j>-1$ there exists $p>1$ such that $pa_j>-1$ for all $j$, and the normal crossing property of $\sum_j E_j$ yields $e^{-\psi^-}\in L^p$, which proves (i). 

The proof of (ii) is straightforward from the very definition of $\mu_0$. In order to prove (iii), let $f\in\cO(X_0)\cap L^2(\mes_\phi)$. We then have 
$$
\int_{\pi^{-1}(X_0)}|f\circ\pi|^2\prod_j|s_j|_{\phi_j}^{2a_j}dV<+\infty.
$$
Since a holomorphic function extends accross a divisor as soon as it is locally $L^2$ near the divisor, the above $L^2$ condition implies that $f\circ\pi$ extends to $\tX\setminus\bigcup_{a_j>0}E_j$, or equivalently that $f$ extends holomorphically to $X\setminus Z$ with $Z:=\pi\left(\bigcup_{a_j>0}E_j\right)$. But the fact that $D=-\pi_*\left(\sum_j a_j E_j\right)$ is effective implies that each $E_j$ with $a_j>0$ is $\pi$-exceptional. As a consequence $Z$ has codimension at least two in $X$, and the normality of $X$ therefore shows that $f$ extends to $X$, hence is constant since $X$ is compact. 
\end{proof}

\subsection{K\"ahler-Einstein metrics}\label{sec:KE} 

We recall the following standard terminology: 

\begin{defi}\label{defi:logfano} 
A \emph{log Fano pair} is a klt pair $(X,D)$ such that $X$ is projective and $-(K_X+D)$ is ample. 
\end{defi}

Let $(X,D)$ be a log Fano pair. We fix a reference smooth strictly psh metric $\phi_0$ on $-(K_X+D)$, with curvature $\om_0$ and adapted measure $\mu_0=\mes_{\phi_0}$. We normalize $\phi_0$ so that $\mu_0\in\cM(X)$ is a probability measure. The volume of $(X,D)$ is 
$$
V:=c_1(X,D)^n=\int_X\om_0^n.
$$
We let $\cT(X,D):=\cT(X,\om_0)$ be the set of closed positive currents (with local potentials) $\om\in c_1(X,D)$, and $\cT^1(X,D):=\cT^1(X,\om_0)$ be those with finite energy. Similarly we denote by $\cM^1(X,D)$ the set of probability measures with finite energy, which are thus of the form $V^{-1}\om^n$ for a unique $\om\in\cT^1(X,D)$. 

For any current with full Monge-Amp\`ere mass $\om\in\Tf(X,D)$, Proposition \ref{prop:hold} guarantees that $e^{-\f_\om}\mu_0$ has finite mass, since $\mu_0$ is tame by Lemma \ref{lem:adapted}. We may thus set
\begin{equation}\label{equ:muon}
\mu_\om:=\frac{e^{-\f_\om}\mu_0}{\int_X e^{-\f_\om}\mu_0}.
\end{equation}

\begin{lem}\label{lem:contmu}
The map $\cT^1(X,\om_0)\to\cM^1(X,\om_0)$ $\om\mapsto\mu_\om$  is continous with respect to the strong topology on both sides. 
\end{lem}

\begin{proof} 
Let $\om_j\to\om$ be a strongly convergent sequence in $\cT^1(X,\om_0)$, and set $\mu_j:=\mu_{\om_j}$ and $\mu:=\mu_{\om}$. By Lemma \ref{lem:boundeden} there exists $C>0$ such that $\f_j:=\f_{\om_j}$ belongs to $\cE^1_C(X,\om_0)$ for all $j$. Since $\cE^1_C(X,\om_0)$ is weakly compact, Proposition \ref{prop:hold} shows that $\mu_j=f_j\mu_0$ and $\mu=f\mu_0$ with $\f_j\to f$ in $L^2(\mu_0)$. This implies that $\mu_j\to\mu$ weakly, and also that $\mu_j$ has uniformly bounded entropy with respect to $\mu_0$. By Theorem \ref{thm:compactness} and Lemma \ref{lem:compact}, it follows as desired that $\mu_j\to\mu$ strongly. 
\end{proof}

\begin{defi}\label{defi:KE} A \emph{K\"ahler-Einstein metric} $\om$ for the log Fano pair $(X,D)$ is a current with full Monge-Amp\`ere mass $\om\in\Tf(X,D)$ such that
\begin{equation}\label{equ:KE}
V^{-1}\om^n=\mu_\om.
\end{equation} 
\end{defi}

\begin{lem}\label{lem:KE} 
A K\"ahler-Einstein metric $\om$ is automatically smooth on $X_0$, with continuous potentials on $X$, and it satisfies
$$
\Ric(\om)=\om+[D]
$$
on $X_\reg$.
\end{lem}

Here $[D]$ is the integration current on $D|_{X_\reg}$. Writing $\Ric(\om)$ on $X_\reg$ implicitely means that the positive measure $\om^n|_{X_\reg}$ corresponds to a singular metric on $-K_{X_\reg}$, whose curvature is then $\Ric(\om)$ by definition. Note that the terminology is slightly abusive, since a K\"ahler-Einstein metric $\om$ for $(X,D)$ is not smooth on $X$ in general, hence not a K\"ahler form on $X$ in the sense of \S\ref{sec:psh}. This should hopefully cause no confusion. 

\begin{proof} 
Set $\f:=\f_\om$. The K\"ahler-Einstein equation (\ref{equ:KE}) reads
$$
(\om_0+dd^c\f)^n=e^{-\f+c}\mu_0
$$
for some constant $c\in\R$. If we choose a log resolution $\pi:\tX\to X$ of $(X,D)$, the equation becomes $(\tom_0+dd^c\tf)^n=e^{-\tf+c}\tmu_0$, where $\tom_0=\pi^*\om_0$ is semipositive and big, $\widetilde{\mu}_0$ satisfies (i) of Lemma \ref{lem:adapted}, and $\tf=\f\circ\pi$ has $e^{-\tf}\in L^q$ for all finite $q$ by Theorem \ref{thm:lelres}. By \cite{EGZ1,EGZ2} $\tf$ is continuous on $\tX$, and hence $\f$ is continuous on $X$ by properness of $\pi$. The smoothness of $\f$ on $X_0$ follows from Theorem \ref{thm:paun} (Appendix B) and the Evans-Krylov theorem.

Finally on $X_\reg$ we have 
$$
\Ric(\om)=-dd^c\log\om^n=dd^c\f-dd^c\log\mu_0=dd^c\f+\om_0+[D]
$$
by definition of the adapted measure $\mu_0=\mes_{\phi_0}$ and the Lelong-Poincar\'e formula. 
\end{proof} 

\begin{rem} Assume that $(X,D)$ is log smooth, \ie  $X$ is smooth and $D=\sum_i a_i E_i$ has simple normal crossing support. If $a_i\ge 1/2$ for all $i$, it is shown in \cite{CGP} that any K\"ahler-Einstein metric for $(X,D)$ has cone singularities along $\sum_i E_i$, with cone angle $2\pi(1-a_i)$ along $E_i$. If the support of $D=aE$ for a single smooth hypersurface $E$, \cite{JMR} shows that $\om$ has cone singularities without the restriction $a\ge 1/2$, and that $\om$ even admits a full asymptotic expansion along $E$.
\end{rem}

The definition of a log Fano pair requires the singularities to be klt. This condition is in fact necessary to obtain K\"ahler-Einstein metrics on the regular part: 

\begin{prop}\label{prop:KEklt} Let $(X,D)$ be any pair with $-(K_X+D)$ ample. Let $\Omega\subset X_\reg$ be a Zariski open subset with complement of codimension at least $2$, and assume the existence of a closed positive $(1,1)$-current $\om$ on $\Omega$ with continuous potentials such that 
$$
\Ric(\om)=\om+[D]
$$
on $\Omega$. Then $(X,D)$ is necessarily klt. We further have
$$
\int_\Omega\om^n\le c_1(X,D)^n, 
$$
with equality iff $\om$ is the restriction to $\Omega$ of a K\"ahler-Einstein metric for $(X,D)$. 
\end{prop}

\begin{proof} Let $\p$ be the singular metric on $-K_\Omega$ corresponding to the measure $\om^n$, with curvature $dd^c\p=\Ric(\om)$. If we let $\phi_D$ be the canonical psh metric on the $\Q$-line bundle attached to $D|_\Omega$, such that $dd^c\phi_D=[D]$, we have by assumption $dd^c\p=\om+dd^c\phi_D$ on $\Omega$, so that $\phi:=\p-\phi_D$ defines a psh metric on the $\Q$-line bundle $-(K_X+D)|_\Omega$, with curvature $\om$. 

Now let $\sigma$ be a local trivialisation of $r(K_X+D)$ for some positive $r\in\N$, defined on an open set $U\subset X$. If we denote by $|\sigma|_{r\phi}$ the length of $\sigma$ with respect to the metric induced by $r\phi$, then $u:=\log|\sigma|^2_{r\phi}$ is a psh function on $U\cap\Omega$, hence it automatically extends to a psh function on $U$ by normality, thanks to \cite{GR}. This means on the one hand that $\p$ extends to a globally defined psh metric on $-(K_X+D)$, so that its curvature $\om$ satisfies $\int_\Omega\om^n\le c_1(X,D)^n$ by \cite[Proposition 1.20]{BEGZ}. On the other hand, unravelling the definitions yields 
\begin{equation}\label{equ:sigmabar}
\left(\sigma\wedge\overline{\sigma}\right)^{1/r}=e^{u/r}\om^n
\end{equation}
on $V\cap\Omega$. Since $u$ is in particular bounded above, this shows that $\left(\sigma\wedge\overline{\sigma}\right)^{1/r}$ has locally finite mass near singular points of $U$, so that $(X,D)$ is klt. 

Now to say that $\int_\Omega\om^n=c_1(X,D)^n$ precisely means that $\om$ belongs to $\Tf(X,D)$, and the last assertion follows from Lemma \ref{lem:KE}. 
\end{proof}

\section{The variational principle} \label{sec:variational}
In this section $(X,D)$ denotes a log Fano pair, and we use the notation of \S\ref{sec:KE}. 

\subsection{The Ding and Mabuchi functionals}\label{sec:func}
Let $H=H_{\mu_0}$ the relative entropy with respect to the given adapted $\mu_0\in\cM(X)$. For each $\f\in\cE^1(X,\om_0)$ we set
\begin{equation}\label{equ:L}
L(\f):=-\log\int_X e^{-\f}\mu_0.
\end{equation}
Note that 
\begin{equation}\label{equ:muomfi}
\mu_{\om_\f}=e^{-\f+L(\f)}\mu_0.
\end{equation}

\begin{lem}\label{lem:L}
The map $L:\cE^1(X,\om)\to\R$ is continous in the strong topology. 
\end{lem}
\begin{proof} Let $\f_j\to\f$ be a convergent sequence in $\cE^1(X,\om_0)$. By Lemma \ref{lem:boundeden} there exists $C>0$ such that $\f_j\in\cE^1_C(X,\om_0)$ for all $j$. Since $\cE^1_C(X,\om_0)$ is weakly compact, we conclude using Proposition \ref{prop:hold}.
\end{proof}

By Lemma \ref{lem:legsci} we have
\begin{equation}\label{equ:Linf}
L(\f)=\inf_{\mu\in\cM(X)}\left(H(\mu)+\int_X\f\mu\right)
\end{equation}
and the infimum is achieved for $\mu=\mu_{\om_\f}$ by (\ref{equ:muomfi}) and the definition of $H$. This should be compared with
$$
E(\f)=\inf_{\mu\in\cM(X)}\left(E^*(\mu)+\int_X\f\mu\right),
$$
where the infimum is achieved for $\mu=\MA(\f)$. Observe also that $L(\f+c)=L(\f)+c$, so that $L-E$ is translation invariant. 
\begin{defi}\label{defi:dingmab} We introduce the following two functionals on the set $\cT^1(X,D)$ of currents with finite energy. 
\begin{itemize} 
\item[(i)] The \emph{Ding functional} $\din:\cT^1(X,D)\to\R$, defined by
$$
\din(\om):=(L-E)(\f_\om). 
$$
\item[(ii)] The \emph{Mabuchi functional} $\mab:\cT^1(X,D)\to]-\infty,+\infty]$, defined by
$$
\mab(\om):=(H-E^*)(V^{-1}\om^n). 
$$
\end{itemize}
\end{defi}

Written in the form (\ref{equ:KE}), the K\"ahler-Einstein equation is, at least formally, the Euler-Lagrange equation of the Ding functional. In the case of Fano manifolds, this functional seems to have been first explicitely considered by W.Y.Ding in \cite[p.465]{Ding}, hence the chosen terminology. 

Regarding the Mabuchi functional, our definition yields the following analogue of Chen and Tian's formula \cite{Che,Tian}
\begin{equation}\label{equ:chentian}
\mab(\om)=V^{-1}\int_X\log\left(\frac{V^{-1}\om^n}{\mu_0}\right)\om^n+(J-I)(\om). 
\end{equation}

\begin{lem}\label{lem:lsc} With respect to the \emph{strong} topology of $\cT^1(X,D)$, the Ding functional is continuous, while the Mabuchi functional is lower semicontinuous.
\end{lem}
\begin{proof} The energy $E$ is continuous on $\cE^1(X,\om_0)$ in the strong topology, by definition of the latter. The continuity of the Ding functional is thus a consequence of Lemma \ref{lem:L}. Similarly, $E^*$ is strongly continuous on $\cM^1(X,D)$, and $H$ is lsc in the weak topology, hence the result for the Mabuchi functional. 
\end{proof}

\begin{lem}\label{lem:comp} 
The Ding and Mabuchi functionals compare as follows:
\begin{itemize}
\item[(i)] For all $\om\in\cT^1(X,D)$ we have $\mab(\om)\ge\din(\om)$, with equality iff $\om$ is a K\"ahler-Einstein metric for $(X,D)$. 
\item[(ii)] We have
$$
\inf_{\cT^1(X,D)}\mab=\inf_{\cT^1(X,D)}\din\in\R\cup\{-\infty\}.
$$
\end{itemize}
\end{lem}

\begin{proof} 
Unravelling the definition we get
$$
\mab(\om)-\din(\om)=\int_X\log\left(\frac{V^{-1}\om^n}{\mu_0}\right)V^{-1}\om^n+\int_X\f_\om V^{-1}\om^n+\log\int_X e^{-\f_\om}\mu_0
$$
$$
=\int_X\log\left(\frac{V^{-1}\om^n}{\mu_\om}\right)V^{-1}\om^n=H_{\mu_\om}\left(V^{-1}\om^n\right).
$$ 
We conclude thanks to Proposition \ref{prop:ent}. 

Part (ii) is proved exactly as \cite[Theorem 3.4]{Ber10} (see also \cite{Li08}). 
We reproduce the short argument for the convenience of the reader. 
Set 
$$
m=\inf_{\cM^1(X,D)}(H-E^*)=\inf_{\cT^1(X,D)} \mab.
$$
 By (i) it is enough to show that $\din(\om)\ge m$ for all $\om\in\cT^1(X,D)$. Write $\om=\om_0+dd^c\f$ with $\f\in\cE^1(X,\om_0)$. By Lemma \ref{lem:entproper} any probability measure $\mu$ with $H(\mu)<+\infty$ belongs to $\cM^1(X,D)$, so the inequality $H(\mu)\ge E^*(\mu)+m$ is actually valid for all $\mu\in\cM(X)$. Using (\ref{equ:Linf}) we thus get
$L(\f)\ge E(\f)+m$, which concludes the proof.
\end{proof}

\subsection{Weak geodesics and convexity}\label{sec:geod}
Let $\om(0),\om(1)\in\cT^1(X,D)$ be two currents with continuous potentials, and set $\f^0:=\f_{\om(0)}$ and $\f^1:=\f_{\om(1)}$. Let $S\subset\C$ be the open strip $0<\Re t<1$ and let $\f$ be the usc upper envelope of the family of all continuous $\om_0$-psh functions (\ie $p_1^*\om_0$-psh function, with $p_1:X\times S\to X$ the first projection) $\psi$ on $X\times\overline S$ such that $\psi\le\f^0$ for $\Re t=0$ and $\psi\le\f^1$ for $\Re t=1$. Setting $\f^t:=\f(\cdot,t)$ and $\om(t):=\om_0+dd^c\f^t$ we call $(\om(t))_{t\in[0,1]}$ the \emph{weak geodesic} joining $\om(0)$ to $\om(1)$ (we also call the function $\f$ the" weak geodesic" joining $\f^0$ to $\f^1$). 

By \cite[\S 2.2]{Bern11} we have: 
\begin{lem}\label{lem:geod} 
Let $\f$ be the $\om_0$-psh envelope defined above. Then: 
\begin{itemize}
\item[(i)] $\f$ is $\om_0$-psh and bounded on $X\times S$.
\item[(ii)] $(\om_0+dd^c\f)^{n+1}=0$ on $X\times S$.
\item[(iii)] $t\mapsto\f^t$ is Lipschitz continuous, and converges uniformly on $X$ to $\f^0$ (resp. $\f^1$) as $\Re t\to 0$ (resp. $\Re t\to 1$).
\end{itemize}
\end{lem}
Here again we write $\om_0$ instead of $p_1^*\om_0$ in (ii) for simplicity. When dealing with K\"ahler forms on a non-singular $X$, (ii) gives the geodesic equation for the Mabuchi metric defined on the space of K\"ahler metrics, as was observed by Donaldson and Semmes. This explain the present terminology.  

\begin{lem}\label{lem:psh} 
Let $S$ be an open subset of $\C$, $\f$ be an $\om_0$-psh function on $X\times S$, and set $\f^t:=\f(\cdot,t)$, which is an $\om_0$-psh function unless $\f^t\equiv-\infty$. 
\begin{itemize}
\item[(i)] $t\mapsto L(\f^t)$ and $t\mapsto E(\f^t)$ are subharmonic on $S$. 
\item[(ii)] If $\f$ further satisfies (i) and (ii) of Lemma \ref{lem:geod} then $t\mapsto E(\f^t)$ is even harmonic on $S$.
\end{itemize}
\end{lem}

\begin{proof} 
The assertions for $E$ are well-known in the smooth case, and the proof in the present context reduces to \cite[Proposition 6.2]{BBGZ} by passing to a log resolution of $(X,D)$. The subharmonicity of $L(\f^t)$ is deeper, and is basically a special case of Berdntsson's theorem(s) on the subharmonic variation of Bergman kernels. Let us briefly explain how to deduce the present result from \cite{Bern06}.

Pick a log resolution $\pi:\tX\to X$, set $\tX_0:=\pi^{-1}(X_0)$, $\tom_0:=\pi^*\om_0$, and semipositive and big form on $\tX$. Since $\tX_0$ is contained in the ample locus $\Amp(\tom_0)$ of $\tom_0$, we may find a $\tom_0$-psh function $\psi$ on $\tX$ which is smooth on $\tX_0$ and such that 
$\tom_0+dd^c\psi\ge\eta$ for some K\"ahler form $\eta$ on $\tX$ (cf. Appendix B). Applying \cite{Dem92} to functions of the form $(1-\de)\f\circ\pi+\de\psi+\de|t|^2$ with $0<\de\ll 1$, we obtain (after perhaps slightly shrinking $S$) a sequence of smooth functions $\f^j$ on $\tX\times S$ such that $\f^j\to\f$ and $\tom_0+dd^c\f^j>0$ on $\tX\times S$. 

Using the isomorphism $\tX_0\simeq X_0$ induced by $\pi$, we now view $\f_j$ as a smooth, bounded $\om_0$-psh function on $X_0\times S$. For each $j$ and $t\in S$ let $\phi_j^t$ be the smooth Hermitian metric on $L:=-K_{X_0}$ defined by the (smooth positive) volume form $e^{-\f^t_j}\mu_0$. By (ii) of Lemma \ref{lem:adapted}, the curvature of $\phi_j$ on $X_0\times S$ equals $\om_0+dd^c\f_j$, hence is positive. By (iii) of Lemma \ref{lem:adapted}, the Bergman kernel for $L$-valued $(n,0)$-forms on $X_0$ with respect to $\phi_j^t$ coincides with the constant function 
$$
\left(\int_{X_0} e^{-\f^t_j}\mu_0\right)^{-1}.
$$ 
In particular, this Bergman kernel is smooth on $X_0\times S$. 

Since H\"ormander's $L^2$-estimates for $L$-valued $(n,q)$-forms apply for the positively curved line bundle $(L,\phi^j_t)$ on the weakly pseudoconvex manifold $X_0$ (cf. for instance \cite{Demhodge}), we may then argue exactly as in \cite[pp.1638-1640]{Bern06} to get that 
$$
t\mapsto-\log\int_{X_0}e^{-\f^t_j}\mu_0
$$
is subharmonic on $S$. The desired result now follows by letting $j\to\infty$.
\end{proof}

Combining these results we get the following crucial convexity property of the Ding functional along weak geodesics: 

\begin{lem}\label{lem:convdin} 
Let $(\om(t))_{t\in[0,1]}$ be the weak geodesic joining two currents $\om(0),\om(1)\in\cT^1(X,D)$ with continuous potentials. Then $t\mapsto\din(\om(t))$ is convex and continuous on $[0,1]$.
\end{lem}

\subsection{Variational characterization of K\"ahler-Einstein metrics}\label{sec:varKE}

In this section we prove the following generalization to log Fano pairs of a result of Ding and Tian for Fano manifolds (without holomorphic vector fields):

\begin{thm}\label{thm:varKE} 
Given $\om\in\cT^1(X,D)$ the following conditions are equivalent.
\begin{itemize}
\item[(i)] $\om$ is a K\"ahler-Einstein metric for $(X,D)$. 
\item[(ii)] $\din(\om)=\inf_{\cT^1(X,D)}\din$.
\item[(iii)] $\mab(\om)=\inf_{\cT^1(X,D)}\mab$. 
\end{itemize}
\end{thm}

\begin{proof} 
The equivalence betwen (i) and (ii) is proved as in \cite[Theorem 6.6]{BBGZ}, which we summarize for completeness. To prove (ii)$\Rightarrow$(i), we introduce the $\om_0$-psh envelope $Pu$ of a function $u$ on $X$ as the usc upper envelope of the family of all $\om_0$-psh functions $\f$ such that $\f\le u$ on $X$ (or $Pu\equiv-\infty$ if this family is empty). Given $v\in C^0(X)$ we set for all $t\in\R$ 
$$
\f^t:=P(\f_\om+t v).
$$ 
On the one hand, $t\mapsto L(\f_\om+t v)$ is concave by H\"older's inequality, and its right-hand derivative at $t=0$ is easily seen to be given by
$$
\int_X v\,\mu_\om,
$$
see \cite[Lemma 6.1]{BBGZ}. On the other hand, the differentiability theorem of \cite{BB} (applied in our case to a resolution of singularities of $X$) shows that $t\mapsto E(\f^t)$ is differentiable, with derivative at $t=0$ given by
$$
\int_X v\,\MA(P\f_\om)=V^{-1}\int_X v\,\om^n.
$$
Since $\f^t$ belongs to $\cE^1(X,\om_0)$, (ii) shows that $L(\f_\om+t v)-E(\f^t)$ achieves is minimum for $t=0$, and hence
$$
\frac{d}{dt}_{0+}\left(L(\f_\om+t v)-E(\f^t)\right)\ge 0, 
$$ 
\ie
$$
\int_X v\,\mu_\om\ge V^{-1}\int_X v\,\om^n. 
$$
Applying this to $-v$ shows that $\mu_\om=V^{-1}\om^n$, which means that $\om$ is a K\"ahler-Einstein metric. 

To prove (i)$\Rightarrow$(ii), we rely on the convexity of the Ding functional along weak geodesics. Let $\om$ be any K\"ahler-Einstein metric. Since every $\om_0$-psh function on $X$ is the decreasing limit of a sequence of continous $\om_0$-psh functions thanks to \cite{EGZ2}, it is enough to show that $\din(\om)\le\din(\om')$ for all $\om'\in\cT^1(X,D)$ with continuous potentials. Let $(\om(t))_{t\in[0,1]}$ be the weak geodesic between $\om(0)=\om$ and $\om(1)=\om'$. By Lemma \ref{lem:convdin} $t\mapsto\din(\om(t))$ is convex and continuous on $[0,1]$. To get as desired that $\din(\om(0))\le \din(\om(1))$, it is thus enough to show that 
\begin{equation}\label{equ:derivdin}
\frac{d}{dt}_{0_+}\din(\om(t))\ge 0, 
\end{equation}
which is proved exactly as in the last part of the proof of \cite[Theorem 6.6]{BBGZ}. More specifically, by convexity wrt $t$ of $\f_t:=\f_{\om(t)}$, the function $u_t:=(\f_t-\f_0)/t$ decreases to a bounded function $v$, and the concavity of $E$ yields
$$
\frac{d}{dt}_{0_+} E(\f_t)\le V^{-1}\int v\om^n.
$$
On the other hand, monotone convergence shows that 
$$
\frac{d}{dt}_{0+}L(\f_t)=\int v\,\om_\om=V^{-1}\int v\om^n,
$$
hence (\ref{equ:derivdin}). Finally, the equivalence between (ii) and (iii) is a consequence of Lemma \ref{lem:comp}. 
\end{proof}

\begin{rem} When $X$ is non-singular with $\Aut^0(X)=\{1\}$, the implication (i)$\Rightarrow$(iii) was proved by Ding and Tian \cite{Tia97} using the continuity method. Their result was generalized to (a priori) singular K\"ahler-Einstein metrics on any non-singular Fano variety in \cite{BBGZ}, using as above Berndtsson's theorem on psh variations of Bergman kernels. \end{rem}

As a corollary, we see that the set of K\"ahler-Einstein metrics is "totally geodesic" in the space of currents with continuous potentials: 
\begin{cor}\label{cor:geodKE} Let $\om(0)$ and $\om(1)$ be two K\"ahler-Einstein metrics, and let $(\om(t))_{t\in[0,1]}$ be the weak geodesic joining them. Then $\om(t)$ is a K\"ahler-Einstein metric for all $t\in[0,1]$. 
\end{cor}
\begin{proof} By Lemma \ref{lem:psh} $t\mapsto\din(\om(t))$ is convex on $[0,1]$, and is equal to $\inf_{\cT^1(X,D)}\din$ at both ends by Theorem \ref{thm:varKE}. It is thus constantly equal to $\inf_{\cT^1(X,D)}\din$ on $[0,1]$, and the result follows by another application of Theorem \ref{thm:varKE}.
\end{proof}

\subsection{Properness and the $\a$-invariant}\label{sec:properness}

We say as usual that the Mabuchi functional (resp. the Ding functional) is \emph{proper} if $\mab\to+\infty$ (resp. $\din\to+\infty$) as $J\to+\infty$. We also say that $\mab$ (resp. $\din$) is \emph{coercive} if there exists $\e,C>0$ such that $\mab\ge\e J-C$ (resp. $\din\ge\e J-C$). 

Regarding coercivity, the following result holds. 
\begin{prop}\label{prop:moser} The following conditions are equivalent:
\begin{itemize}
\item[(i)] The Ding functional is coercive.
\item[(ii)] The Mabuchi functional is coercive.
\item[(iii)] A Moser-Trudinger type estimate
$$
\|e^{-\f}\|_{L^p(\mu_0)}\le A\,e^{-E(\f)}
$$
holds for some $p>1$, $A>0$, and all $\f\in\cE^1(X,\om_0)$. 
\end{itemize}
\end{prop}
\begin{proof} (i)$\Longrightarrow$(ii) is obvious since $\mab\ge\din$ by Lemma \ref{lem:comp}. If (ii) holds then there exists $\e>0$ and $C>0$ such that $H(\mu)-E^*(\mu)\ge\e E^*(\mu)-C$, hence $H(\mu)\ge p E^*(\mu)-C$ with $p:=1+\e$. By Lemma \ref{lem:legsci} we then have for each $\f\in\cE^1(X,\om_0)$
$$
\log\int_X e^{-p\f}\mu_0=\sup_\mu\left(\int_X (p\f)\mu-H(\mu)\right)\le-p\inf_\mu\left(E^*(\mu)+\int_X\f\mu\right)+pC=-p E(\f)+pC 
$$
and hence $\|e^{-\f}\|_{L^p(\mu_0)}\le e^C e^{-E(\f)}$, which proves that (ii)$\Longrightarrow$(iii). Assume now that (iii) holds. Set $\e:=p-1>0$. The assumption reads
\begin{equation}\label{equ:dinglp}
\tfrac{1}{1+\e}\log\left(\int_X e^{-(1+\e)\f}\mu_0\right)\le -E(\f)+C_1
\end{equation}
for some $C_1>0$ and all $\f\in\cE^1(X,\om_0)$. On the other hand, since $\a_{\om_0}(\mu_0)>0$ by Proposition \ref{prop:hold}, we may assume that $\e>0$ is small enough so that 
\begin{equation}\label{equ:epalpha}
\log\left(\int_X e^{-\e\f}\mu_0\right)\le\e\int_X\f\MA(0)+C_2
\end{equation} 
for some $C_2>0$ and all $\f\in\psh(X,\om_0)$.

Writing $\f=(1-\e)(1+\e)\f+\e^2\f$, we have by convexity
$$
\log\left(\int_X e^{-\f}\mu_0\right)\le(1-\e)\log\left(\int_X e^{-(1+\e)\f}\mu_0\right)+\e\log\left(\int_Xe^{-\e\f}\mu_0\right)
$$
Using (\ref{equ:dinglp}) and (\ref{equ:epalpha}), it follows that
$$
L(\f)\ge(1-\e^2)E(\f)+\e^2\int_X\f\MA(0)-C_3, 
$$
and we conclude that 
$$
\din(\f)=L(\f)-E(\f)\ge\e^2\left(\int_X\f\MA(0)-E(\f)\right)-C_3=\e^2 J(\f)-C_3
$$
for $\f\in\cE^1(X,\om_0)$. We have this shown that (iii)$\Longrightarrow$(i).
\end{proof}

In order to extend Tian's well-known criterion of properness \cite{Tia87}, we introduce: 
\begin{defi}\label{defi:alphafano} The $\a$-invariant of of the log Fano pair $(X,D)$ is defined as $\a(X,D):=\a_{\om_0}(\mu_0)$, \ie
$$
\a(X,D)=\sup\left\{\a>0\mid\sup_{\f\in\pshn(X,\om_0)}\int_Xe^{-\a\f}\mu_0<+\infty\right\}. 
$$
\end{defi}
Here $\om_0$ and $\mu_0$ denote as before the curvature and the adapted measure of a smooth strictly psh metric $\phi_0$ on $-(K_X+D)$, the definition being easily seen to be independent of the choice of $\phi_0$. As an immediate consequence of Lemma \ref{lem:entproper}, we obtain as desired: 

\begin{prop}\label{prop:coercive} 
If $\a(X,D)>\frac{n}{n+1}$ then the Mabuchi functional (or, equivalently, the Ding functional) of $(X,D)$ is coercive, and hence proper. 
\end{prop}

The following topological characterization of properness is crucial to our approach. 
\begin{thm}\label{thm:mab} The Mabuchi functional 
$$
\mab:\cT^1(X,D)\to\R\cup\{+\infty\}
$$ 
is proper in the above sense if and only if it is an exhaustion function with respect to the strong topology, \ie the sublevel set
$$
\left\{\om\in\cT^1(X,D)\mid\mab(\om)\le m\right\}
$$ 
is strongly compact for all $m\in\R$. 
\end{thm}

\begin{proof} By Proposition \ref{prop:conven}, the statement is equivalent to the strong compactness of
$$
C_m:=\left\{\mu\in\cM^1(X,D)\mid H(\mu)\le E^*(\mu)+m\right\}.
$$
This set is strongly closed in $\cM^1(X,D)$ by Lemma \ref{lem:lsc}. Using (\ref{equ:compen}), the properness assumption shows that $E^*$ is bounded on $C_m$, hence $C_m\subset\cH_A(X,\mu)$ for some $A>0$. The result follows since $\cH_A(X,\mu)$ is strongly compact by Theorem \ref{thm:compactness}. The converse is straightforward to see, using Proposition \ref{prop:conven} and (\ref{equ:compen}). 
\end{proof}

\begin{cor}\label{cor:KE}
It the Mabuchi functional is proper, then $(X,D)$ admits a K\"ahler-Einstein metric. 
\end{cor}

\begin{proof} Theorem \ref{thm:mab} implies that $\mab$ admits a minimizer in $\cT^1(X,D)$, which is necessarily a K\"ahler-Einstein metric by Theorem \ref{thm:varKE}.  
\end{proof}

We will prove in Theorem \ref{thm:uniqueKE} below that $\om$ is furthermore unique in that case, and that $\Aut^0(X,D)=\{1\}$. Recall that, for $X$ non-singular with $\Aut^0(X)=\{1\}$ (and $D=0$), a deep result of Tian \cite{Tia97}, strengthened in \cite{PSSW}, conversely shows that the existence of a K\"ahler-Einstein metric implies the properness of the Mabuchi functional - an infinite dimensional version of the Kempf-Ness theorem.

\section{Uniqueness and reductivity} \label{sec:unique}

The goal of this section is to prove a singular version of Bando and Mabuchi's uniqueness theorem \cite{BM87}. 
As we shall see, it is a fairly direct consequence of a slight variant of a result of Berndtsson \cite{Bern11}, stated and proved with full details in Appendix C.
  
\begin{thm}\label{thm:unique} Let $(X,D)$ be a log Fano pair. For any two K\"ahler-Einstein metrics $\om,\om'$ for $(X,D)$, there exists a $1$-parameter subgroup $\la:(\C,+)\to\Aut^0(X,D)$ such that 
$\la(1)^*\om=\om'$ and $\la(is)^*\om=\om$ for all $s\in\R$. 
\end{thm}
Here $\Aut(X,D)$ denotes the stabilizer of $D$ in $\Aut(X)$. Since $\Aut(X,D)$ preserves the polarization $-(K_X+D)$, it is realized as the stabilizer of $X$ and $D$ in the linear group of $H^0(X,-m(K_X+D))$ for $m$ large enough, and it is therefore a linear algebraic group. Further, recall that the identity component $\Aut^0(X)$ of the full automorphism group is also a complex algebraic group with Lie algebra $H^0(X,T_X)$, where $T_X=(\Omega^1_X)^*$ denotes the Zariski tangent sheaf (cf.~\cite[Exercise 2.6.4]{Koll}).

\begin{proof}
Let $\pi:\tX\to X$ be a log resolution of $(X,D)$. The klt condition enables to write in a unique way 
$$
\pi^*(K_X+D)=K_{\tX}+\D-E
$$ 
where $\D$ is an effective $\Q$-divisor on $\tX$ with coefficients in $[0,1)$, $E$ is an effective divisor on $\tX$ with integer coefficients, and $\pi_*(\D-E)=D$. We do not claim that $\D$ and $E$ are without common components, but on the other hand observe that $\D$ and $E$ have SNC support and $E$ is necessarily $\pi$-exceptional. Set $L:=-K_{\tX}+E$, so that the canonical section of $\cO_X(E)$ induces a holomorphic $L$-valued $n$-form $u$ on $X$ having $E$ as its zero divisor. Since $E$ is exceptional, we have $H^0(\tX,K_{\tX}+L)=\C u$. Note also that $L=M+\D$ where $M:=-\pi^*(K_X+D)$ is a semipositive, big $\Q$-line bundle. By the Kawamata-Viehweg (or Nadel) vanishing theorem, we thus have $H^1(\tX,K_{\tX}+L)=0$. 

Choose continuous psh metrics $\psi_0,\psi_1$ on $-(K_X+D)$ with curvature $\om$, $\om'$ respectively, and normalized so that $E(\psi_0)=E(\psi_1)=0$. Let $\psi$ be the 'weak geodesic' joining $\p_0$ to $\p_1$ as in \S\ref{sec:geod}, so that $\psi$ is a locally bounded psh metric on the pull-back of $-(K_X+D)$ to $X\times S$ with $S=\left\{t\in\C\mid 0<\Re t<1\right\}$. Recall that $t\mapsto\psi_t$ is Lipschitz continuous on $S$, independent of $\Im t$, and converges uniformly to $\psi_0$ (resp. $\psi_1$) as $t\to 0$ (resp. $t\to 1$). 

By Lemma \ref{lem:psh}, $E(\psi_t)$ is an affine function of $t\in(0,1)$, which converges to $E(\psi_0)=E(\psi_1)=0$ as $t\to 0$ and $1$, and hence $E(\psi_t)=0$ for all $t\in S$.   

By Lemma \ref{lem:convdin}, $t\mapsto\din(\psi_t)=L(\psi_t)$ is continuous and convex on $[0,1]$, with $\din(\psi_0)=\din(\psi_1)=\inf\din$ thanks to the variational characterization of K\"ahler-Einstein metrics (Theorem \ref{thm:varKE}). It follows that $\din(\psi_t)=\inf\din$ for all $t\in S$, and another application of the variational characterization shows that $\om_t:=dd^c\psi_t$ is a K\"ahler-Einstein metric for $(X,D)$ for all $t\in S$.

Now set $\tau:=\pi^*\p$, let $\phi:=\tau+\phi_\D$ be the, with $\phi_\D$ the canonical metric on $\D$ with curvature current equal to $[\D]$, so that $\phi$ is a psh metric on the pull-back of $L$ to $X\times S$, and observe that 
$$
L(\psi_t)=-\log\int_{\tX} u\wedge\bar u\,e^{-\phi_t}.
$$
By Theorem \ref{thm:berndt} in Appendix C, there exists a holomorphic vector field $V$ on $\tX\setminus E$ such that 
$dd^c_z\phi_t=dd^c_z\tau_t+[\D]$ moves along the local flow of $V$ on $\tX\setminus E$. Using for instance the uniqueness of the Siu decomposition of a closed positive $(1,1)$-current, it follows that $dd^c_z\tau_t=\pi^*\om_t$ also moves along the local flow of $V$, \ie
\begin{equation}\label{equ:lie}
\left(\cL_V+\frac{\partial}{\partial t}\right)\pi^*\om_t=0.
\end{equation}
Since $E$ is $\pi$-exceptional, $V$ induces a holomorphic vector field on a Zariski open set $\Omega\subset X_\reg$ with $X\setminus\Omega$ of codimension at least $2$. But the Zariski tangent sheaf $T_X$ is a dual sheaf and $X$ is normal, so $V$ extends to a section in $H^0(X,T_X)$, still denoted by $V$ for simplicity. Since $H^0(X,T_X)$ is the Lie algebra of the complex Lie group $\Aut^0(X)$, $\la(t):=\exp(-tV)$ defines a $(\C,+)$-action on $X$. By (\ref{equ:lie}), $\la(-t)^*\om_t$ is independent of $t\in S$. For all $t,t'\in S$, $\la(-t)^*\psi_t-\la(-t')^*\psi_{t'}$ is thus a constant, and
$$
L(\la(-t)^*\psi_t)=L(\psi_t)=L(\psi_{t'})=L(\la(-t')^*\psi_{t'})
$$
shows that in fact $\la(-t)^*\psi_t=\la(-t')^*\psi_{t'}$. Since $\psi_{t'}$ converges uniformly to $\psi_0$ as $t'\to 0$, we conclude that $\psi_t=\la(t)^*\psi_0$ for all $t\in S$, and hence $\psi_1=\la(1)^*\psi_0$ in the limit. Since $\psi_t$ is independent of $\Im t$, we have $\la(\e+ is)^*\psi_0=\psi_\e$ for $0<\e<1$ and $s\in\R$, and hence $\la(is)^*\psi_0=\psi_0$ in the limit. Finally, the K\"ahler-Einstein equation shows that
$$
\la(t)^*(\om_0+[D])=\la(t)^*\Ric(\om_0)=\Ric(\la(t)^*\om_0)=\Ric(\om_t)=\om_t+[D]=\la(t)^*\om_0+[D]
$$ 
on $X_\reg$, so that $\la$ is indeed a $1$-parameter subgroup of $\Aut^0(X,D)$. 
\end{proof}


Following \cite{CDS3}, we observe that Theorem \ref{thm:unique} yields the following generalization of a classical result of Matsushima \cite{Mat}, which plays a key role in \cite{CDS3}.

\begin{thm}\label{thm:reductive} 
If the log Fano pair $(X,D)$ admits a K\"ahler-Einstein metric $\om$, then $\Aut^0(X,D)$ coincides with the complexification of the group $K$ of holomorphic isometries of $(X,\om)$. Further, $K$ is compact, and $\Aut^0(X,D)$ is thus reductive. 
\end{thm}
\begin{proof} For each $g\in\Aut^0(X,D)$, $g^*\om$ is a K\"ahler-Einstein metric for $(X,D)$. By Theorem \ref{thm:unique}, we may thus find a one-parameter subgroup $\la$ of $\Aut^0(X,D)$ with $\la(is)\in K$ for all $s\in\R$ and such that $\la^*(1)\om=g^*\om$. The first condition implies that $\la(1)$ lies in the complexification of $K$, and the second one means $g\in K\la(1)$, which proves the first assertion. 

Next, write $\om$ as the curvature form of a metric $\phi$ on the $\Q$-line bundle $-(K_X+D)$, normalized so that $L(\phi)=0$. For each $g\in K$, we have $g^*\om^n=\om^n$. Since $L(g^*\phi)=L(\phi)=0$, the K\"ahler-Einstein equation yields as before $g^*\phi=\phi$. In other words, $K$ is the stabilizer of $\phi$, and we conclude by Lemma \ref{lem:compactisom} below. 
\end{proof} 

\begin{lem}\label{lem:compactisom} Let $(X,L)$ be a polarized normal variety, and let $G=\Aut(X,L)$ be its (linear algebraic) group of automorphisms. For each psh metric $\phi\in\cE^1(X,L)$, the stabilizer $G_\phi=\left\{g\in G\mid g^*\phi=\phi\right\}$ is then compact. 
\end{lem} 

\begin{proof} Replacing $L$ with a large enough multiple, we may assume that $L$ is very ample. Choosing a basis $(s_1,\dots,s_N)$ for $H^0(X,L)$, we get an embedding $X\subset\PP^{N-1}$, which realizes $G$ as the Zariski closed subgroup of $\mathrm{GL}(N,\C)$ sending $X$ to itself. Let 
$$
\phi_0=\log\left(\sum|s_i|^2\right)
$$
be the corresponding Fubini-Study metric on $L$, with curvature form $\om_0$. We then introduce the $\om_0$-psh functions $\p:=\phi-\phi_0\in\cE^1(X,\om_0)$, and
\begin{equation}\label{equ:phig}
\f_g:=g^*\phi_0-\phi_0=\log\left(\sum_i|\sum_j g_{ij} s_j|^2_{\phi_0}\right)
\end{equation}
for each $g=(g_{ij})\in\mathrm{GL}(N,\C)$. By homogeneity, it is easy to see that
$$
\int_X\f_g\om_0^n=\log\|g\|^2+O(1)
$$
for any given choice of norm $\|\cdot\|$ on the space of complex $N\times N$ matrices (compare the proof of \cite[Lemma 3.2]{Tia14}). It will thus be enough to show that $\int_X\f_g\om_0^n$ remains bounded for $g\in G_\phi$. 

By (\ref{equ:energybis}), we have
$$
E(\p)-E(\f_g)=\frac{1}{n+1}\sum_{j=0}^nV^{-1}\int(g^*\phi-g^*\phi_0)(dd^c g^*\phi)^j\wedge(dd^c g^*\phi_0)^{n-j}
$$
$$
=\frac{1}{n+1}\sum_{j=0}^nV^{-1}\int(\phi-\phi_0)(dd^c \phi)^j\wedge(dd^c \phi_0)^{n-j}=E(\p),
$$
and hence $E(\f_g)=0$ for all $g\in G_\phi$. For similar reasons, we have
$J_\p(\f_g)=J_{\p}(0)$, and Lemma \ref{lem:encompare} yields as desired that $J(\f_g)=\int\f_g\om_0^n$ is bounded. 
\end{proof}

The next result summarizes the consequences of the properness of the Mabuchi functional regarding K\"ahler-Einstein metrics. 

\begin{thm}\label{thm:uniqueKE} 
Assume that the Mabuchi functional of $(X,D)$ is proper. Then we have: 
\begin{itemize}
\item[(i)] $\Aut^0(X,D)=\{1\}$.
\item[(ii)] $(X,D)$ admits a unique K\"ahler-Einstein metric $\om_\mathrm{KE}$.
\item[(iii)] For every sequence $\om_j\in\cT^1(X,D)$ such that 
$\mab(\om_j)$ converges to $\inf_{\cT^1(X,D)}\mab$, we have $\om_j\to\om_\mathrm{KE}$ in the strong topology of $\cT^1(X,D)$. 
\end{itemize}
\end{thm}

\begin{proof} 
By Corollary \ref{cor:KE} there exists a K\"ahler-Einstein metric $\om$. Let us prove (i). Let $\la$ be a  one-parameter subgroup of $\Aut^0(X,D)$. We claim that $\la$ preserves $\om$ for all $t\in\C$. Granting this, the affine variety $\Aut^0(X,D)$ will be contained in the compact group of isometries of $\om$, and hence will be trivial. To prove the claim,  observe that $\la(t)^*\om$ is a K\"ahler-Einstein metric for each $t\in\C$. If we let $\phi$ be a continuous psh metric on $-(K_X+D)$ with curvature $\om$ and set $\f^t:=\la(t)^*\phi-\phi_0$, then $\f(x,t):=\f^t(x)$ is a continuous $\om_0$-psh function on $X\times\C$ such that 
$$
(\om_0+dd^c\f)^{n+1}=0
$$ 
on $X\times\C$. By Lemma \ref{lem:psh} $E(\f^t)$ is thus harmonic on $\C$, while $\int_X(\f^t-\f_\om)\MA(\f_\om)$ is subharmonic, simply because $t\mapsto\f^t(x)$ is subharmonic for each $x\in X$ fixed.  It follows that 
$$
J_{\om}(\la(t)^*\om)=E(\f_\om)-E(\f^t)+\int_X(\f^t-\f_\om)\MA(\f_\om)
$$
is subharmonic and bounded on $\C$, hence constant since it vanishes for $t=0$. By \cite[Theorem 4.1]{BBGZ} it follows that as desired that $\la(t)^*\om=\om$ for all $t\in\C$. 

By Theorem \ref{thm:unique}, (i) implies the uniqueness part in (ii). It remains to prove (iii). Since $\mab(\om_j)$ is in particular bounded above, $(\om_j)$ stays in a strongly compact set by Theorem \ref{thm:mab}. It is thus enough to show that any strong limit point $\om_\infty$ of $\om_j$ has to coincide with $\om$. But since $\mab$ is lsc in the strong topology by Lemma \ref{lem:lsc}, we have $\mab(\om_\infty)\le\liminf_j\mab(\om_j)=\inf_{\cT^1(X,D)}\mab$. By Theorem \ref{thm:varKE} it follows as desired that $\om_\infty=\om$.
\end{proof} 

We will also use the following variant for the Ding functional to prove the convergence of the K\"ahler-Ricci flow: 
\begin{lem}\label{lem:cvdin} 
Assume that the Mabuchi functional of $(X,D)$ is proper, and let $\om_j\in\cT^1(X,\D)$ be a sequence such that $\mab(\om_j)$ is bounded above. If $\din(\om_j)$ converges to $\inf_{\cT^1(X,D)}\din$, then $\om_j$ strongly converges to the unique K\"ahler-Einstein metric $\om_\mathrm{KE}$ of $(X,D)$. 
\end{lem}

\begin{proof} The assumption guarantees that $(\om_j)$ stays in a strongly compact set, and it is thus enough to show that any limit point $\om_\infty$ of $(\om_j)$ in the strong topology necessarily coincides with $\om_\mathrm{KE}$. But the minimizing assumption on $(\om_j)$ and the strong continuity of $\din$ (Lemma \ref{lem:lsc}) imply that $\din(\om_\infty)=\inf_{\cT^1(X,D)}\din$, and we conclude by Theorem \ref{thm:varKE}.
\end{proof}

\section{Ricci iteration}\label{sec:iteration}

We still denote by $(X,D)$ a log Fano pair, and use the notation of \S\ref{sec:func}. 

\subsection{The Ricci inverse operator}
For each $\om\in\cT^1(X,D)$, the measure $\mu_\om$ is tame by Proposition \ref{prop:hold}. Using Lemma \ref{lem:tame} we may thus introduce:

\begin{defi}\label{defi:R} For each $\om\in\cT^1(X,D)$ we let $R\om\in\cT^1(X,D)$ be the unique current with continuous potentials such that 
\begin{equation}\label{equ:R}
V^{-1}\left(R\om\right)^n=\mu_\om.
\end{equation}
The map $R:\cT^1(X,D)\to\cT^1(X,D)$ so defined is called the \emph{Ricci inverse operator}. 
\end{defi}
The defining equation for $R\om$ may be rewritten as 
\begin{equation}\label{equ:ricciR}
\Ric(R\om)=\om+[D]
\end{equation}
on $X_\reg$, and $R\om=\om$ iff $\om$ is a K\"ahler-Einstein metric. 

\begin{lem}\label{lem:R} 
The Ricci inverse operator satisfies the following properties.
\begin{itemize}
\item[(i)] $R:\cT^1(X,D)\to\cT^1(X,D)$ is continuous with respect to the strong topology.
\item[(ii)] If $\om\in\cT^1(X,D)$ is smooth on a given open subset $U$ of $X_0$, then so is $R\om$.
\end{itemize}
\end{lem}

\begin{proof} (i) follows from Proposition \ref{prop:conven2} and Lemma \ref{lem:contmu}. As in the proof of Lemma \ref{lem:KE}, assertion (ii) is a consequence of Theorem \ref{thm:paun} applied to a log resolution of $(X,D)$, combined with the Evans-Krylov theorm. 
\end{proof}

As in \cite{Rub,Kel} we next observe that the Mabuchi functional decreases along $R$: 

\begin{lem}\label{lem:monotone} For all $\om\in\cT^1(X,D)$ we have $\mab(R\om)\le\mab(\om)$, with equality iff $R\om=\om$, \ie $\om$ is a K\"ahler-Einstein metric. 
\end{lem}

\begin{proof} We have by definition 
$$
\mab(R\om)=(H-E^*)(V^{-1}\om^n)=(H-E^*)(\mu_\om).
$$
Now (\ref{equ:estar}) implies in particular that $E^*(\mu_\om)\ge E(\f_\om)-\int_X\f_\om\,\mu_\om$, whereas we have
$$
H(\mu_\om)=L(\f_\om)-\int_X\f_\om\,\mu_\om
$$ 
by definition of $L$ and $H$. As a consequence we get
$$
\mab(R\om)\le (L-E)(\f_\om)=\din(\om)
$$
and the result follows thanks to Lemma \ref{lem:comp}.
\end{proof}

\subsection{Convergence of Ricci iteration}

Our goal in this section is to prove the following result, which extends in particular \cite[Theorem 3.3]{Rub}. 

\begin{thm}\label{thm:ri} 
Let $(X,D)$ be a log Fano pair whose Mabuchi functional is proper, and let $\om_\mathrm{KE}$ be its unique K\"ahler-Einstein metric.  
\begin{itemize}
\item[(i)] For all $\om\in\cT^1(X,D)$ we have 
$$
\lim_{j\to+\infty}R^j\om=\om_\mathrm{KE}
$$ 
in the strong topology, the convergence being uniform at the level of potentials. 
\item[(ii)] If $\om$ is smooth on an open subset $U$ of $X_0$, then $R^j\om$ is also smooth on $U$ for all $j$, and the convergence holds in $C^\infty(U)$.  
\end{itemize}
\end{thm}

As mentioned in the introduction, a more precise version of this result was obtained in \cite{JMR} when $X$ is non-singular and the support of $D$ is a smooth hypersurface.

\begin{proof} {\bf Step 1: Convergence in energy}. Set 
$$
\cM:=\left\{\om'\in\cT^1(X,D)\mid\mab(\om')\le\mab(\om)\right\}.
$$
Note that $\om_\mathrm{KE}$ belongs to $\cM$, since it minimizes $\mab$ by Theorem \ref{thm:varKE}. Theorem \ref{thm:mab} implies that $\cM$ is strongly compact, since the Mabuchi functional of $(X,D)$ is assumed to be proper. By Lemma \ref{lem:monotone} $R$ defines a strongly continuous map $R:\cM\to\cM$. We are going to show that $R^j\om$ converges strongly to $\om_\mathrm{KE}$ by using a Lyapunov-type argument. Since $\cM$ is strongly compact, it is enough to show that any limit point $\om_\infty$ of $R^j\om$ necessarily coincides with $\om_\mathrm{KE}$, \ie is a fixed point of $R$. Since $\mab(R^j\om)$ is non-increasing by Lemma \ref{lem:monotone} and bounded below by Theorem \ref{thm:varKE}, it admits a limit 
$$
\lim_{j\to\infty}\mab(R^j\om)=m. 
$$ 
By continuity of $R$, both $\om_\infty$ and $R\om_\infty$ are limit points of $(R^j\om)$, hence $\mab(R\om_\infty)=\mab(\om_\infty)=m$, which shows that $\om$ is a fixed point of $R$ by Lemma \ref{lem:monotone}.\\

\noindent {\bf Step 2: Uniform convergence on $X$.} Setting $\f_j:=\f_{R^j\om}$ and $\f_\mathrm{KE}:=\f_{\om_\mathrm{KE}}$. By Lemma \ref{lem:R}, all $\f_j$ are continuous, and we are to show that $\f_j\to\f_\mathrm{KE}$ uniformly on $X$. Unravelling the definitions, we get that 
$$
\MA(\f_{j+1})=e^{-\f_j+L(\f_j)}\mu_0
$$ 
for all $j\in\N$. Since $\f_j$ converges strongly to $\f$, we have $L(\f_j)\to L(\f)$ by Lemma \ref{lem:lsc}, while $e^{-\f_j}\to e^{-\f}$ in $L^p(\mu_0)$ for all finite $p$ by Proposition \ref{prop:hold}. If we pick a log resolution $\pi:\tX\to X$ and use the usual notation, we get that
$$
f_j:=\frac{(\tom_0+dd^c\tf_j)^n}{dV}
$$ 
converges in $L^p$ 
to
$$
f:=\frac{(\tom_0+dd^c\tf_\mathrm{KE})^n}{dV}
$$
for some $p>1$. By \cite[Theorem A]{EGZ1} it follows that $\tf_j\to\tf_\mathrm{KE}$ uniformly on $\tX$, hence $\f_j\to\f_\mathrm{KE}$ uniformly on $X$.\\ 

\noindent {\bf Step 3: Smooth convergence}. Let again $\pi:\tX\to X$ be a log resolution, and write $\tmu_0=e^{\p^+-\p^-}dV$ as in Lemma \ref{lem:adapted}. We then have for all $j\in\N$ 
\begin{equation}\label{equ:fj}
V^{-1}(\tom_0+dd^c\tf_{j+1})^n=e^{\psi^+-\psi^--\tf_j+L(\f_j)}dV
\end{equation}
By Step 2 $\tf_j$ and $L(\f_j)$ are uniformly bounded. Since $\p^-$ is locally bounded below on $\tU:=\pi^{-1}(U)\simeq U$, Theorem \ref{thm:paun} shows that the complex Hessian of $\tf_j$ is locally bounded on $\tU$, uniformly with respect to $j$. In particular, the functions $\p^+-\p^--\tf_j+L(\f_j)$ appearing in the right-hand side of (\ref{equ:fj}) are locally Lipschitz continuous on $\tU$, uniformly with respect to $j$. Applying the version of the Evans-Krylov \emph{a priori}  estimate given in \cite[Theorem 4.5.1]{Blo}, we get a uniform $C^{2+\a}$-bound for $\f_j$ on compact subsets of $U$, which implies as desired that $\tf_j$ is smooth and uniformly bounded in $C^\infty(\tU)$, by applying the standard elliptic boot-strapping argument to (\ref{equ:fj}). 

\end{proof}

\section{Convergence of the K\"ahler-Ricci flow}\label{sec:flow}
To avoid unnecessary technical complications in the definition of the K\"ahler-Ricci flow, we assume in this section that $D=0$, so that $X$ is a $\Q$-Fano variety with log terminal singularities. 

\subsection{The K\"ahler-Ricci flow}
The following result can be deduced from \cite{ST} (a detailed proof is provided in \cite{BG}):

\begin{thm}\label{thm:ST}\cite{ST} 
Given any initial closed positive current $\om(0)\in\cT^1(X)$ with continuous potentials, there exists a unique solution $(\om(t))_{t\in]0,+\infty[}$, of the normalized K\"ahler-Ricci flow, in the following sense: 
\begin{itemize} 
\item[(i)] for each $t\in]0,+\infty[$, $\om(t)\in\cT^1(X)$ has continuous potentials on $X$;
\item[(ii)] on $X_\reg\times]0,+\infty[$ $\om(t)$ is smooth and satisfies $\dot\om(t)=-\Ric(\om(t))+\om(t)$;
\item[(iii)] $\lim_{t\to 0_+}\om(t)=\om(0)$, in the sense that their local potentials converge in 
${\mathcal C}^0(X_\reg)$.  
\end{itemize}
\end{thm}

Our goal is to prove the following convergence result.
\begin{thm}\label{thm:convKRF} Let $(\om(t))_{t\in]0,+\infty[}$ be the K\"ahler-Ricci flow with initial data $\om(0)$ as above. Assume that the Mabuchi functional of $X$ is proper, and let $\om_\mathrm{KE}$ be its unique K\"ahler-Einstein metric as in Theorem \ref{thm:uniqueKE}. Then 
$$
\lim_{t\to+\infty}\om(t)=\om_\mathrm{KE}
$$ 
in the strong topology of $\cT^1(X)$. 
\end{thm}

\subsection{Monotonicity along the flow}\label{sec:monotone}
We show in this section that the Ding and Mabuchi functionals are both non-increasing along the flow, as in the usual non-singular setting. 

\begin{prop}\label{prop:monotoneflow} Let $(\om(t))_{t\in]0,+\infty[}$ be the K\"ahler-Ricci flow with initial data $\om(0)$. Then $\mab(\om(t))$ and $\din(\om(t))$ are both non-increasing functions of $t$. For all $t'>t>0$ we have more precisely
$$
\din(\om(t'))-\din(\om(t))\le-\int_{t}^{t'}\left\|V^{-1}\om(s)^n-\mu_{\om(s)}\right\|^2 ds.
$$
\end{prop}

The monotonicity of the two functionals is standard in the non-singular case, where $\om(t)$ is smooth on $X\times[0,+\infty[$. The technical difficulty in the present case is that we cannot directly differentiate $\din(\om(t))$ and $\mab(\om(t))$  since $\dot\om(t)$ is \emph{a priori} not globally bounded on $X_\reg$. We will rely on an approximation argument, using the following specific information about the construction of $\om(t)$. What Song and Tian construct in \cite{ST} is a function $\f:X\times]0,+\infty[\to\R$ with the following properties:
\begin{itemize}
\item $\f$ is smooth on $X_\reg\times]0,+\infty[$, and $\f^t:=\f(\cdot,t)$ is a continuous $\om_0$-psh function for each $t$ fixed. \item On $X_\reg\times]0,+\infty[$ we have
$$
\pddt\f=\log\frac{V^{-1}(\om_0+dd^c\f^t)^n}{\mu(t)},
$$
with
$$
\mu(t):=\frac{e^{-\f^t}\mu_0}{\int_Xe^{-\f^t}\mu_0}. 
$$
\item $\lim_{t\to 0_+}\f^t=\f_{\om(0)}$ uniformly on compact subsets of $X_\reg$.  
\end{itemize}
Let $\pi:\tX\to X$ be a log resolution of $X$, so that the exceptional divisor $E=\pi^{-1}(X_\mathrm{sing})$ has simple normal crossings. Set $\tom_0:=\pi^*\om_0$, which is semipositive and big on $\tX$, with ample locus $\tX_0:=\tX\setminus E$. By Lemma \ref{lem:adapted}, the pull-back of $\mu_0$ to $\tX$ is of the form 
$$
\widetilde{\mu}_0=e^{\p^+-\p^-}dV,
$$
where $\p^\pm$ are quasi-psh functions with analytic singularities alors $E$. Pick a K\"ahler form $\eta$ on $\tX$. In Song and Tian's construction, the restriction of $\tf:=\f^t\circ\pi$ to $\tX_0\times]0,+\infty[$ is the $C^\infty$-limit (on compact sets) of the restriction of a sequence of smooth functions $\f_j:\tX\times]0,+\infty[$ such that:
\begin{itemize} 
\item there exists $\e_j>0$ converging to $0$ such that $\f^t_j$ is $\om_j$-psh with $\om_j:=\tom_0+\e_j\eta$.
\item On $\tX\times]0,+\infty[$ we have 
\begin{equation}\label{equ:krfj}
\pddt\f_j=\log\frac{V_j^{-1}(\om_j+dd^c\f^t_j)^n}{\mu_j(t)},
\end{equation}
where $V_j=\int_{\tX}\om_j^n$,
$$
\mu_j(t)=\frac{e^{-\f^t_j}\mu_j}{\int_Xe^{-\f^t_j}\mu_j}
$$
with $\mu_j=e^{\p^+_j-\p^-_j}dV$ for decreasing sequences of smooth approximants $\p^\pm_j$ of $\p^\pm$. 
\end{itemize}

We also extract from \cite{ST} the following estimate to be used in what follows:

\begin{lem}\label{lem:lpbound} There exists $p>1$ such that the measures $(\om_j+dd^c\f^t_j)^n$ are bounded in $L^p$, uniformly with respect to $j$, as long as $t$ stays in a compact subset of $]0,+\infty[$. 
\end{lem}
\begin{proof} By \cite[Corollary 3.4]{ST} we have a uniform estimate $(\om_j+dd^c\f^t_j)^n\le C\,\mu_j$, as long as $t$ stays in a compact subset of $]0,+\infty[$. The result follows since $\mu_j=e^{\p_j^+-\p^-_j}dV$ is bounded in $L^p$. 
\end{proof}

\begin{lem}\label{lem:convep} 
Let $E_j$ be the energy functional on $\cE^1(\tX,\om_j)$, and $E_j^*$ the dual functional on $\cM^1(\tX,\om_j)$. Then we have for each $t>0$ fixed
$$
\din(\om(t))=\lim_{j\to\infty}\left(-\log\left(\int_{\tX}e^{-\f^t_j}\mu_j\right)-E_j(\f^t_j)\right)
$$
and 
$$
\mab(\om(t))=\lim_{j\to\infty}\left(H_{\mu_j}-E_j^*\right)\left(V_j^{-1}\left(\om_j+dd^c\f^t_j\right)^n\right).
$$
\end{lem}

\begin{proof} 
By Lemma \ref{lem:lpbound} 
$\f^t_j$ is uniformly bounded with respect to $j$, $t$ being fixed. Since $\f^t_j\to\tf^t$ in $C^\infty$ topology on $\tX_0$, dominated convergence yields
$$
\lim_{j\to+\infty}\int_{\tX_0}\left(\f^t_j-\f^t\right)\left(\om_j+dd^c\tf^t\right)^n=0.
$$
On the other hand, since $(\om_j+dd^c\f^t_j)^n\to(\tom_0+dd^c\tf^t)^n$ pointwise on $\tX_0$ and $(\om_j+dd^c\f^t_j)^n$ is bounded in $L^p$, we also get by dominated convergence
$$
\lim_{j\to+\infty}\int_{\tX}(\f^t_j-\tf^t)(\om_j+dd^c\f^t_j)^n=0
$$
as $j\to\infty$, and similarly
$$
\lim_{j\to+\infty}\int_{\tX}e^{-\f^t_j}\mu_j=\int_{\tX}e^{-\tf^t}\widetilde{\mu}_0.
$$
Since both $\f^t_j$ and $\tf^t$ are $\om_j$-psh, the concavity of $E_j$ yields
$$
\int_{\tX}(\f^t_j-\tf^t)\left(\om_j+dd^c\f^t_j\right)^n\le E_j(\f^t_j)-E_j(\tf^t)
$$
$$
\le\int_{\tX}(\f^t_j-\tf^t)\left(\om_j+dd^c\tf^t\right)^n. 
$$
We thus see that 
$$
\lim_{j\to+\infty}E_j(\f^t_j)=E_{\tom_0}(\tf^t)=E(\f^t),
$$ 
which proves that the first assertion, as well as the convergence of 
$$
E_j^*\left(\left(\om_j+dd^c\tf^t_j\right)^n\right)=E_j(\f^t_j)-\int_{\tX}\f^t_j\left(\om_j+dd^c\tf^t_j\right)^n
$$
to $E^*\left(\MA(\f^t)\right)$. If we set 
$$
f_j:=\frac{V_j^{-1}\left(\om_j+dd^c\tf^t_j\right)^n}{\mu_j}
$$
and 
$$
f:=\frac{V^{-1}\left(\tom_0+dd^c\tf^t\right)^n}{\widetilde{\mu}_0},
$$
it remains to show that 
$$
H_{\mu_j}\left(V_j^{-1}\left(\om_j+dd^c\tf^t_j\right)^n\right)=\int_{\tX}(f_j\log f_j)\mu_j
$$
converges to $H(\MA(\f^t))=\int_{\tX}(f\log f)\widetilde{\mu}_0$. But since $f_j\log f_j$ is uniformly bounded and converges pointwise to $f\log f$ on $\tX_0$, this follows again from the $L^p$ convergence $\mu_j\to\tmu_0$.
\end{proof}

\begin{proof}[Proof of Proposition \ref{prop:monotoneflow}] 
We perform the following standard computation:
$$
\frac{d}{dt}\left(H_{\mu_j}-E_j^*\right)\left(V_j^{-1}\left(\om_j+dd^c\f^t_j\right)^n\right)
$$
$$
=n V_j^{-1}\int_{\tX} \left(\f^j_t+\log\left(\frac{V_j^{-1}\left(\om_j+dd^c\f^j_t\right)^n}{\mu_j}\right)\right)dd^c\dot\f_j^t\wedge\left(\om_j+dd^c\f_j^t\right)^{n-1}
$$
$$
=n V_j^{-1}\int\dot\f_j^t\,dd^c\log\left(\frac{V_j^{-1}\left(\om_j+dd^c\f^j_t\right)^n}{\mu_j(t)}\right)\wedge(\om_j+dd^c\f^j_t)^{n-1}
$$
$$
=-n V_j^{-1}\int d\dot\f_t^j\wedge d^c\dot\f^t_j\wedge(\om_j+dd^c\f^j_t)^{n-1}\le 0
$$
using (\ref{equ:krfj}). 
By Lemma \ref{lem:convep} it follows that $\mab(\om(t))$ is non-increasing along the flow. 

Similarly we compute
$$
\frac{d}{dt}\left(\log\left(\int_{\tX}e^{-\f^t_j}\mu_j\right)+E_j(\f^t_j)\right)
$$
$$
=-\int_{\tX}\dot\f^t_j\,\mu_j(t)+V_j^{-1}\int_{\tX}\dot\f^t_j\left(\om_j+dd^c\f^j_t\right)^n
$$
$$
=H_{\mu_j(t)}\left(V_j^{-1}\left(\om_j+dd^c\f^t_j\right)^n\right)+H_{V_j^{-1}\left(\om_j+dd^c\f^t_j\right)^n}\left(\mu_j(t)\right),
$$
using again (\ref{equ:krfj}). By Pinsker's inequality (see Proposition \ref{prop:ent}), it follows that
$$
\left(\log\left(\int_{\tX}e^{-\f^{t'}_j}\mu_j\right)+E_j(\f^{t'}_j)\right)-\left(\log\left(\int_{\tX}e^{-\f^t_j}\mu_j\right)+E_j(\f^t_j)\right)
$$
$$
\ge\int_t^{t'}\|V_j^{-1}\left(\om_j+dd^c\f^s_j\right)^n-\mu_j(s)\|^2 ds.
$$
By Lemma \ref{lem:convep}, the left-hand side converges to $-\din(\om(t'))+\din(\om(t))$ as $j\to\infty$. On the other hand  
$$
\liminf_{j\to+\infty}\|V_j^{-1}\left(\om_j+dd^c\f^s_j\right)^n-\mu_j(s)\|\ge\|\MA(\f^s)-\mu(s)\|=\|V^{-1}\om(s)^n-\mu_{\om(s)}\|
$$
by lower semicontinuity of the total variation with respect to weak convergence, and we get the desired result thanks to Fatou's lemma. 
\end{proof}

\subsection{Proof of Theorem \ref{thm:convKRF}}
By Proposition \ref{prop:monotoneflow}, $\mab(\om(t))$ is bounded above for, say, $t\ge 1$ . Thanks to Lemma \ref{lem:cvdin}, we are thus reduced to showing that 
$$
\lim_{t\to+\infty}\din(\om(t))=\inf_{\cT^1(X,D)}\din.
$$
Since $\din(\om(t))$ is bounded below, Proposition \ref{prop:monotoneflow} yields the existence of a sequence $t_j\to+\infty$ such that $\|V^{-1}\om(t_j)^n-\mu_{\om(t_j)}\|\to 0$ as $j\to\infty$. Since $\om(t)$ stays in a strongly compact set, we may assume upon passing to a subsequence that $\om(t_j)$ converges strongly to some $\om_\infty\in\cT^1(X,D)$. By Proposition \ref{prop:conven} we have $\om(t_j)^n\to\om_\infty^n$ strongly. The same argument as in the proof of Lemma \ref{lem:lsc} (relying on Proposition \ref{prop:hold}) shows that $\mu_{\om(t_j)}\to\mu_{\om_\infty}$ weakly. We conclude that $V^{-1}\om_\infty^n=\mu_{\om_\infty}$, and hence $\din(\om_\infty)=\inf_{\cT^1(X,D)}\din$ by Theorem \ref{thm:varKE}. By strong continuty of $\din$ (Lemma \ref{lem:lsc}), it follows that 
$$
\lim_{j\to\infty}\din(\om(t_j))=\din(\om_\infty)=\inf_{\cT^1(X,D)}\din,
$$
which concludes the proof.

\section{Examples} \label{sec:example}

\subsection{Log Fano pairs}
As explained in \cite{GK}, to each orbifold $\fX$ is attached a klt pair $(X,D)$, where the normal variety $X$ has quotient singularities and the boundary $D$ has an irreducible decomposition of the form 
$$
D=\sum_E(1-\tfrac{1}{m_E})E
$$
with $m_E\in\N$. This boundary encodes the ramification of $\fX$ in codimension one, and $\fX$ is uniquely determined by the pair $(X,D)$. If $\fX$ is a Fano orbifold then $(X,D)$ is a log Fano pair. A K\"ahler-Einstein metric $\om$ for $(X,D)$ is then smooth in the orbifold sense. 

A related class of log Fano pairs arises by taking quotients of Fano varieties. More specifically, let $Y$ be a $\Q$-Fano variety with log terminal singularities, let $G$ be a finite group of automorphisms of $Z$, and set $X:=Y/G$. Then $p:Y\to X$ is a ramified Galois cover, and there exists a unique effective $\Q$-divisor $D$ supported on the ramification locus of $X$ such that $K_Z=p^*(K_X+D)$. This shows that $(X,D)$ has klt singularities and $-(K_X+D)$ is ample, so that $(X,D)$ is a log Fano pair. When $Z$ is non-singular this is a special case of the previous examples, with $\fX:=[Z/G]$. 

Note that $\cT(X,D)\simeq\cT(Z)^G$, and in particular K\"ahler-Einstein metrics on $(X,D)$ correspond precisely to $G$-invariant K\"ahler-Einstein metrics on $Z$.

\subsection{Properness of the Mabuchi functional}
Inspired by a nice construction of \cite{AGP06}, we prove a criterion that produces a rather broad class of log Fano pairs having a proper Mabuchi functional. 
 
\begin{thm} \label{thm:exaFanosing}
Let $X$ be a $\Q$-Fano variety with log terminal singularities, and let $D$ be an effective $\Q$-Cartier divisor satisfying
\begin{itemize}
\item[(i)] $D\sim_\Q -K_X$, 
\item[(ii)] $(X,D)$ is klt,
\end{itemize}
so that $(X,(1-\la)D)$ is in particular a log Fano pair for every (rational) $\la\in]0,1[$. If the Ding functional (or, equivalently, the Mabuchi functional) of $X$ is bounded below (in particular, if $X$ admits a K\"ahler-Einstein metric), then the Ding and Mabuchi functionals of $(X,(1-\la)D)$ are coercive for all rational numbers $\la\in]0,1[$. 
\end{thm}

\begin{rem} It is interesting to compare this result with \cite[Theorem 7]{Ber10}, which deals with the case where $X$ is a (non-singular) Fano manifold and $D$ is reduced, smooth and irreducible (so that $(X,D)$ is merely lc in that case). Without any further assumption on $X$, it is then proved that $\a(X,(1-\la)D)\to 1$ as $\la\to 0$, which implies in particular that the Mabuchi functional of $(X,(1-\la)D)$ is coercive for $0<\la\ll 1$. As a consequence, $(X,(1-\la)\D)$ admits a unique K\"ahler-Einstein metric for $0<\la\ll 1$, which is further known to have cone singularities of cone angle $2\pi\la$ along $D$ by \cite{JMR}. 

Note on the other hand that the irreducibility of $D$ is crucial in this result: for $X=\PP^1$ and $D=[0]+[\infty]$, the Mabuchi functional of $(X,(1-\la)D)$ cannot proper even for $0<\la\ll 1$, since $\Aut^0(X,(1-\la)D)=\Aut^0(X,D)=\C^*$ is not trivial. This also shows that it is not enough to assume $(X,D)$ lc in (ii)  of Theorem \ref{thm:exaFanosing}. 
\end{rem}

It is shown in \cite[Proposition 2.5]{Lee} that any effective divisor $D$ on $\PP^n$ of degree $d$ such that $(\PP^n,\frac{n+1}{d}D)$ is klt defines a stable point in the projective space $\left|\cO_{\PP^n}(d)\right|$ with respect to the action of the reductive group $\Aut(\PP^n)$. As a consequence of the above result, we get the following generalization of this fact: 

\begin{cor}\label{cor:stab} Let $X$ be a K\"ahler-Einstein Fano manifold (so that $G:=\Aut^0(X)$ is reductive by \cite{Mat}), and let $L$ be an ample $G$-line bundle on $X$ with $c L\sim_\Q -K_X$ for some $c\in\Q_+$. Then every effective divisor $D\sim L$ such that $(X,c D)$ is klt defines a $G$-stable point of $|L|=\PP H^0(X,L)$. 
\end{cor}
\begin{proof} By semicontinuity \cite{Kol97}, $U:=\left\{D\in|L|\mid (X,c D)\text{ klt}\right\}$ is a $G$-invariant Zariski open subset of $|L|$. It is thus enough to show that the stabilizer $G_{D}$ of $D$ in $G$ is finite for all $D\in U$ (compare \cite[Proposition 1.26]{Bri}), which amounts to $G^0_D=\{1\}$ since $G_D$ is an algebraic group. But the Mabuchi functional of $(X, cD)$ is proper by Theorem \ref{thm:exaFanosing}, hence $G^0_{D}=\Aut^0(X,D)=\Aut^0(X,c D)=\{1\}$ by Theorem \ref{thm:uniqueKE}.
\end{proof}

\begin{ex}  
Let $H$ be an irreducible hypersurface of degree $d$ in $X:=\PP^n$, with $n\ge 3$. Assume that $n+2\le d\le 2n+1$ and that the singularities of $H$ are at most log canonical (lc for short). By inversion of adjunction, it follows that the pair $(X,H)$ is lc as well (see \cite[Theorem 7.5]{Kol97}). Since $\frac{n+1}{d}<1$, it follows that $(X,\frac{n+1}{d}H)$ is klt. But we also have $\frac{1}{2}<\frac{n+1}{d}$, thus Theorem \ref{thm:exaFanosing} implies that $(X,\tfrac 1 2 H)$ admits a unique K\"ahler-Einstein metric. Since $H$ has even degree, we can construct a double cover $p:Y\to X$ ramified along $H$, which satisfies $K_Y=p^*\left(K_X+\tfrac 12 H\right)$, and $Y$ is thus a $\Q$-Fano variety with log terminal singularities and a K\"ahler-Einstein metric (invariant under the Galois group of $p$). 

If the singularities of $H$ are for instance at most ordinary double points (\ie locally analytically isomorphic to $\left\{\sum_{i=1}^n z_i^2=0\right\}$, which are lc), then the singularities of the double cover $Y$ are also ordinary double points, and are not quotient singularities since ordinary double points have a trivial local fundamental group in dimension $n\ge 3$. It follows that the K\"ahler-Einstein metric of $Y$ cannot be constructed by orbifold methods. 
\end{ex}

\begin{proof}[Proof of Theorem \ref{thm:exaFanosing}]
Let us first fix some notation. Since $D\sim_\Q-K_X$ we have 
$$
-(K_X+(1-\la)D)\sim_\Q-\la K_X.
$$ 
Let $\phi_0$ be a reference smooth strictly psh metric on $-K_X$, with curvature form $\om_0$ and adapted probability measure $\mu_0$. We use $\phi_\la:=\la\phi_0$ as a reference smooth strictly psh metric on $-(K_X+(1-\la) D)$, with curvature form $\om_\la$ and adapted measure $\mu_\la$. Note that $\f\mapsto\p=\la\f$ sets up an isomorphism $\psh(X,\om_0)\simeq\psh(X,\om_\la)$. Denoting by $E_\la$ the energy functional of $\cE^1(X,\om_\la)$, it is straightforward to check that
\begin{equation}\label{equ:entransfo}
E_\la(\p)=\la E(\f). 
\end{equation}
By Proposition \ref{prop:moser}, we will be done if we can prove that functions $\p\in\cE^1(X,\om_\la)$ satisfy a Moser-Trudinger condition
$$
\|e^{-\p}\|_{L^p(\mu_\la)}\le A\,e^{-E_\la(\p)}
$$
for some $p>1$ and $A>0$ (independent of $\p$). 

Since the Ding functional of $X$ is assumed to be bounded below, we have an estimate
$$
\|e^{-\f}\|_{L^1(\mu_0)}\le A\,e^{-E(\f)}
$$
for all $\f\in\cE^1(X,\om_0)$. By (\ref{equ:entransfo}), it follows that
\begin{equation}\label{equ:moserla}
\|e^{-\p}\|_{L^{\la^{-1}}(\mu_0)}\le A\,e^{-E_\la(\p)}
\end{equation}
for all $\p\in\cE^1(X,\om_\la)$. On the other hand, it is immediate to check from the definition that $\mu_\la=e^{-(1-\la)\rho}\mu_0$ for some quasi-psh function $\rho$ which locally satisfies $\rho=\log|f|^2+O(1)$, where $f$ is a local equation of $D$. Since $(X,D)$ is klt, we thus have $e^{-\rho}\in L^{q}(\mu_0)$ for some $q>1$. 

Now pick $\de\in]q^{-1}(1-\la),1-\la[$. By H\"older's inequality we have
$$
\int_Xe^{-(1-\de)\la^{-1}\p}\mu_\la=\int_Xe^{-(1-\de)\la^{-1}\p-(1-\la)\rho}\mu_0\le\left(\int_X e^{-\la^{-1}\p}\mu_0\right)^{1-\de}\left(\int_Xe^{-\de^{-1}(1-\la)\rho}\mu_0\right)^\de.
$$
Since $\de^{-1}(1-\la)<q$ and $e^{-\rho}\in L^q(\mu_0)$, we have $\int_Xe^{-\de^{-1}(1-\la)\rho}\mu_0<+\infty$. We thus get $C>0$ and
$$
p:=(1-\de)\la^{-1}>1.
$$ 
such that
$$
\|e^{-\p}\|_{L^{p}(\mu_\la)}\le C\,\|e^{-\p}\|_{L^{\la^{-1}}(\mu_0)}
$$
for all $\p\in\cE^1(X,\om_\la)$. Combining this with (\ref{equ:moserla}) yields the desired Moser-Trudinger condition. 
\end{proof}

\section{Appendix A: an Izumi-type estimate} 

Let $X$ be a normal complex space with a given point $x\in X$ and let $\f$ be a psh function on $X$. Choose local generators $(f_i)$ of the maximal ideal $\fm_x$ of $\cO_{X,x}$ and define the \emph{slope} of $\f$ at $x$ by 
\begin{equation}\label{equ:slope}
s(\f,x):=\sup\left\{s\ge 0\mid\f\le s\log\sum_i|f_i|+O(1)\right\}\in[0,+\infty[
\end{equation}
Since $\log\sum_i|f_i|$ only depends on the choice of generators up to a bounded term, it is clear that $s(\f,x)$ is independent of the choice of $(f_i)$. For $f\in\cO_{X,x}$ we have
$$
s(\log|f|,x)=\overline{\ord}_x(f):=\lim_{m\to\infty}\tfrac 1m\ord_x(f^m),
$$
with
$$
\ord_x(f):=\max\left\{k\in\N\mid f\in\fm_x^k\right\}.
$$
\begin{rem} By \cite[p.50, Corollaire 6.6]{DemSMF} the non-decreasing function 
$$
\chi(t):=\sup_{\{\sum_i|f_i|<e^t\}}\f
$$ 
is convex (generalized three-circle theorem), and we have $s(\f,x)=\lim_{t\to-\infty}\chi(t)/t$. This implies in particular that the supremum in (\ref{equ:slope}) is attained. 
\end{rem}

Izumi's theorem \cite{Izu81} states that for every resolution of singularities $\pi:\tX\to X$ and every prime divisor $E\subset\tX$ lying above $x\in X$, there exists a constant $C>0$ such that 
$$
\ord_E(f\circ\pi)\le C\,\overline{\ord}_x(f)
$$
for all $f\in\cO_{X,x}$. Our goal here is to prove the following extension of this result to psh functions: 

\begin{thm}\label{thm:izumi} 
Let $\pi:\tX\to X$ be any resolution of singularities and let $E\subset\tX$ be a prime divisor above $x\in X$. Then there exists $C>0$ such that 
$$
\nu(\f\circ\pi,E)\le C s(\f,x)
$$ 
for all psh functions $\f$ on $X$. 
\end{thm}

Here 
$$
\nu(\f\circ\pi,E)=\min_{p\in E}\nu(\f\circ\pi,p)
$$
is the generic Lelong number of $\f\circ\pi$ along $E$. Note that $\ord_E(f\circ\pi)=\nu(\f\circ\pi,E)$ with $\f=\log|f|$. 

\begin{cor}\label{cor:izumi} 
If $\f$ is a psh function with $s(\f,x)=0$ for some $x\in X$, then $\nu(\f\circ\pi,p)=0$ for every resolution of singularities and every $p\in\pi^{-1}(x)$. 
\end{cor}

\begin{proof}
Let $b:X' \rightarrow \tX$ be the blow up of $\tX$ at point $p \in \tX$. Then
$\pi'=\pi \circ b:X' \to X$ is yet another resolution of singularities. Set $E=\b^{-1}(p)$. This is a prime divisor to which we can apply Theorem \ref{thm:izumi} . The conclusion follows then by recalling the 
following classical interpretation of Lelong number:
$
\nu(\f \circ \pi \circ b,E)=\nu(\f \circ \pi, p).
$
\end{proof}

\begin{proof}[Proof of Theorem \ref{thm:izumi}] 
By Hironaka's theorem we may assume that $\pi:\tX\to X$ dominates the blow-up of $X$ at $x$, so that the scheme-theoretic fiber $\pi^{-1}(x)$ is an effective divisor $\sum_i a_i E_i$. Note that $\sum_i E_i$ is connected by Zariski's "main theorem". 

Set $b_i:=\nu(\pi^*\f,E_i)$. Using the Siu decomposition of the positive current $T:=dd^c\pi^*\f$ we may write $T=R+B$ where $B=\sum_ib_i E_i$ is an effective $\R$-divisor and $R$ is a positive current such that $\nu(R,E_i)=0$ for all $i$. We first claim that 
\begin{equation}\label{equ:slopefrac}
s(\f,x)=\min_i\frac{b_i}{a_i}.
\end{equation}
Indeed, if we write $\fm_x=(f_i)$ as above then $\pi^*\log\sum_i|f_i|$ has analytic singularities described by the divisor $\pi^{-1}(x)$, \ie locally on $X$ we have
$$
\pi^*\log\sum_i|f_i|=\sum_i a_i\log|z_i|+O(1)
$$
where $z_i$ is a local equation of $E_i$. We thus see that 
$$
\f\le s\log\sum_i|f_i|+O(1)\Longleftrightarrow \pi^*\f\le\sum_i s a_i\log|z_i|+O(1)
$$ 
locally on $\tX$, and the positivity of $R=T-B$ shows that this holds iff $b_i=\nu(\pi^*\f,E_i)\ge s a_i$ for all $i$, hence the claim.

In view of (\ref{equ:slopefrac}), the desired statement amounts to an estimate $\max_i b_i\le C\min_i b_i$ for some $C>0$ independent of $\f$. Thanks to Lemma \ref{lem:izumi} below, this will hold if we can show that $-B|_{E_i}$ is pseudoeffective for all $i$. Since the restriction to each $E_i$ of the cohomology class of $T=\pi^*dd^c\f$ is trivial, we are reduced to showing that $\{R\}|_{E_i}$ is pseudoeffective. Since $\nu(R,E_i)=0$, this follows from Demailly's regularization theorem. Let us recall the standard argument: by\cite{Dem92}, after perhaps shrinking $X$ slightly about $0$ we may write $R$ as a weak limit of closed positive $(1,1)$-currents $R_k$ with analytic singularities such that $\{R_k\}=\{R\}$, $R_k\ge-\e_k\om$ for some $\e_k\to 0$ and $R_k$ is less singular than $R$. In particular we have $\nu(R_k,E_i)=0$ for all $i$, which means that the local potentials of $R_k$ are not entirely singular along $E_i$, so that $R_k|_{E_i}$ is a well-defined closed $(1,1)$-current. We thus see that $\left(\{R\}+\e_k\{\om\}\right)|_{E_i}$ is pseudoeffective for all $k$, and the claim follows. 
\end{proof}

\begin{lem}\label{lem:izumi} Let $E=\sum_i E_i$ be a reduced compact connected divisor on a K\"ahler manifold $M$. Let $B=\sum b_i E_i$ be an effective $\R$-divisor supported in $E$, and assume that $-B|_{E_i}$ is pseudoeffective for all $i$. Then there exists a constant $C>0$ only depending on $E$ such that $\max_i b_i\le C\min_i b_i$. 
\end{lem}
The proof to follow is directly inspired from \cite[\S 6.1]{BFJ12}. 
\begin{proof} Let $\om$ be a K\"ahler form on $M$. Thanks to the connectedness of $E$, we may index the $E_i$ such that $B=\sum_{i=1}^N b_i E_i$ with $b_1=\min_i b_i$, $b_r=\max_i b_i$ for $1\le r\le N$, and $E_i\cap E_{i+1}\neq\emptyset$ for all $i=1,...,r-1$. For each $i$ we have
$$
(-D|_{E_i})\cdot(\om|_{E_i})^{n-2}=-\sum_j b_j c_{i,j}\ge 0,
$$
with 
$$
c_{i,j}:=\left(E_i\cdot E_j\cdot\om^{n-2}\right),
$$
hence 
\begin{equation}\label{equ:izumi}
\sum_{j\neq i} b_j c_{i,j}\le b_i \left| c_{i,i}\right|.
\end{equation}
Now $c_{i,j}\ge 0$ if $j\neq i$, and $c_{i,i+1}>0$ for all $i$ since $E_i$ meets $E_{i+1}$. It follows that 
$$
b_{i+1}\le\frac{\left|c_{i,i}\right|}{c_{i,i+1}}b_i
$$
for all $i$, hence 
$\max_i b_i=b_r\le C b_1=\min_i b_i$
with $C:=\prod_{i=1}^{r-1}\frac{\left|c_{i,i}\right|}{c_{i,i+1}}$ 
\end{proof} 

\begin{rem} Besides the slope $s(\f,x)$ considered above, Demailly introduced in \cite{DemSMF} a different generalization of Lelong numbers on normal complex spaces, defined as the intersection multiplicity
$$
\nu(\f,x):=(dd^c\f)\wedge(dd^c\log\sum_i|f_i|)^{n-1}\left(\{x\}\right), 
$$
where $(f_i)$ are generators of $\fm_x$, the definition being independent of that choice. When $\fa=(g_1,...,g_r)$ is an $\fm_x$-primary ideal and $\f=\log\sum_i|g_i|$ then $\nu(\f,x)$ computes the mixed (Hilbert-Samuel) multiplicity $\langle\fa,\fm_x,...,\fm_x\rangle$. In particular for $\f=\p$ we have $\nu(\p,x)=m(X,x)$, the multiplicity of $X$ at $x$. By Demailly's comparison theorem we have
$$
\nu(\f,x)\ge s(\f,x) m(X,x),
$$ 
and the inequality is strict in general. Using the notation of the proof of Theorem \ref{thm:izumi} and recalling that $-\pi^{-1}(x)$ is $\pi$-nef, we conjecture by analogy with the algebraic case that 
$$
\nu(\f,x)=(B\cdot (-\pi^{-1}(x))^{n-1}).
$$ 
By Theorem \ref{thm:izumi} this would imply in particular that conversely $\nu(\f,x)\le C s(\f,x)$ for some $C>0$ independent of $\f$. 
\end{rem}

\section{Appendix B:  Laplacian estimate}\label{sec:MAreg}

The goal of this section is to present an explicit version of the main result of \cite{Pau}, in order to make it suitable to our purpose. 

In what follows $(X,\om)$ denotes a compact K\"ahler manifold, 
$\D=\tr_\om dd^c$ is the (analysts') Laplace operator with respect to the reference K\"ahler form $\om$,
and $\theta \geq 0$ is a semi-positive closed $(1,1)$-form such that $\int_X \theta^n>0$,
where $n=\dim_\C X$. We let $\Amp(\theta)$ denote the ample locus of
(the cohomology class of) $\theta$.

\begin{thm}\label{thm:paun} Let $\mu$ be a positive measure on $X$ of the form $\mu=e^{\p^+-\p^-}dV$ with $\p^\pm$ quasi-psh and $e^{-\p^-}\in L^p$ for some $p>1$. Assume that $\f$ is a bounded $\theta$-psh function such that $(\theta+dd^c\f)^n=\mu$. Then we have $\D\f=O(e^{-\psi^-})$ locally in $\Amp(\theta)$. 

More precisely, assume given a constant $C>0$ such that 
\begin{itemize}
\item[(i)] $dd^c\psi^+\ge -C\,\om$ and $\sup_X\psi^+\le C$.
\item[(ii)] $dd^c\psi^-\ge -C\,\om$ and $\|e^{-\psi^-}\|_{L^p}\le C$. 
\end{itemize}
Let also $U\Subset\Amp(\theta)$ be a relatively compact open subset. Then there exists $A>0$ only depending on $\theta$, $p$, $C$ and $U$ such that 
$$
0\le\theta+dd^c\f\le A\,e^{-\psi^-}\,\om
$$
on $U$.
\end{thm}

This result recovers in particular \cite[Theorem 7, p.398]{Yau}. 

\begin{proof} We may of course assume that $\f$ is normalized. During the proof $A, A_1,...$ will denote positive constants that may vary from line to line, but are \emph{under control} in the sense that they only depend on $\theta$, $p$, $C$ and $U$. Since $U$ is contained in $\Amp(\theta)$, we may choose a Zariski open set $\Omega\supset U$ and a $\theta$-psh function $\psi$ such that $(\theta+dd^c\psi)|_\Omega$ is the restriction of a K\"ahler form $\tom$ on a higher compactification $\tX$ of $\Omega$, so that 
$$
\tom\ge\de\om \text{ on } \Omega \text{ for some } \de>0
\text{ and }
\p \to -\infty \text{ near } \partial \Omega.
$$

The proof of Theorem \ref{thm:paun} is divided in two steps. In the first and main one,  an \emph{a priori} estimate for smooth solutions of non-degenerate perturbations of the equation is established. In the second step we conclude using a regularization argument.\\

\noindent {\bf Step 1: A priori estimates.} 
For $0<\e\le 1$ we set $\om_\e:=\tom+\e\,\om$, viewed as a K\"ahler form on $\Omega$. Note that $\om_\e\ge\de\,\om$, so that 
\begin{equation}\label{equ:tr}
\tr_{\om_\e}(\a)\le\de^{-1}\tr_\om(\a)
\end{equation}
for every positive $(1,1)$-form $\a$. Assume that $\psi^+$ and $\psi^-$ are \emph{smooth} functions satisfying (i) and (ii) of Theorem \ref{thm:paun}, and assume given a smooth normalized $\theta_\e$-psh function $\f_\e$ such that 
\begin{equation}\label{equ:tep}
(\theta+\e\om+dd^c\f_\e)^n=e^{\psi^+-\psi^-}\,dV.
\end{equation}
The goal of Step 1 is to establish that $|\D\f_\e|\le A\,e^{-\psi^-}$ on $U$ with $A>0$ under control. Since we have $\om_\e\le A\om$ over $U$ with $A$ under control, it will be enough to prove that 
$$
\om'_\e:=\theta+\e\om+dd^c\f_\e
$$
satisfies $\tr_{\om_\e}(\om'_\e)\le A\,e^{-\psi^-}$ on $U$. 

We first recall the Laplacian inequality obtained in \cite[pp.98-99]{Siu}: if $\tau,\tau'$ are two K\"ahler forms on a complex manifold, then there exists a constant $B>0$ only depending on a lower bound for the holomorphic bisectional curvature of $\tau$ such that
\begin{equation}\label{equ:lap}
\D_{\tau'}\log\tr_{\tau}(\tau')\ge-\frac{\tr_\tau\Ric(\tau')}{\tr_\tau(\tau')}-B\tr_{\tau'}(\tau).
\end{equation}
We remark that Siu's argument uses the fact that $\tau$ and $\tau'$ are $dd^c$-cohomologous. But the general case is valid as well since Siu's computations are purely local and any K\"ahler form is even locally $dd^c$-exact. This being said, let us apply this inequality to the two K\"ahler forms $\om_\e$ and $\om'_\e$ on $\Omega$.

Since $\tom$ extends to a K\"ahler form on a higher compactification $\tX$ of $\Omega$, the holomorphic bisectional curvature of $\om_\e=\tom+\e\,\om$ is obviously bounded over $\Omega$ by a constant $B>0$ under control, and (\ref{equ:lap}) yields
\begin{equation}\label{equ:lap1}
\D_{\om'_\e}\log\tr_{\om_\e}(\om'_\e)\ge-\frac{\tr_{\om_\e}\Ric(\om'_\e)}{\tr_{\om_\e}(\om'_\e)}-B\tr_{\om'_\e}(\om_\e).
\end{equation}
On the other hand, applying $dd^c\log$ to $(\om_\e')^n=e^{\p^+-\p^-}\om^n$ yields 
$$
-\Ric(\om'_\e)=-\Ric(\om)+dd^c\psi^+-dd^c\psi^-\ge -A\om-dd^c\psi^-
$$
where $A$ is under control thanks to (i). Using $\tr_{\om_\e}(\om)\le n\de^{-1}$ and the trivial inequality 
\begin{equation}\label{equ:triv}
n\le\tr_{\om_\e}(\om'_\e)\tr_{\om'_\e}(\om_\e)
\end{equation} 
we thus infer from (\ref{equ:lap1})
\begin{equation}\label{equ:lapl2}
\D_{\om'_\e}\log\tr_{\om_\e}(\om'_\e)\ge-\frac{\D_{\om_\e}\psi^-}{\tr_{\om_\e}(\om'_\e)}-A\tr_{\om'_\e}(\om_\e).
\end{equation}
with $A$ under control. 

We next argue along the lines of \cite[Lemma 3.2]{Pau} to take care of the term $\D_{\om_\e}\psi^-$. By (ii) we have $A\om_\e+dd^c\psi^-\ge 0$ with $A$ under control. Applying $\tr_{\om_\e}$ to the trivial inequality 
$$
0\le A\om_\e+dd^c\psi^-\le\tr_{\om'_\e}(A\om_\e+dd^c\psi^-)\om'_\e
$$
yields 
$$
0\le An+\D_{\om_\e}\psi^-\le (A\tr_{\om'_\e}(\om_\e)+\D_{\om'_\e}\psi^-)\tr_{\om_\e}(\om'_\e). 
$$
Plugging this into (\ref{equ:lapl2}) and using again (\ref{equ:triv}) we thus obtain
\begin{equation}\label{equ:lapl3}
\D_{\om'_\e}\left(\log\tr_{\om_\e}(\om'_\e)+\psi^-\right)\ge - A\tr_{\om'_\e}(\om_\e).
\end{equation}
where $A$ is under control. Now set
$$
\rho_\e:=\f_\e-\p,
$$ 
so that $\om_\e'=\om_\e+dd^c\rho_\e$. We then have $n=\tr_{\om'_\e}(\om_\e)+\D_{\om'_\e}\rho_\e$, and we finally deduce from (\ref{equ:lapl3}) that
\begin{equation}\label{equ:lapl4}
\D_{\om'_\e}\left(\log\tr_{\om_\e}(\om'_\e)+\psi^--A_1\rho_\e\right)\ge\tr_{\om'_\e}(\om_\e)-A_2
\end{equation}
on $\Omega$, with $A_1,A_2$ under control. 

We are now in a position to apply the maximum principle. On the one hand, $\rho_\e=\f_\e-\p$ tends to $+\infty$ near $\partial\Omega$. On the other hand, $\tr_{\om_\e}(\om'_\e)\le\de^{-1}\tr_\om(\om'_\e)$ is bounded above on $\Omega$ since $\om'_\e$ is smooth over $X$. The function
$$
H:=\log\tr_{\om_\e}(\om'_\e)+\psi^--A_1\rho_\e
$$
therefore achieves its maximum at some $x_0\in\Omega$, and (\ref{equ:lapl4}) yields $\tr_{\om'_\e}(\om_\e)(x_0)\le A_2$. On the other hand, trivial eigenvalue considerations show that
$$
\tr_{\tau_1}(\tau_2)\le n\left(\tau_2^n/\tau_1^n\right)\tr_{\tau_2}(\tau_1)^{n-1}
$$
for any two K\"ahler forms $\tau_1,\tau_2$, whence 
$$
\log\tr_{\om_\e}(\om'_\e)\le\psi^+-\psi^-+\log\left(\frac{\om^n}{\om_\e^n}\right)+(n-1)\log\tr_{\om'_\e}(\om_\e)+\log n
$$
by (\ref{equ:tep}). Using $\om\le\de^{-1}\om_\e$ it follows that
$$
H\le A_3\log\tr_{\om'_\e}(\om_\e)+A_4-A_1\rho_\e
$$
where $A_3,A_4$ are under control, and we obtain 
$$
\sup_\Omega H=H(x_0)\le A_5-A_1\inf_\Omega\rho_\e\le A_5-A_1\inf_X\f_\e
$$
with $A_5$ under control, since $\rho_\e=\f_\e-\psi$ and $\psi\le 0$. By the $L^\infty$-estimate provided by \cite{EGZ1}, we now obtain 
$$
\log\tr_{\om_\e}(\om'_\e)+\psi^--A_1\rho_\e=H\le A
$$
on $\Omega$ for some constant $A$ under control. Since $\f_\e$ is normalized we conversely have $\rho_\e\le-\psi\le A_6$ over $U\Subset\Omega$, and we finally infer as desired $\tr_{\om_\e}(\om'_\e)\le A e^{-\psi^-}$ on $U$.\\

\noindent{\bf Step 2: Regularization.} We now consider the set-up of Theorem \ref{thm:paun}. By Demailly's regularization theorem \cite{Dem92}, there exist two decreasing sequences of smooth functions $\psi^\pm_j$ such that 
\begin{itemize}
\item $\lim_{j\to\infty}\psi^\pm_j=\psi^\pm$ on $X$. 
\item $dd^c\psi^\pm_j\ge-A\,\om$ for some $A>0$ under control.
\end{itemize}
In fact, the constant $A>0$ depends in principle on the Lelong numbers of the quasi-psh functions $\psi^\pm$ according to Demailly's result, but these Lelong numbers  can be uniformly bounded in terms of the lower bound $-C\om$ for $dd^c\psi^\pm$ by a standard argument, see for instance \cite[Lemma 2.5]{Bou02}. 

For each $0<\e\le 1$ the closed $(1,1)$-form $\theta+\e\,\om$ is K\"ahler, and Yau's theorem \cite{Yau} yields smooth normalized $\theta_\e$-psh functions $\f_{\e,j}$ such that 
$$
(\theta_\e+dd^c\f_{\e,j})^n=e^{\psi^+_j-\psi^-_j+c_{\e,j}}\om^n,
$$
where $c_{\e,j}\in\R$ is a normalizing constant. Since  $e^{\psi^+_j-\psi^-_j}\le e^{C-\psi^-}$ is uniformly bounded in $L^p$, $c_{\e,j}$ is under control and Step 1 of the proof shows that 
\begin{equation}\label{equ:boundlap}
|\D\f_{\e,j}|\le A\,e^{-\psi^-_j}
\end{equation}
over $U$, with $A>0$ under control. 

Now for each fixed $j$ it follows from \cite[Lemma 5.3]{BEGZ} that $\f_{\e,j}$ converges weakly as $\e\to 0$ to the normalized solution $\f_j$ of 
$$
(\theta+dd^c\f_j)^n=e^{\psi^+_j-\psi^-_j+c_j}\om^n,
$$
which therefore satisfies as well $|\D\f_j|\le A\,e^{-\psi^-_j}$ on $U$. But we also have $e^{\psi^+_j-\psi^-_j}\to e^{\psi^+-\psi^-}$ in $L^p$ by dominated convergence, and it follows that $\f_j\to\f$ weakly on $X$ by \cite[Theorem A]{EGZ1}, which concludes the proof of Theorem \ref{thm:paun}.  
\end{proof}

\section{Appendix C: a version of Berndtsson's convexity theorem}

The goal of this section is to extract from \cite{Bern11} the proof the following result. 

\begin{thm}\label{thm:berndt} Let $X$ be a compact K\"ahler manifold and $L$ a line bundle on $X$ such that:
\begin{itemize}
\item[(i)] $h^0(X,K_X+L)=1$ and $h^1(X,K_X+L)=0$; 
\item[(ii)] $L=M+\D$ where $M$ is a semipositive $\Q$-line bundle, $\D=\sum_i a_i D_i$ is an effective $\Q$-divisor with SNC support and $a_i\in (0,1)$. 
\end{itemize}
Set $S:=\left\{t\in\C\mid 0<\Re t<1\right\}$ and consider a psh metric $\phi$ on the pull-back of $L$ to $X\times S$ of the form  $\phi=\tau+\phi_\D$ where 
\begin{itemize}
\item[(iii)] $\tau$ is a bounded psh metric on the pull-back of $M$ to $X\times S$, with $t\mapsto\tau_t$ only depending on $\Re t$ and Lipschitz continuous; 
\item[(iv)] $\phi_\D=\sum_i a_i\log|s_i|^2$ with $s_i$ the canonical section of $\cO(D_i)$, so that $dd^c\phi_\D=[\D]$. 
\end{itemize}
For each generator $u$ of $H^0(X,K_X+L)$, viewed as an $L$-valued holomorphic $n$-form on $X$, the function
$$
L(t)=-\log\|u\|^2_{\phi_t}=-\log\int_X i^{n^2} u\wedge\bar u\,e^{-\phi_t}
$$ 
is then convex on $(0,1)$. If it is further affine, then there exists a holomorphic vector field $V$ on $\{u\ne 0\}\subset X$ such that 
$$
\left(\mathcal L_V+\frac{\partial}{\partial t}\right)dd^c_z\phi_t=0
$$ 
on $X\times S$, with $\cL_V$ the Lie derivative along $V$. 
\end{thm}
The situation here is a slight variant of \cite[\S 6.2]{Bern11}, which corresponds to the case where $u$ is nowhere zero (and hence $L=-K_X$). The arguments given in that part of the paper are rather brief, and a more precise exposition of the proof is presented in \cite[Appendix 1]{CDS3}. However, the latter still suffers from some minor oversights, having to do with the negative part of the curvature in the regularization and the possibly non-uniform convergence of the curvature formula. We therefore take the opportunity to present here a proof with full details. We are very grateful to Bo Berndtsson who kindly answered our questions on his proof and carefully checked our arguments. 

In what follows we fix a reference K\"ahler metric $\om$ on $X$.\\

\noindent {\bf Step 0: Preliminary facts}. 
Assume for the moment that $\phi$ is a fixed smooth metric on $L$. The $(1,0)$-part of the induced Chern connection is given by 
\begin{equation}\label{equ:localdphi}
\partial^{\phi}=\partial-\partial\phi\wedge\bullet
\end{equation}
in any local trivialization of $L$. It is related to the adjoint $\dbar^*_\phi$ of $\dbar$ by the K\"ahler commutation identity $i\partial^\phi=[\dbar^*_\phi,\om\wedge\bullet]$. For an $L$-valued $(p,0)$-form $v$, this becomes
\begin{equation}\label{equ:kahler}
i\partial^{\phi} v=\dbar^*_{\phi}(\om\wedge v),
\end{equation}
which shows in particular that the image of $\partial^{\phi}$ on $(p,0)$-forms is orthogonal to the kernel of $\dbar$. For $p=n-1$, $v\mapsto\om\wedge v$ is a pointwise isometry between $L$-valued $(n-1,0)$ and $(n,1)$-forms,  and the Hodge star operator satisfies $\star(\om\wedge v)=i^{(n-1)^2}v$. In particular, 
\begin{equation}\label{equ:scalint}
\langle\om\wedge v,\alpha\rangle_{L^2(\phi)}=i^{(n-1)^2}\int_X v\wedge\overline\a.
\end{equation}
for any $L$-valued $(n,1)$-form $\a$. 

\begin{lem}\label{lem:dbar} For each $L$-valued $(n,0)$-form $\eta$ on $X$, there exists a unique $L$-valued $(n-1,0)$-form $v$ such that
\begin{itemize}
\item[(i)] $\partial^{\phi}v=P\eta$, the projection of $\eta$ orthogonal to the kernel of $\dbar$; 
\item[(ii)] $\om\wedge\dbar v=0$. 
\end{itemize}
\end{lem}
\begin{proof} The image of $\dbar$ is closed, since it has finite codimension in $\Ker\dbar$. As a result, $P\eta\in(\Ker\dbar)^\perp=\Im\dbar^*_\phi$ may be uniquely written as $P\eta=\dbar^*_\phi\b$ for an $L$-valued $(n,1)$-form $\b\in(\Ker\dbar^*_\phi)^\perp=\Im\dbar$, and $\b=\om\wedge v$ for a unique $L$-valued $(n-1,0)$-form $v$. 

Since we are assuming that $H^{n,1}(X,\C)=H^1(X,K_X+L)=0$, $\b$ above is in fact unique in $\Ker\dbar$, which concludes the proof. 
\end{proof}

\begin{rem}\label{rem:ortho} For later use, note that any $L$-valued $(n-1,0)$-form $v$ for which (ii) holds satisfies
$$
\int_X v\wedge\overline{P\a}\,e^{-\phi}=0
$$
for all $L$-valued $(n,1)$-form $\a$, by (\ref{equ:scalint}). 
\end{rem}

\noindent {\bf Step 1: Regularization}.
As in \cite[\S 2.3]{Bern11}, we rely on \cite{BK07} to write the bounded psh metric $\tau$ on the pull-back of $M$ to $X\times S$ as the decreasing limit of a sequence of smooth metrics $\tau^\nu$ over $X\times S_\nu$ for a slightly smaller strip
$$
S_\nu=\{t\in\C\mid \de_\nu<\Re t<1-\de_\nu\}
$$ 
with $\de_\nu\to 0$, such that 
$$
dd^c\tau^\nu\ge-\e_\nu\om
$$
on $X\times S_\nu$ for some sequence $\e_\nu\to 0$. We denote by $dd^c$ the operator on the product; an additional index $z$ or $t$ will indicate partial derivatives. Note that shrinking the time interval is necessary in the regularization process, since we are working over the non-compact product manifold $X\times S$.  

Since $t\mapsto\tau_t$ is Lipschitz continuous and only depends on $\Re t$, we can further arrange that $t\mapsto\tau^\nu_t$ is uniformly Lipschitz continuous and only depends on $\Re t$ (by averaging). 
 
We also introduce a regularization of $\phi_\D$ by setting 
$$
\phi_\D^\nu:=\sum_i a_i\log\left(|s_i|^2+\nu^{-1}e^{\psi_i}\right)
$$
with $\psi_i$ a smooth metric on $\cO(D_i)$. It satisfies: 
\begin{itemize}
\item[(i)] $dd^c\phi_\D^\nu\ge-C\om$ for some uniform constant $C>0$;
\item[(ii)] for each neighborhood $U$ of $\supp\D$, there exists $\e^\nu_U\to 0$ such that $dd^c\phi_\D^\nu\ge-\e^\nu_U\om$ outside $U$. 
\end{itemize}

Setting $\phi^\nu:=\tau^\nu+\phi_\D^\nu$ defines a smooth metric on the pull-back of $L$ to $X\times S$, only depending on $\Re t$, with time derivative $\dot\phi^\nu_t=\dot\tau^\nu_t\in C^\infty(X)$ uniformly bounded and converging a.e. to $\dot\phi_t$.\\

\noindent {\bf Step 2: Hodge theoretic estimates}. 
For each $t,\nu$, we denote by $\|\eta\|_{\phi^\nu_t}$ the $L^2$-norm of an $L$-valued $(p,q)$-form $\eta$ on $X$ with respect to the fixed K\"ahler metric $\om$ and the hermitian metric $\phi^\nu_t$ on $L$. We write $P^\nu_t\eta$ for the projection of $\eta$ orthogonal to the kernel of $\dbar$, and $\partial_z^{\phi^\nu_t}$ for the $(1,0)$-part of the Chern connection associated to $\phi^\nu_t$. As explained in Remark 3.2 and Lemma 6.3 of \cite{Bern11}, the equation in Lemma \ref{lem:dbar} satisfies the following uniform estimate: 

\begin{lem}\label{lem:dbarbis} There exists a constant $C>0$ such that for each $t,\nu$ and each $L$-valued $(n,0)$-form $\eta$, the unique $L$-valued $(n-1,0)$-form $v$ solving
\begin{itemize}
\item[(i)] $\partial_z^{\phi^\nu_t}v=P^\nu_t\eta$;
\item[(ii)] $\om\wedge\dbar_z v=0$. 
\end{itemize}
satisfies $\|v\|_{\phi^\nu_t}\le C\|\eta\|_{\phi^\nu_t}$. 
\end{lem}

We will also rely on the following estimate, which follows from (the proof of) \cite[Lemma 6.5]{Bern11}. 

\begin{lem}\label{lem:neigh} For each $\de>0$, there exists a neighborhood $U_\de\subset X$ of $\supp\D$ such that 
$$
\int_{U_\de}|v|^2_{\phi^\nu_t}\le\de\left(\||v\|^2_{\phi^\nu_t}+\|\dbar_z v\|^2_{\phi^\nu_t}\right)
$$
for all $L$-valued $(n-1,0)$-forms $v$ on $X$, all $\nu$ and $t\in S_\nu$.  
\end{lem}

Combining these facts, we obtain the following key technical result. 
\begin{lem}\label{lem:vnut} For each $\nu$, there exists a unique smooth family $v^\nu=(v^\nu_t)_{t\in S_\nu}$ of $L$-valued $(n-1,0)$-forms such that 
\begin{itemize}
\item[(i)] $\partial_z^{\phi^\nu_t}v^\nu_t=P^\nu_t(\dot\phi^\nu_t u)$; 
\item[(ii)] $\om\wedge\dbar_z v^\nu_t=0$. 
\end{itemize}
The $L^2$-norm $\|v^\nu_t\|_{\phi^\nu_t}$ is bounded independently of $t$ and $\nu$. After perhaps passing to a subsequence, we can further find a sequence of smooth cut-off functions $0\le\chi_\nu\le 1$ on $X$ (with $\chi_\nu\equiv 0$ on some neighborhood of $\supp\D$) such that
\begin{itemize}
\item[(iii)] $\chi_\nu dd^c\phi^\nu_t\ge-\e_\nu\om$;
\item[(iv)] $\int_X(1-\chi_\nu) |v^\nu_t|^2_{\phi^\nu_t}\le\e_\nu\left(1+\|\dbar_z v^\nu_t\|^2_{\phi^\nu_t}\right)$,
\end{itemize}
for some sequence $\e_\nu>0$ converging to $0$. 
\end{lem}

\noindent {\bf Step 3: Subharmonicity of $L$}.
Our goal here is to show that $L(t)=-\log\| u\|^2_{\phi_t}$ is subharmonic on $S$. This function is the decreasing limit of $L^\nu(t):=-\log\|u\|^2_{\phi^\nu_t}$, which may be viewed as the weight of the $L^2$-metric induced by $\phi^\nu_t$ on the trivial line bundle $S_\nu\times H^0(X,K_X+L)$. By \cite[Theorem 3.1]{Bern11} (see also \cite[Lemma 14]{CDS3} for a direct computation), we thus have the curvature formula

\begin{equation}\label{equ:curv}
\|u\|^2_{\phi^\nu_t} dd^c_t L^\nu=\|\dbar_z v^\nu_t\|^2_{\phi^\nu_t} idt\wedge d\bar t
+\int_X\Theta_\nu, 
\end{equation}
where we have set
\begin{equation}\label{equ:Thetanu}
\Theta_\nu:=i^{n^2} dd^c\phi^\nu_t\wedge w_\nu\wedge\overline{w_\nu}
\end{equation}
with
\begin{equation}\label{equ:wnu}
w_\nu:=u-dt\wedge v^\nu_t,
\end{equation}
and $\int_X$ denotes fiber integration. 

First, we observe that the left-hand coefficient satisfies
$$
C^{-1}\le\|u\|^2_{\phi^\nu_t}\le C
$$
for some uniform constant $C>0$. This is a consequence of $e^{-\phi^\nu_t}\le e^{-\phi_t}=e^{-\tau_t-\phi_\D}$, since $e^{-\phi_\D}$ is integrable while $e^{-\tau_t}\le C e^{-\tau_{t_0}}$ for any fixed $t_0$ by Lipschitz continuity of $t\mapsto\tau_t$. 

Next, as in \cite[\S 6.2]{Bern11}, we note that
\begin{equation}\label{equ:intaway}
\int_X\chi_\nu\Theta_\nu\ge-C\e_\nu i dt\wedge d\bar t,
\end{equation}
thanks to the $L^2$-bound $\|v^\nu_t\|_{\phi^\nu_t}\le C$ and the curvature lower bound $\chi_\nu dd^c\phi^\nu_t\ge-\e_\nu\om$. On the other hand, the global curvature bound $dd^c\phi^\nu_t\ge-C\om$ combined with (iv) in Lemma \ref{lem:vnut} yields 
\begin{equation}\label{equ:inton}
\int_X (1-\chi_\nu) \Theta_\nu\ge-C\left(\int_X(1-\chi_\nu) |v_t^\nu|^2_{\phi^\nu_t}\right)idt\wedge d\bar t\ge-C\e_\nu\left(1+\|\dbar_z v^\nu_t\|^2_{\phi^\nu_t}\right)idt\wedge d\bar t.
\end{equation}
Injecting these estimates in the curvature formula (\ref{equ:curv}), we obtain
\begin{equation}\label{equ:ddcL}
dd^c_t L^\nu\ge\left(c\|\dbar_z v^\nu_t\|^2_{\phi^\nu_t}-\e_\nu\right) i dt\wedge d\bar t
\end{equation}
for some uniform constant $c>0$ and $\e_\nu\to 0$. In particular, we get as desired $dd^c_t L\ge 0$ in the limit, thereby proving that $L$ is subharmonic.\\  

\noindent {\bf Step 4: Holomorphy of $v$}.
From now on, we assume that $L$ is harmonic, so that $dd^c L^\nu\to 0$ weakly on $S$. As a first consequence, we obtain the following estimates: 

\begin{lem}\label{lem:harm} The following fiber integrals converge weakly to zero on $S$ as $\nu\to\infty$.
\begin{itemize}
\item[(i)] $\int_X|\dbar_z v^\nu_t|^2_{\phi^\nu_t}$; 
\item[(ii)] $\int_X\chi_\nu\Theta_\nu$; 
\item[(iii)] $\int_X(1-\chi_\nu)\Theta_\nu$; 
\item[(iv)] $\int_X(1-\chi_\nu)|v^\nu_t|^2_{\phi^\nu_t}$.
\end{itemize}
\end{lem}
\begin{proof} (i) follows directly from (\ref{equ:ddcL}). Another application of the curvature formula (\ref{equ:curv}) then yields
$\int_X\Theta_\nu\to 0$. Now let $f\in C^\infty_c(S)$ be a non-negative test function. Injecting (i) in (\ref{equ:inton}) yields
$\int_{X\times S}f(1-\chi_\nu) \Theta_\nu\ge-\e^\nu_K$, while (\ref{equ:intaway}) gives $\int_{X\times S}f\chi_\nu\Theta_\nu\ge-\e^\nu_K$. Since the sum converges to zero by what we just saw, we get (ii) and (iii). Finally, (iv) is a consequence of (i) and point (iv) of Lemma \ref{lem:vnut}. 
\end{proof}

By the uniform $L^2$-bound on $v^\nu_t$, the corresponding sequence $v^\nu$ on $X\times S$ is bounded in $L^2_\loc$ (with respect to a smooth reference metric on $L$). After passing to a subsequence, we may assume that $v^\nu$ converges weakly in $L^2_\loc(X\times S)$ to a section $v$. Our goal is to show that $v$ is in fact holomorphic on $X\times S$.

As a direct consequence of estimate (i) in Lemma \ref{lem:harm}, we have $\dbar_z v=0$ weakly. The hard part is to prove that $\partial v/\partial\bar t=0$ holds weakly. We first observe that it is enough to show

\begin{equation}\label{equ:holo}
\lim_\nu\int_{X\times S} idt\wedge d\bar t\wedge\frac{\partial v^\nu}{\partial\bar t}\wedge\overline{\a_t}\,e^{-\phi^\nu}=0
\end{equation}
for all compactly supported Lipschitz continuous families $\a_t$ of bounded $L$-valued $(n,1)$-forms on $X$. Indeed, choosing $\a_t$ supported in a local coordinate chart in which $L$ is trivialized and identifiying metrics on $L$ with functions, we can write
$$
v^\nu=\sum_{j=1}^n f^\nu_j(z,t)dz_1\wedge\dots\wedge\widehat{dz_j}\wedge\dots\wedge dz_n, 
$$
and it is then enough to choose $\a_t$ of the form 
$$
\a_t(z)=e^{\phi_t(z)}g(z,t)dz_1\wedge\dots\wedge dz_n\wedge d\bar z_j
$$
with $g\in C^\infty_c$. 

\medskip

Let $K\subset S$ be a compact set such that $\a_t=0$ for $t\notin K$. Using again $H^{n,1}(X,L)=0$, we get for each $t,\nu$ a unique $L$-valued $(n,0)$-form $\b^\nu_t$, orthogonal to the kernel of $\dbar_z$ and such that
$$
\a_t=P^\nu_t\a_t+\dbar_z \b^\nu_t
$$
By \cite[Lemma 4.2]{Bern11}, $t\mapsto \b^\nu_t$ is uniformly Lipschitz continuous as a map from $S$ to $L^2$, again with respect to any choice of a reference smooth metric on $L$. We will rely on the following identity. 

\begin{lem}\label{lem:ident}  For each $t,\nu$, we have 
$$
\int_{X\times K}idt\wedge d\bar t\wedge\frac{\partial v^\nu}{\partial\bar t}\wedge\overline{\a_t}\,e^{-\phi^\nu}=(-1)^n\int_{X\times K} dd^c\phi^\nu\wedge w_\nu\wedge\overline{\b^\nu_t}\,e^{-\phi^\nu}.
$$
\end{lem}
Recall that we have set $w_\nu=u-dt\wedge v^\nu$.
\begin{proof} By construction, $v^\nu$ satisfies $\om\wedge\dbar_z v^\nu=0$, and hence $\dbar_z\left(\frac{\partial v^\nu}{\partial\bar t}\right)\wedge\om=0$ as well. As noted in Remark \ref{rem:ortho}, it follows that $\int_X \frac{\partial v^\nu}{\partial\bar t}\wedge\overline{P^\nu_t\a_t}\,e^{-\phi^\nu_t}=0$, and hence 
\begin{equation}\label{equ:ident1}
\int_X\frac{\partial v^\nu}{\partial\bar t}\wedge\bar\a_t\,e^{-\phi^\nu_t}=\int_X\frac{\partial v^\nu}{\partial\bar t}\wedge\overline{\dbar_z \b^\nu_t}\,e^{-\phi^\nu_t}.
\end{equation}
Next, we claim that
\begin{equation}\label{equ:ident2}
\int_X\frac{\partial v^\nu}{\partial\bar t}\wedge\overline{\dbar_z \b^\nu_t}\,e^{-\phi^\nu_t}=(-1)^n\int_X\partial_z^{\phi^\nu_t}\left(\frac{\partial v^\nu}{\partial\bar t}\right)\wedge\overline{\b^\nu_t}\,e^{-\phi^\nu_t}.
\end{equation}
Indeed, (\ref{equ:scalint}) gives
$$
i^{(n-1)^2}\int_X\frac{\partial v^\nu}{\partial\bar t}\wedge\overline{\dbar_z \b^\nu_t}\,e^{-\phi^\nu_t}=\langle\frac{\partial v^\nu}{\partial\bar t}\wedge\om,\dbar_z \b^\nu_t\rangle_{L^2(\phi^\nu_t)}
$$
$$
=\langle\dbar^*_{\phi^\nu}\left(\frac{\partial v^\nu}{\partial\bar t}\wedge\om\right),\b^\nu_t\rangle_{L^2(\phi^\nu_t)}
=i\langle\partial_z^{\phi^\nu}\left(\frac{\partial v^\nu}{\partial\bar t}\right),\b^\nu_t\rangle_{L^2(\phi^\nu_t)}=i^{n^2+1}\int_X\partial_z^{\phi^\nu}\left(\frac{\partial v^\nu}{\partial\bar t}\right)\wedge\overline{\b^\nu_t}\,e^{-\phi^\nu_t},
$$
using the K\"ahler identity (\ref{equ:kahler}). The claim follows since $i^{n^2+1-(n-1)^2}=(-1)^n$. 

Now, a simple computation shows that
\begin{equation}\label{equ:partialtbar}
\partial_z^{\phi^\nu}\left(\frac{\partial v^\nu}{\partial\bar t}\right)=P^\nu_t u^\nu_t
\end{equation}
with 
$$
u^\nu_t:=\partial_z\dot\phi^\nu_t\wedge v^\nu+\ddot\phi^\nu_t u.
$$
To see this, recall that $\eta^\nu_t:=\partial_z^{\phi^\nu}v^\nu-\dot\phi^\nu_t u$ satisfies by construction $\dbar_z\eta^\nu_t=0$. Using the local description $\partial_z^{\phi^\nu_t}=\partial_z-\partial_z\phi^\nu_t\wedge\cdot$, we apply $\partial/\partial\bar t$ to get 
$$
\frac{\partial\eta^\nu_t}{\partial\bar t}=\partial_z^{\phi^\nu_t}\left(\frac{\partial v^\nu}{\partial\bar t}\right)-\partial_z\dot\phi^\nu_t\wedge v^\nu-\ddot\phi^\nu u.
$$
The desired identity follows since the left-hand side is in the kernel of $\dbar_z$ while 
$$
\partial_z^{\phi^\nu_t}\left(\frac{\partial v^\nu}{\partial\bar t}\right)=\dbar^*_{\phi^\nu_t}\left(\frac{\partial v^\nu}{\partial\bar t}\wedge\om\right)
$$ 
is orthogonal to the kernel of $\dbar_z$. Finally, writing
$$
dd^c\phi^\nu=\ddot\phi_t^\nu i dt\wedge d\bar t+i\partial_z\dot\phi_t^\nu\wedge d\bar t+i dt\wedge\dbar_z\dot\phi_t^\nu+dd^c_z\phi^\nu
$$
and using the fact that $\overline{\b^\nu_t}$ has type $(0,n)$ on $X$ shows that
$$
dd^c\phi^\nu\wedge w_\nu\wedge\overline{\b^\nu_t}= i dt\wedge d\bar t\wedge u^\nu_t\wedge\overline{\b^\nu_t}. 
$$
As a result, we get
$$
\int_{X\times K}dd^c\phi^\nu\wedge w_\nu\wedge\overline{\b^\nu_t}\,e^{-\phi^\nu}=\int_{X\times K} i dt\wedge d\bar t\wedge u^\nu_t\wedge\overline{\b^\nu_t}\,e^{-\phi^\nu}
$$
$$
=\int_{X\times K} i dt\wedge d\bar t\wedge P^\nu_tu^\nu_t\wedge\overline{\b^\nu_t}\,e^{-\phi^\nu}=\int_{X\times K} i dt\wedge d\bar t\wedge\partial_z^{\phi^\nu}\left(\frac{\partial v^\nu}{\partial\bar t}\right)\wedge\overline{\b^\nu_t}\,e^{-\phi^\nu},
$$
where the second equality uses that $\b^\nu_t$ is orthogonal to the kernel of $\dbar_z$ and the third one comes from (\ref{equ:partialtbar}). Lemma \ref{lem:ident} now follows in view of (\ref{equ:ident1}) and (\ref{equ:ident2}). 
\end{proof}

Thanks to the previous lemma, the desired estimate (\ref{equ:holo}) boils down to the following. 

\begin{lem}\label{lem:W} For each non-negative $f\in C^\infty_c(S)$, $\int_{X\times S} f\,dd^c\phi^\nu\wedge w_\nu\wedge\overline{\b^\nu_t}\, e^{-\phi^\nu}$ tends to $0$.  
\end{lem}

The proof will rely on the following special case of the Bochner-Kodaira-Nakano identity, referred to as the 'one-variable H\"ormander inequality' in \cite[\S 4]{Bern11}. 

\begin{lem}\label{lem:BKN} Let $S$ be a Riemann surface and $\f$ (resp. $u$) be a smooth real valued (resp. complex valued, compactly supported) function on $S$. Then
$$
\int_S |u|^2 e^{-\f} dd^c\f\le i\int_S(\partial u-u\partial\f)\wedge\overline{(\partial u-u\partial\f)}\,e^{-\f}. 
$$
\end{lem}

\begin{proof} Pick any K\"ahler form $\om$ on $S$, and view $\f$ (resp. $u$) as a metric (resp. section) of the trivial line bundle on $M$, so that $\partial^\f u=\partial u-u\partial\f$. For bidegree reasons, the Bochner-Kodaira-Nakano identity (cf.~for instance \cite[\S 13.2]{Demhodge}) gives
$$
\|\dbar u\|_\f^2=\|\partial^\f u\|_\f^2+\langle[dd^c\phi,\Lambda]u,u\rangle_{L^2(\f)}
$$
with $\Lambda$ the pointwise adjoint of $\om\wedge\bullet$. For bidegree reasons again (compare (\ref{equ:scalint})), we have  
$$
\langle[dd^c\f,\Lambda]s,s\rangle_{L^2(\f)}=-\langle  s dd^c\f,s \om\rangle_{L^2(\f)}
$$
$$
=-\int_S|s|^2 e^{-\f} dd^c\f. 
$$
\end{proof}

\begin{proof}[Proof of Lemma \ref{lem:W}] Because of the large negative part of the curvature $dd^c\phi^\nu$ near $\supp\D$, we cut the integral $\int_{X\times S} f\,dd^c\phi^\nu\wedge w_\nu\wedge\overline{\b^\nu_t}\, e^{-\phi^\nu}$ into two pieces using $\chi_\nu$. Note that we may and do assume that $f\in C^\infty_c(S)$ has been chosen so that $f^{1/2}$ is smooth. 

First, the curvature bound $\chi_\nu(z) dd^c\phi^\nu+\e_\nu\om\ge 0$ and the Cauchy-Schwarz inequality yield
$$
\left|\int_{X\times S} f(t)\left(\chi_\nu\,dd^c\phi^\nu+\e_\nu\om\right)\wedge  w_\nu\wedge\overline{\b^\nu_t}\,e^{-\phi^\nu}\right|
$$
$$
\le\left(i^{n^2}\int_{X\times S}f(t)\left(\chi_\nu\,dd^c\phi^\nu+\e_\nu\om\right)\wedge w_\nu\wedge\bar w_\nu\,e^{-\phi^\nu}\right)^{1/2}\times
$$
$$
\left(i^{n^2}\int_{X\times S}f(t)\left(\chi_\nu\,dd^c\phi^\nu+\e_\nu\om\right)\wedge \b^\nu_t\wedge\overline{\b^\nu_t}\,e^{-\phi^\nu}\right)^{1/2}.
$$
The first right-hand factor
$$
i^{n^2}\int_{X\times S}f(t)\left(\chi_\nu\,dd^c\phi^\nu+\e_\nu\om\right)\wedge w_\nu\wedge\bar w_\nu\,e^{-\phi^\nu}
$$
$$
=\int_{X\times S}f(t)\chi_\nu\Theta_\nu+\e_\nu i^{n^2}\int_{X\times S}f(t)\chi_\nu\om\wedge w_\nu\wedge\bar w_\nu\,e^{-\phi^\nu}
$$
tends to $0$ thanks to Lemma \ref{lem:harm} and the $L^2$-bound on $v^\nu_t$. To show that the second factor is bounded, note that $dd^c\phi^\nu\wedge\b^\nu_t\wedge\overline{\b^\nu_t}=dd^c_t\phi^\nu\wedge\b^\nu_t\wedge\overline{\b^\nu_t}$, since $\b^\nu_t$ is an $(n,0)$-form on $X$. Thanks to the Lipschitz bounds for $t\mapsto \b^\nu_t\in L^2$, $t\mapsto\phi^\nu_t$ and $f(t)^{1/2}$, the one-variable H\"ormander inequality of Lemma \ref{lem:BKN} yields a uniform bound upper bound for
$$
\int_{X\times S} f(t)\chi_\nu\,dd^c\phi^\nu\wedge\b^\nu_t\wedge\overline{\b^\nu_t}e^{-\phi^\nu_t}, 
$$
which shows that
$$
i^{n^2}\int_{X\times S}f(t)\left(\chi_\nu\,dd^c\phi^\nu+\e_\nu\om\right)\wedge \b^\nu_t\wedge\overline{\b^\nu_t}\,e^{-\phi^\nu}
$$
is indeed bounded. We infer from the above that 
$$
\int_{X\times S}f\chi_\nu\,dd^c\phi^\nu\wedge w_\nu\wedge\overline{\b^\nu_t}\, e^{-\phi^\nu}\to 0.
$$
Using now the global curvature bound $dd^c\phi^\nu+C\om\ge 0$, we similarly write
$$
\left|\int_{X\times S}f\,(1-\chi_\nu)(dd^c\phi^\nu+C\om)\wedge  w_\nu\wedge\overline{\b^\nu_t}\,e^{-\phi^\nu}\right|
$$
$$
\le\left(\int_{X\times S} f\,(1-\chi_\nu)(dd^c\phi^\nu+C\om)\wedge w_\nu\wedge\bar w_\nu\,e^{-\phi^\nu}\right)^{1/2}\times
$$
$$
\left(\int_{X\times S} f\,(1-\chi_\nu)(dd^c\phi^\nu+C\om)\wedge \b^\nu_t\wedge\overline{\b^\nu_t}\,e^{-\phi^\nu}\right)^{1/2}.
$$
The first factor 
$$
\int_{X\times S} f\,(1-\chi_\nu)(dd^c\phi^\nu+C\om)\wedge w_\nu\wedge\bar w_\nu\,e^{-\phi^\nu}\le\int_{X\times K}(1-\chi_\nu)\Theta_\nu+C'\int_{X\times K}(1-\chi_\nu)|v^\nu_t|^2_{\phi^\nu}
$$
tends to zero by Lemma \ref{lem:harm}, and the second factor is bounded for the same reason as above, thanks to the one-variable H\"ormander inequality, and hence
$$
\int_{X\times S}f\,(1-\chi_\nu)(dd^c\phi^\nu+C\om)\wedge  w_\nu\wedge\overline{\b^\nu_t}\,e^{-\phi^\nu}\to 0. 
$$ 
Since
$$
\int_{X\times S}f\,(1-\chi_\nu)\om\wedge w_\nu\wedge\bar w_\nu\,e^{-\phi^\nu}\le C\int_{X\times S}f\,(1-\chi_\nu)|v^\nu_t|^2_{\phi^\nu}
$$
tends to $0$, we conclude as desired that
$$
\int_{X\times S}f dd^c\phi^\nu\wedge  w_\nu\wedge\overline{\b^\nu_t}\,e^{-\phi^\nu}=
$$
$$
\int_{X\times S}f\chi_\nu dd^c\phi^\nu\wedge  w_\nu\wedge\overline{\b^\nu_t}\,e^{-\phi^\nu}+\int_{X\times S}f\,(1-\chi_\nu)dd^c\phi^\nu\wedge  w_\nu\wedge\overline{\b^\nu_t}\,e^{-\phi^\nu}
$$
tends to $0$. 
\end{proof}

\noindent {\bf Step 5: End of the proof}. Recall that $v$, which is now known to be holomorphic on $X\times S$, is obtained as the weak $L^2_\loc$ limit on $X\times S$ of $v^\nu$, and that $u\in H^0(X,K_X+L)$ is the given non-zero holomorphic $L$-valued $n$-form. 

\begin{lem}\label{lem:local} The distributional equation $\partial_z\dbar_z\phi\wedge v=\partial_z\dot\phi_t\wedge u$
is satisfied on $X\times S$. 
\end{lem}
\begin{proof} Set $h^\nu:=P^\nu_t(\dot\phi^\nu_t u)-\dot\phi^\nu_t u$, which satisfies $\dbar_z h^\nu=0$. These functions are uniformly bounded in $L^2_{\mathrm{loc}}(X\times S)$, since $\int_X|h^\nu_t|^2_{\phi^\nu_t}$ is uniformly bounded thanks to the uniform Lipschitz bound for $t\mapsto\phi^\nu_t$. We may thus assume that $h^\nu\to h$ weakly in $L^2_{\mathrm{loc}}(X\times S)$. 

Since $\dbar_z h=0$, the desired result will follow from the identity
\begin{equation}\label{equ:dephi}
\partial_z v-\partial_z\phi\wedge v=\dot\phi_t u+h, 
\end{equation}
understood locally on $X\times S$. Recall that all (pluri)subharmonic functions belong to the Sobolev space $W^{1,1}_\loc$, basically because the Newton kernel has the same property. In particular, (\ref{equ:dephi}) is an equality in $L^1_\loc(X\times S)$, and it will thus be enough to argue on the open set $U:=(X\setminus\supp\D)\times S$ where the psh function $\phi$ is locally bounded. 

Rewrite $\partial_z^{\phi^\nu}v^\nu=P^\nu_t(\dot\phi^\nu_t u)$ as 
$$
\partial_z(e^{-\phi^\nu} v^\nu)=(\dot\phi^\nu_t u+h^\nu)e^{-\phi^\nu}=u\frac{\partial}{\partial t}(e^{-\phi^\nu})+h^\nu\,e^{-\phi^\nu}. 
$$
On $U$, we have $e^{-\phi^\nu}\to e^{-\phi}$ strongly in $L^2_\loc$, and $v^\nu\to v$ and $h^\nu\to h$ weakly in $L^2_\loc$. This is enough to get
\begin{equation}\label{equ:parexp}
\partial_z(e^{-\phi} v)=u\frac{\partial}{\partial t}(e^{-\phi})+h\,e^{-\phi}
\end{equation}
on $U$. Since the psh function $\phi$ is locally bounded on $U$, it satisfies the chain rule 
$$
d(e^{-\phi})=e^{-\phi}d\phi, 
$$ 
see for instance \cite[Lemma 1.9]{BEGZ}, and (\ref{equ:dephi}) thus follows from (\ref{equ:parexp}). 
\end{proof}
By Lemma \ref{lem:vnut}, $v^\nu_t$ is uniquely determined by an equation whose only dependence on $t$ is through $\phi^\nu_t$. As a result, $v^\nu_t$ is independent of $\Im t$, and hence so is $v_t$. Being holomorphic in $t$, the latter is thus independent of $t$. 

On the open set $\{u\ne 0\}$, define a holomorphic vector field $V$ by requiring that $i_V u=-v$. Since $\theta_t:=dd^c_z\phi_t$ satisfies $\theta_t\wedge u=0$ for bidegree reasons, we have 
$$
(i_V\theta_t)\wedge u=\theta_t\wedge (i_V u)=-\theta_t\wedge v,
$$
and Lemma \ref{lem:local} thus gives
$$
(i_V\theta_t+i\dbar_z\dot\phi_t)\wedge u=0. 
$$
For bidegree reasons and since we are working over the locus where $u$ does not vanish, it follows that $i_V\theta_t+i\dbar_z\dot\phi_t=0$. Using the Cartan identity $\cL_V=d i_V+i_V d$ for the Lie derivative, we obtain the desired equation
$$
\left(\mathcal L_V+\frac{\partial}{\partial t}\right)\theta_t=0,
$$ 
thereby concluding the proof of Theorem \ref{thm:berndt}.

\end{document}